\documentclass[12pt]{amsart}
\usepackage[margin=1in]{geometry}

\usepackage{fullpage}
\usepackage{amssymb,amsfonts}
\usepackage[all,arc]{xy}
\usepackage{enumerate}
\usepackage{enumitem}
\usepackage{mathrsfs}
\usepackage{mathtools}
\usepackage{tikz}
\usepackage{tikz-cd}
\usepackage{graphicx}
\usepackage[font={small}]{caption}
\usepackage{float}
\usepackage{subcaption}
\usepackage[colorlinks=true]{hyperref}
\usepackage{rotating}
\usepackage{changes}
\usepackage{standalone}
\usepackage{comment}
\usepackage{relsize}
\usetikzlibrary{patterns}
\usepackage{accents}
\usepackage{tabularx}
\usepackage{bm}

\usepackage{longtable}
\usepackage{lscape}

\newcommand\SmallMatrix[1]{{%
  \scriptsize\arraycolsep=0.6\arraycolsep\ensuremath{\begin{pmatrix}#1\end{pmatrix}}}}
  
\DeclareMathOperator{\Spinc}{Spin^c}
\DeclareMathOperator{\spinc}{spin^c}
\DeclareMathOperator{\Char}{Char}

\newcommand{\agraph}{\Gamma}
\newcommand{\agraphp}{\Gamma'}

\newcommand{\amfld}{Y}

\newcommand{\adeg}{\delta_{amb}}

\newcommand{\amatrix}{M}
\newcommand{\amatrixp}{M'}

\newcommand{\mmatrix}{M_{v_0}}

\newcommand{\mgraph}{\Gamma_{v_{0}}}
\newcommand{\mgraphp}{\Gamma_{v_{0}}'}

\newcommand{\mmfld}{Y_{v_{0}}}

\newcommand{\sgraph}{\Gamma_{v_{0},m_{0}}}

\newcommand{\smfld}{Y_{v_{0},m_{0}}}

\newcommand{\sdeg}{\delta}

\newcommand{\smatrix}{M_{v_{0},m_{0}}}

\newcommand{\reldeg}{\h{\delta}}

\newcommand{\V}{\mathcal{V}}

\DeclareMathOperator{\Edges}{\mathcal{E}}

\newcommand{\knot}{\mathcal{K}}

\renewcommand{\root}{R}
\newcommand{\rootbi}{R^{\rm bi}}
\newcommand{\superlevel}{\mathcal{S}}
\newcommand{\superlevelone}{\b{\mathcal{S}}}

\newcommand{\one}{u}
\newcommand{\onev}{\lambda}
\newcommand{\onevp}{\onev'}
\newcommand{\coset}{\mathcal{A}}
\newcommand{\Zhat}{\widehat{Z}}
\newcommand{\Zhathat}{\widehat{\vphantom{\rule{5pt}{10pt}}\smash{\widehat{Z}}\,}\!}

\newcommand{\qexp}{\xi}
\newcommand{\zexp}{\zeta}
\newcommand{\texp}{\theta}

\newcommand\hdual[1]{{#1}^{\intercal}}

\newcommand{\Ring}{\mathcal{R}}

\renewcommand{\sf}{\operatorname{sf}}

\newcommand{\relmap}{\omega} 

\newcommand{\linkn}{\ell k}

\newcommand{\shift}{\sigma}

\newcommand{\eps}{\varepsilon}

\DeclareMathOperator{\F}{\mathbb{F}}

\DeclareMathOperator{\Z}{\mathbb{Z}}
\DeclareMathOperator{\Q}{\mathbb{Q}}

\DeclareMathOperator{\Hom}{Hom}

\newcommand{\til}[1]{\widetilde{#1}}

\newcommand{\h}[1]{\widehat{#1}}

\renewcommand{\b}[1]{\overline{#1}}

\newcommand{\BI}{\Ring[[z,z^{-1}]]}

\hypersetup{
anchorcolor =.
}
\makeatletter
 \newcommand{\htarget}[1]{\Hy@raisedlink{\hypertarget{#1}{\label{#1}}}}
\makeatother

\newtheorem{thm}{Theorem}[section]
\newtheorem*{thm*}{Theorem}

\newtheorem{prop}[thm]{Proposition}
\newtheorem{lem}[thm]{Lemma}

\theoremstyle{definition}
\newtheorem{defn}[thm]{Definition}

\newtheorem{exmp}[thm]{Example}

\theoremstyle{remark}
\newtheorem{rem}[thm]{Remark}

\makeatletter
 \newcommand{\linkdest}[1]{\Hy@raisedlink{\hypertarget{#1}{}}}
\makeatother

\title{Knot lattice homology and $q$-series invariants for plumbed knot complements}

\author[R. Akhmechet]{Rostislav Akhmechet}
\address{Department of Mathematics, Columbia University, New York, NY 10027}
\email{\href{mailto:akhmechet@math.columbia.edu}{akhmechet@math.columbia.edu}}

\author[P. Johnson]{Peter K. Johnson}
\address{Department of Mathematics, Michigan State University, East Lansing, MI 48824}
\email{\href{mailto:john8251@msu.edu}{john8251@msu.edu}}

\author[S. Park]{Sunghyuk Park}
\address{Department of  Mathematics, Harvard University, Cambridge, MA 02138}
\email{\href{mailto:sunghyukpark@math.harvard.edu}{sunghyukpark@math.harvard.edu}}


\begin{document}

\maketitle

\begin{abstract}
We introduce an invariant of negative definite plumbed knot complements unifying knot lattice homology, due to Ozsv{\'a}th, Stipsicz, and Szab{\'o}, and the BPS $q$-series of Gukov and Manolescu. 
This invariant is a natural extension of weighted graded roots of negative definite plumbed 3-manifolds introduced earlier by the first two authors and Krushkal. 
We prove a surgery formula relating our invariant with the weighted graded root of the surgered 3-manifold. 
\end{abstract}

\tableofcontents 

\section{Introduction}
In this paper, we introduce an invariant of negative definite plumbed knot complements that unifies the homological degree zero part of knot lattice homology \cite{OSS} and the BPS $q$-series \cite{GM}. 

\emph{Lattice homology} $\mathbb{H}_*$, defined by N\'{e}methi \cite{Nem_On_the, Nem_Lattice_cohomology}, is an invariant of negative definite plumbed 3-manifolds which has important applications to both low-dimensional topology and singularity theory.
It expands upon earlier work of Ozsv\'{a}th and Szab\'{o} \cite{OS-on_the_Floer} in which they study the Heegaard Floer homology of a certain class of negative definite plumbed 3-manifolds. 

Given a closed oriented 3-manifold $Y$, described as a negative definite plumbing, $\mathbb{H}_*(Y)$ is a module over the polynomial ring $\Z[U]$. 
It decomposes as a direct sum over $\spinc$ structures of $Y$, $\mathbb{H}_{*}(Y)=\bigoplus_{\mathfrak{s}\in\spinc(Y)}\mathbb{H}_{*}(Y,\mathfrak{s})$.
Moreover, for each $\mathfrak{s}\in \spinc(Y)$, $\mathbb{H}_{*}(Y, \mathfrak{s})$ carries two gradings: the \emph{homological} grading and the \emph{Maslov} grading.
The homological grading is given by the index $*\in \Z_{\geq 0}$. 
For each fixed homological grading $i\in \Z_{\geq 0}$, $\mathbb{H}_{i}(Y,\mathfrak{s})$ is itself a Maslov graded $\Z[U]$-module, where $U$ is in Maslov degree $-2$.

For each $\mathfrak{s}\in \spinc(Y)$, the homological degree zero part of lattice homology $\mathbb{H}_{0}(Y,\mathfrak{s})$ can be conveniently encoded by an infinite graph called the \emph{graded root}. 
Note $\mathbb{H}_{0}$ was already present in \cite{OS-on_the_Floer}, before the general formulation of lattice homology. 
For a subclass of negative definite plumbings called \emph{almost rational plumbings}, lattice homology is concentrated in homological degree zero \cite{Nem_Lattice_cohomology}. 
Using a completed version of lattice homology obtained by working over the ground ring $\F[[U]]$ where $\F = \Z/2\Z$, 
Zemke \cite{Zem} established the equivalence of lattice homology (using all homological gradings) and Heegaard Floer homology $HF^-$ defined by Ozsv{\'a}th-Szab{\'o} \cite{OS-HF} for plumbing trees (not necessarily negative definite), extending earlier proofs of this equivalence in special cases \cite{OS-on_the_Floer,Nem_On_the, OSS, OSS_spectral_sequence}.

The \emph{BPS $q$-series} $\Zhat$ \cite{GPPV}, also known as the \emph{homological block} or the \emph{Gukov-Pei-Putrov-Vafa (GPPV) invariant}, is another invariant of negative definite plumbed 3-manifolds. 
Like lattice homology, BPS $q$-series are indexed by the set of $\spinc$ structures of the $3$-manifold.
As the name suggests, this invariant takes the form of a power series in $q$ with integer coefficients (up to a simple overall factor). 
These $q$-series encode the Witten-Reshetikhin-Turaev (WRT) invariants  \cite{Witten,RT} in the sense that WRT invariants can be recovered in the radial limit to roots of unity of certain linear combinations of these $q$-series over spin$^c$ structures \cite{GPPV, Mur}. 
For some classes of negative definite plumbed $3$-manifolds, the BPS $q$-series are known to satisfy (quantum) modularity  \cite{Lawrence-Zagier, Zagier, 3d_modularity, BMM}. 

While lattice homology and BPS $q$-series have very different origins, they are both defined in terms of the lattice of characteristic vectors of the 4-manifold coming from the plumbing description bounded by the plumbed 3-manifold. 
Based on this observation, the first two authors and Krushkal \cite{AJK} assigned to each node in the graded root a Laurent polynomial weight in two variables $q$ and $t$, resulting in the \emph{weighted graded root}, which unifies the graded root and the BPS $q$-series. 
These weights depend on a choice of \emph{admissible family of functions} (Definition \ref{def:admissible family}). 
In an appropriate sense (see \cite[Section 6]{AJK}), the weights stabilize to a two-variable power series in $q$ whose coefficients are Laurent polynomials in $t$. 
For a specific choice of admissible family $\h{W}$, evaluating the resulting power series at $t=1$ yields exactly the BPS $q$-series. 
Recent work of Liles and McSpirit \cite{Liles-McSpirit} studied these two variable refinements of $\Zhat$ and established quantum modularity for other specializations of $t$. 

Both lattice homology and BPS $q$-series have natural extensions to plumbed knot complements, namely \emph{knot lattice homology}, introduced by Ozsv{\'a}th-Stipsicz-Szab{\'o} \cite{OSS}, and \emph{BPS $q$-series for knot complements}, introduced by Gukov-Manolescu \cite{GM}, respectively. 
It has been shown that knot lattice homology is isomorphic to knot Floer homology for certain classes of knots \cite{OSS_L_spaces}. 
Niemi-Colvin \cite{Niemi-Colvin} reformulated knot lattice homology as the singular homology of a double filtration of a Euclidean space and proved that the homotopy type of this double filtration is an invariant of the plumbed knot complement. 
It is implicit in \cite{Niemi-Colvin} that the homological degree zero part of knot lattice homology can be naturally encoded by a certain infinite graph that we refer to as the \emph{bigraded root}. 
The two gradings of the bigraded root reflect the graded $\F[U,V]$-module structure of the knot Floer homology. 

A natural question that arises from the construction of \cite{AJK} is whether the weighted graded root can be extended to plumbed knot complements. 
Our first main result gives a positive answer to this question. 
As an executive summary, a negative definite marked plumbing graph describes a closed plumbed $3$-manifold $Y$ and a knot $\knot \subset Y$. 
The knot complement $Y\setminus \knot$ is equipped with a specified curve $\mu_{\knot}$ on its boundary given by the meridian of $\knot$. 
We refer to the pair $(Y\setminus \knot, \mu_{\knot})$ as a \emph{negative definite plumbed knot complement}, and we will often omit the boundary curve $\mu_{\knot}$ from the notation. 
If two negative definite marked plumbing graphs represent  plumbed knot complements for which there is an orientation-preserving diffeomorphism sending one boundary curve to the other, then the graphs are related by a finite sequence of \emph{Neumann moves} (Figures \ref{fig:Neumann moves closed} and \ref{fig:Neumann moves}). 

For each $\spinc$ structure on $Y$, we assign three-variable weights to each node of the bigraded root of $Y\setminus \knot$, resulting in the \emph{weighted bigraded root} for the plumbed knot complement; see Definition \ref{def:weighted bigraded root}. 
The weights are constant along the $V$-grading direction; see Figure \ref{fig:trefoil weighted bigraded root} for an example. 
Moreover, as we lower the $U$-grading, the weights stabilize to the BPS $q$-series for the plumbed knot complement. 
We note that knot lattice homology is indexed by the set of $\spinc$ structures of the $3$-manifold $Y$ containing $\knot$, while the BPS $q$-series for the knot complement depends on a choice of a \emph{relative} $\spinc$ structure on $Y\setminus \knot$. 
In Section \ref{sec:BPS q series for plumbed knot complements} we renormalize the $q$-series so that it depends only on the $\spinc$ structure on $Y$. 
Consequently, our weighted bigraded roots are indexed by $\spinc(Y)$.
We prove the following. 

\begin{figure}
    \centering
    \begin{tikzpicture}[scale=0.72]
\node at (-1,-1) {$\bullet$};
\node at (1,-1) {$\bullet$};
\node at (-2,-2) {$\bullet$};
\node at (0,-2) {$\bullet$};
\node at (2,-2) {$\bullet$};
\node at (-3,-3) {$\bullet$};
\node at (-1,-3) {$\bullet$};
\node at (1,-3) {$\bullet$};
\node at (3,-3) {$\bullet$};
\node at (-4,-4) {$\bullet$};
\node at (-2,-4) {$\bullet$};
\node at (0,-4) {$\bullet$};
\node at (2,-4) {$\bullet$};
\node at (4,-4) {$\bullet$};

\node[right] at (1,-1) {$\mathsmaller{w_{01}}$}; 
\node[right] at (2,-2) {$\mathsmaller{w_{02}}$}; 
\node[right] at (3,-3) {$\mathsmaller{w_{03}}$}; 
\node[right] at (4,-4) {$\mathsmaller{w_{04}}$};
\node[right] at (-1,-1) {$\mathsmaller{w_{10}}$}; 
\node[right] at (0,-2) {$\mathsmaller{w_{11}}$}; 
\node[right] at (1,-3) {$\mathsmaller{w_{12}}$}; 
\node[right] at (2,-4) {$\mathsmaller{w_{13}}$};
\node[right] at (-2,-2) {$\mathsmaller{w_{20}}$}; 
\node[right] at (-1,-3) {$\mathsmaller{w_{21}}$}; 
\node[right] at (0,-4) {$\mathsmaller{w_{22}}$};
\node[right] at (-3,-3) {$\mathsmaller{w_{30}}$}; 
\node[right] at (-2,-4) {$\mathsmaller{w_{31}}$};
\node[right] at (-4,-4) {$\mathsmaller{w_{40}}$}; 

\draw (-2,-2)--(-1,-1)--(0,-2)--(1,-1)--(2,-2);
\draw (-3,-3)--(-2,-2)--(-1,-3)--(0,-2)--(1,-3)--(2,-2)--(3,-3);
\draw (-4,-4)--(-3,-3)--(-2,-4)--(-1,-3)--(0,-4)--(1,-3)--(2,-4)--(3,-3)--(4,-4);


\node at (-5, -5) {\reflectbox{$\vdots$}};
\node at (5, -5) {$\vdots$};
\node at (0, -5) {$\vdots$};

\node at (1, 1) {$0$};
\node at (2, 0) {$-2$};
\node at (3, -1) {$-4$};
\node at (4, -2) {$-6$};
\node at (5, -3) {$-8$};

\node at (-1, 1) {$0$};
\node at (-2, 0) {$-2$};
\node at (-3, -1) {$-4$};
\node at (-4, -2) {$-6$};
\node at (-5, -3) {$-8$};

\draw (-1,2)--(5.5,-4.5);
\draw (1,2)--(-5.5,-4.5);
\draw (0,3)--(0,1);

\node at (-0.4,2.1) {$h_U$};
\node at (0.4,2.1) {$h_V$};
\begin{scope}[shift={(5,.5)}]
\node at (1,0) [align = left, anchor=west] {$\mathsmaller{w_{0i} = \frac{1}{2}(-q z + q^2 z^5 + q^3 z^7)}$};
\node at (1,-1.5) [align = left, anchor=west] {$\mathsmaller{w_{1i} = \frac{1}{2}(q z^{-1} -q z + q^2 z^5 + q^3 z^7 -q^6 z^{11} -q^8 z^{13})}$};
\node at (1,-3) [align = left, anchor=west] {$\mathsmaller{w_{2i} = \frac{1}{2}(-q^2 z^{-5} + q z^{-1} -q z + q^2 z^5 + q^3 z^7 -q^6 z^{11} -q^8 z^{13})}$};
\node at (1,-4.5) [align = left, anchor=west] {$\mathsmaller{w_{3i} = \frac{1}{2}(-q^3 z^{-7} -q^2 z^{-5} + q z^{-1} -q z + q^2 z^5 + q^3 z^7 -q^6 z^{11} -q^8 z^{13} + q^{13}z^{17})}$};
\node at (2.1, -5.5) {$\vdots$};
\end{scope}
\end{tikzpicture}
    \caption{The weighted bigraded root of the trefoil at $t=1$ corresponding to the admissible family $\h{W}$ and $\eps=1$.}
    \label{fig:trefoil weighted bigraded root}
\end{figure}
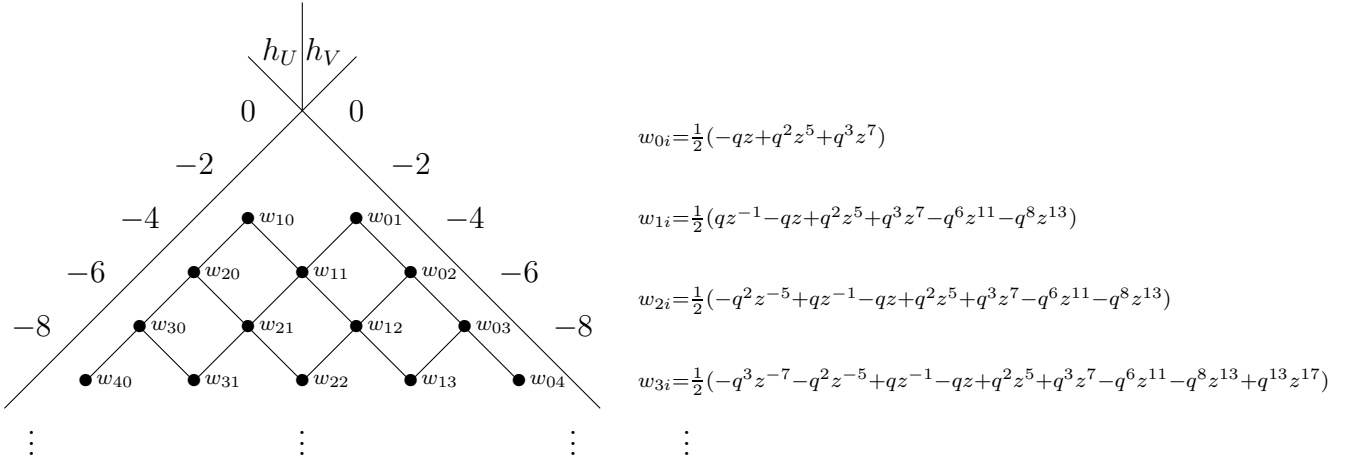

\begin{thm}[Invariance under Neumann moves; detailed version in Theorem \ref{thm:invariance}]
For each $\spinc$ structure, the weighted bigraded root for negative definite plumbed knot complements is invariant under Neumann moves. 
\end{thm}

As in \cite{AJK}, the weighted graded root depends on a choice of admissible family of functions. 
In the present paper the weights also depend on a choice of $\eps \in \{\pm 1\}$. 
These two choices, discussed in Section \ref{sec:Zhat closed}, correspond to two natural ways to identify the lattices used to define BPS $q$-series and (knot) lattice homology. 
Moreover, the two choices of $\eps$ clarify the behavior of the weighted graded root under $\spinc$ conjugation.   

Surgery formulas for knot lattice homology and for BPS $q$-series for knot complements were established in \cite{OSS}  and \cite{GM}, respectively. 
Our next main result unifies the two surgery formulas by relating the weighted (bi)graded root of the plumbed knot complement $Y\setminus \knot$ with that of the plumbed $3$-manifold obtained from  surgery on $\knot$.
\begin{thm}[Surgery formula; detailed version in Theorem \ref{thm:p surgery}] 
Let $Y\setminus \knot$ be a plumbed knot complement obtained from a negative definite marked plumbing graph. 
Let $Y'$ be a closed plumbed $3$-manifold built from a negative definite plumbing graph obtained by attaching an integer framing to the marked vertex. 
Then the weighted bigraded roots of $Y\setminus \knot$ determine the  weighted graded roots of $Y'$.
\end{thm}

\subsection*{Organization of this paper}
In Section \ref{sec:plumbed_mflds}, we review the basics of plumbing graphs and plumbed $3$-manifolds. 

In Section \ref{sec:AJK}, we review graded roots, the BPS $q$-series $\Zhat$, and weighted graded roots for closed plumbed 3-manifolds introduced in \cite{AJK}. 

In Section \ref{sec:weighted_bigraded_roots}, we review knot lattice homology, following the approach of Niemi-Colvin \cite{Niemi-Colvin}, and BPS $q$-series for plumbed knot complements. We then give our main construction, the weighted bigraded root for plumbed knot complements, and prove its invariance under Neumann moves. 

In Section \ref{sec:surgery_formula}, we prove the surgery formula for the weighted graded roots. 
We also illustrate it with an explicit example. 

In Appendix \ref{sec:remarks on invariants of plumbed mflds}, we discuss some subtleties regarding invariants of plumbed manifolds equipped with a $\spinc$ structure.

\subsection*{Summary of notation}\label{sec:notation}
We summarize some notations that will be used in this paper. 
\bgroup
\def\arraystretch{1.7}
\begin{longtable}[H]{l|p{13cm}|l}
    $K^2$ & The square of $K\in H^2(X(\agraph);\Z)$, given by $K^2 = \hdual{K} M^{-1} K$ & Eq. \eqref{eq:square} \\
    $\one$ &  $\one = (1,\ldots, 1)$. The length of this vector is determined by context & Eq \eqref{eq:\one} \\
    $\onev$ & $\onev= (\lambda_{1}, \ldots, \lambda_{s})$ where $\lambda_{i}= 1$ if $v_{i}$ is adjacent to $v_{0}$ in $\mgraph$  and $0$ otherwise & Eq. \eqref{eq:onev vector} \\
    $\reldeg$ & $\reldeg = \sdeg + e_0$ & Def. \ref{def:relative spinc}  \\ 
    $\Sigma$ & $\Sigma = \dfrac{\smatrix^{-1}e_0}{\hdual{e_0}\smatrix^{-1}e_0} = (1, - M^{-1}\onev) \in H_2(X(\Gamma_{v_0,m_0}); \mathbb{Q}) \cong \Q^{s+1}$  & Eq. \eqref{eq:Sigma}\\
    $\Sigma^2$ & $\Sigma^2 = \hdual{\Sigma} M_{v_0, m_0}\Sigma = \dfrac{1}{\hdual{e_0}\smatrix^{-1}e_{0}} = m_0 - \hdual{\onev}\amatrix^{-1} \onev \in \Q$  & Eq. \eqref{eq:Sigma^2} \\
    $\sf$ & The (rational) Seifert framing of $\knot$, given by $\hdual{\onev} M^{-1} \onev = m_0 -\Sigma^2$, which is an integer if $\knot$ is null-homologous.  &  Pg. \hyperlink{sf}{ \pageref{sf}}
\end{longtable}
\egroup

\subsection*{Acknowledgements}
We thank Antonio Alfieri, Sergei Gukov, Matthew Hedden, Slava Krushkal, Andr\'{a}s N\'{e}methi, Seppo Niemi-Colvin, and Matthew Stoffregen for interesting discussions. 

P.J and R.A. were partially supported by NSF RTG Grant DMS-1839968 while working on this project. 
S.P. gratefully acknowledges support from Simons Foundation through Simons Collaboration on Global Categorical Symmetries.

\section{Plumbed  manifolds}\label{sec:plumbed_mflds}

\subsection{Closed plumbed $3$-manifolds}
\label{subsec:closed plumbed manifolds}
In this subsection, we review closed plumbed 3-manifolds, their $\spinc$ structures, and Neumann moves. 

Given a graph $\Gamma$, we denote its set of vertices by $\V(\Gamma)$. 
We say $\Gamma$ is \emph{integer weighted} if it is equipped with a function $m: \V(\Gamma)\to \Z$. 
In this paper, a \textit{plumbing graph} will mean an integer weighted forest $\Gamma$ with finitely many vertices.

To a plumbing graph $\Gamma$, one can associate a 4-manifold $X = X(\Gamma)$ and a 3-manifold $Y = Y(\Gamma)$ as follows. 
First, form a framed link $\mathcal{L}=\mathcal{L}(\agraph)\subset S^3 = \partial D^4$ by associating to each $v\in \V(\Gamma)$ a standard unknot $L_{v}$ with framing $m(v)$ and Hopf linking $L_{v}$ and $L_{w}$ if and only if $v$ is adjacent to $w$. 
See Figure \ref{fig:plumbing and link ex} for an example. 
Define $X$ to be the result of attaching 2-handles to the 4-ball $D^4$ along $\mathcal{L}$ and define $Y$ to be the result of Dehn surgery along $\mathcal{L}$. 
Note, $Y = \partial X$. 

\begin{figure}[H]
\centering
\subcaptionbox{A plumbing graph $\agraph$.
\label{fig:plumbing ex}}[.45\linewidth]
{\begin{tikzpicture}
\node (m) at (0,0) {$\bullet$};
\node (r) at (2,0) {$\bullet$};
\node (l) at (-2,0) {$\bullet$};
\node (b) at (0,-2) {$\bullet$};
\node (s) at (4,0) {$\bullet$};
\node (t) at (6,0) {$\bullet$};

\draw (-2,0) --++ (8,0);
\draw (0,0) --++ (0,-2);

\node[above] at (m) {$-1$};
\node[above] at (r) {$-3$};
\node[above] at (l) {$-2$};
\node[right] at (b) {$-15$};
\node[above] at (s) {$-2$};
\node[above] at (t) {$-2$};

\end{tikzpicture}
}
\subcaptionbox{The framed link $\mathcal{L}(\agraph)$.
\label{fig:link ex}
}[.5\linewidth]
{\begin{tikzpicture}

\draw (.5,1) arc (90:180:1 and 1);

\draw[,preaction={draw, line width=6pt, white}]  (0,0) arc (0:180:1 and 1);

\draw[-{[scale=1.5]}](-2,0) arc (180:360:1 and 1);

\draw(1.5,-1.6) arc (0:90:1 and 1);

\draw[-{[scale=1.5]},preaction={draw, line width=6pt, white}] (.5,-1) arc (270:360:1 and 1);

\draw[-{[scale=1.5]},preaction={draw, line width=6pt, white}] (1,0) arc (180:360:1 and 1);

\draw[-{[scale=1.5]},preaction={draw, line width=6pt, white}] (2.5,0) arc (180:360:1 and 1);

\draw[-{[scale=1.5]},preaction={draw, line width=6pt, white}] (4,0) arc (180:360:1 and 1);

\draw[preaction={draw, line width=6pt, white}] (-.5,0) arc (180:270:1 and 1);

\draw[preaction={draw, line width=6pt, white}](.5,-.6) arc (90:180:1 and 1);

\draw[-{[scale=1.5]}] (-.5,-1.6) arc (180:360:1 and 1);

\draw[preaction={draw, line width=6pt, white}]  (6,0) arc (0:180:1 and 1);

\draw[preaction={draw, line width=6pt, white}]  (4.5,0) arc (0:180:1 and 1);

\draw[preaction={draw, line width=6pt, white}]  (3,0) arc (0:180:1 and 1);

\draw[preaction={draw, line width=6pt, white}] (1.5,0) arc (0:90:1 and 1);

\node[above] at (-1,1) {$-2$};

\node[above] at (2,1) {$-3$};

\node[above] at (3.5,1) {$-2$};

\node[above] at (5,1) {$-2$};

\node[above] at (.5,1) {$-1$};

\node[below] at (.5,-2.6) {$-15$};

\end{tikzpicture}
}
\caption{A plumbing $\agraph$ and its associated framed link $\mathcal{L}(\agraph)$. The $3$-manifold $\amfld(\agraph)$ is the Brieskorn sphere $\Sigma(2,7,15)$.}\label{fig:plumbing and link ex}
\end{figure}
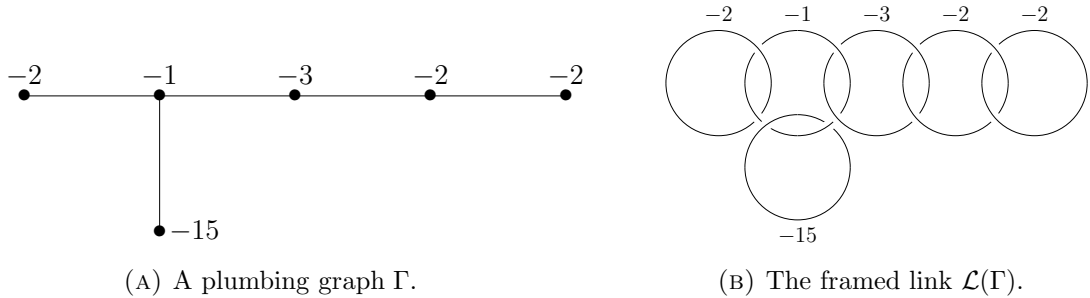

\begin{defn}\label{def: plumbed mfld}
    An oriented 3-manifold is \textit{plumbed} if it is diffeomorphic\footnote{``Diffeomorphism'' always means orientation-preserving diffeomorphism in this paper.} to $Y(\Gamma)$ for some plumbing graph $\Gamma$. 
    Moreover, if $\Gamma$ can be chosen such that the associated 4-manifold $X(\Gamma)$ has negative definite intersection form, we call the plumbed 3-manifold \emph{negative definite}.  
\end{defn}

\begin{rem}
    While the framed link $\mathcal{L}(\Gamma)$ is not uniquely determined by $\Gamma$, any pair of framed links built from $\Gamma$ in the manner described above will be isotopic. 
    Moreover, any such isotopy between framed links will determine a diffeomorphism between the corresponding manifolds. 
\end{rem}

We now identify various topological and algebraic quantities in terms of the data encoded by the plumbing graph $\agraph$. 
First, fix an orientation on $\mathcal{L}$ such that if $v$ is adjacent to $w$, then $\linkn(L_{v}, L_{w}) = +1$. 
For each $v\in \V(\Gamma)$, let $S^2_{v}\subset X$ denote the $2$-sphere obtained by capping off the core of the 2-handle attached to $L_{v}$ with a disk. 
We choose orientations on these spheres so that they agree with the orientation of $\mathcal{L}$ in the sense that if $v,w\in \V(\agraph)$ are adjacent, then the algebraic intersection number of $S^2_{v}$ and $S^2_{w}$ is equal to $\linkn(L_{v}, L_{w}) = +1$. 
By abuse of notation, we let $v\in H_{2}(X;\Z)$ denote the homology class of the oriented sphere $S^2_{v}$. 
Correspondingly, let $v^*\in H^2(X;\Z)$ denote the image of the Poincar\'{e} dual of $v$ under the map $H^2(X, \partial X;\Z)\to H^2(X;\Z)$ 
and let $v^{\dagger}\in \Hom(H_{2}(X;\Z),\Z)\cong H^2(X;\Z)$ denote the hom-dual of $v$, i.e., $v^{\dagger}(w)=\delta_{v,w}$. 
Then, $H_2(X;\Z)$ and $H^2(X;\Z)$ are free abelian groups with bases $\{v\}_{v\in \V(\agraph)}$ and  $\{v^{\dagger}\}_{v\in \V(\agraph)}$, respectively. 

Choose an ordering\footnote{This ordering is chosen for convenience, but is not essential. One could describe all of the constructions in this paper without this choice.} $v_{1}, \ldots, v_{s}$ of $\V(\Gamma)$. 
This ordering yields the following identifications  
\begin{align}
    H_{2}(X;\Z) &\cong \Z v_{1}\oplus \cdots \oplus \Z v_{s}\cong \Z^s, \label{eq:H_2} \\
    H^{2}(X;\Z) &\cong \Z{v}_{1}^{\dagger}\oplus\cdots\oplus \Z{v}_{s}^{\dagger}\cong \Z^s. \label{eq:H^2} 
\end{align}

We will often use the above identifications and work with vectors in $\Z^s$. We let $e_{i}$ denote the $i$-th standard basis vector in $\Z^s$ and for $x\in \Z^s$, we write $\hdual{x}$ for its transpose. 

For $v\in \V(\Gamma)$, we let $\delta(v)$ denote its degree. 
Given the chosen ordering $v_{1}, \ldots, v_{s}$ of $\V(\Gamma)$, we write $m = (m_1, \ldots, m_s),\;\delta = (\delta_1, \ldots, \delta_s) \in \Z^s$, with $m_i = m(v_i)$ and $\delta_i = \delta(v_i)$. 
They are called the \emph{weight vector} and the \emph{degree vector}, respectively. 

Under the identification in equation \eqref{eq:H_2}, the intersection form $\langle\cdot, \cdot\rangle: H_{2}(X;\Z)\times H_{2}(X;\Z)\to \Z$ is given by the adjacency matrix $M$ of $\Gamma$,
 \begin{align}
        M_{ij} = 
        \begin{cases}
        m_{i} &\text{if } i=j,\\
        1 &\text{if }i\neq j, v_{i}\text{ and }v_{j}\text{ share an edge,}\\
        0 &\text{otherwise}.
        \end{cases}
    \end{align}
Parallel to Definition \ref{def: plumbed mfld}, we say the graph $\Gamma$ is \emph{negative definite} if $M$ is negative definite. 

Given $K\in H^2(X;\Z)$, under the identification in equation \eqref{eq:H^2} we define
\begin{align}
    K^2 = \hdual{K}M^{-1}K.\label{eq:square}
\end{align}

The $\spinc$ structures of the 4-manifold $X(\Gamma)$ and the closed oriented 3-manifold $Y(\Gamma)$ can be conveniently described from the data of $\Gamma$. 
The first Chern class $c_1$ provides a bijection $c_{1}:\spinc(X)\to \Char(X)$, where $\Char(X)$ is the set of \textit{characteristic elements}. 
Specifically, 
\[
\Char(X) = \{ K \in H^2(X;\Z) \mid K(x) + \langle x, x\rangle \equiv 0 \bmod 2 \text{ for all } x\in H_2(X;\Z)\}.
\]
Therefore, we can think of $\spinc$ structures on $X$ as elements of the set $\Char(X)$. 
Under the identification in equation \eqref{eq:H^2}, $\Char(X)$ is identified with the set $m + 2\Z^s$. 
We write 
\begin{equation}
    \label{eq:char(Gamma)}
    \Char(\Gamma) = m +2\Z^s
\end{equation}
when working in coordinates. 

By considering the restriction of $\spinc$ structures on $X$ to the boundary $Y$, one can obtain a similar coordinate description of $\spinc(Y)$, which is standard in the lattice homology literature. 
Namely, there is a bijection
\begin{align}
    \spinc(Y)\xrightarrow{\sim} \frac{m+2\Z^s}{2M\Z^s} 
\end{align}
where the quotient on the right is via the action by $2M\Z^s$ on $m+2\Z^s$ given by $(2Mx, k)\mapsto k+2Mx$. We write
\begin{equation}
    \label{eq:spinc closed}
    \spinc(\Gamma)= \frac{m+2\Z^s}{2M\Z^s}
\end{equation}
when working in coordinates. 
For a characteristic vector $k\in \Char(\Gamma)$, let $[k]\in \spinc(\Gamma)$ denote its image in the quotient. 
Often we will think of $[k]$ as the lattice $\{k +2Mx\mid x\in \Z^s\}$, which is a sublattice of $\Char(\Gamma)$.

$\Spinc$ structures in general have a natural $\Z/2\Z$ \textit{conjugation} action. 
For $\spinc$ structures on $X$ thought of as elements of $\Char(\Gamma)$, the conjugation action sends $k\in m + 2\Z^s$ to $-k$. 
For $\spinc$ structures on $Y$ thought of as elements of $\spinc(\Gamma)$, the conjugation action sends $[k]\to [-k]$.

There is also an action of $H_{1}(Y;\Z)$ (or equivalently $H^2(Y;\Z)$ via Poincar\'{e} duality) on $\spinc(Y)$. 
We describe this action in coordinates. 
First, note that there is an isomorphism 
\[
H_1(Y;\Z) \cong \Z^s/M \Z^s 
\]
sending an oriented meridian linking $L_{v_i}$ positively to the coset of the standard basis vector $e_i$. 
Like for $\spinc(\Gamma)$, for $x\in \Z^s$, we write $[x]$ to denote its image in the quotient $\Z^s/M \Z^s$. 
Given $x\in \Z^s$ and $k\in \Char(\Gamma)$, the homology action is given by
\begin{equation}
     [x]\cdot [k] = [k+2x]\in \spinc(\Gamma)
\end{equation}
Note, the $H_1$ action is free and transitive. 

There is another realization of $\spinc(\Gamma)$ common in the literature, namely 
\begin{equation}
\label{eq:spinc with delta}
\spinc(\Gamma) = \dfrac{\delta+2\Z^s}{2M \Z^s}.
\end{equation}
The conjugation and homology actions using this definition are defined analogously: 
\begin{equation}
\begin{aligned}
\label{eq:conj and H_1 for delta version}
    [a] &\mapsto [-a], \\
    [x] \cdot[a] &= [a + 2x],
\end{aligned}
\end{equation}
for $a\in \delta + 2\Z^s$ and $x\in \Z^s$. 
To relate these two definitions of $\spinc(\Gamma)$, first let
\begin{equation}
\label{eq:\one}
    \one = (1,\ldots, 1).
\end{equation}
Note that $M\one= m + \delta$. 
Then, there is a bijection 
\begin{equation}
\label{eq:delta vs m}
    \psi: \dfrac{\delta + 2\mathbb{Z}^s}{2M\mathbb{Z}^s}  \xrightarrow{\sim}  \dfrac{m + 2\mathbb{Z}^s}{2M\mathbb{Z}^s},\ \psi([a]) = [a+Mu]
\end{equation}
which commutes with the conjugation and homology actions. 
See also \cite[Section 4.2]{GM}, in particular the discussion surrounding \cite[Equation (36)]{GM}. 
Unless otherwise stated, by $\spinc(\Gamma)$ we will mean the set in equation \eqref{eq:spinc closed}.

We now recall three moves on plumbing graphs called \emph{the type \hyperlink{A}{(A)}, \hyperlink{B}{(B)}, and \hyperlink{C}{(C)} Neumann moves}, described in Figure \ref{fig:Neumann moves closed}. 
Figure \ref{fig:Neumann moves closed} is to be interpreted as follows. 
A type \hyperlink{A}{(A)} move applied to a plumbing graph $\Gamma$ at an edge $e$ results in a new plumbing graph $\Gamma'$ which is identical to $\Gamma$ except the edge $e$ is subdivided into two edges meeting at a new vertex which is given a weight of $-1$ and the weights of the other two vertices bounding the original edge $e$ are both decreased by $1$. 
The type \hyperlink{B}{(B)} and \hyperlink{C}{(C)} moves are similarly interpreted from the figure.

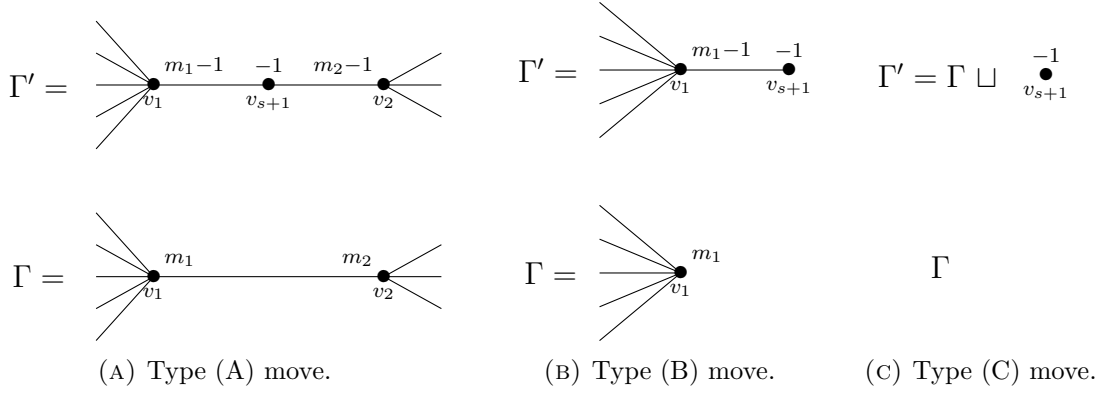
\begin{figure}[H]
\centering
\subcaptionbox{Type \hypertarget{A}{(A)} move. 
\label{fig:type a0}
}[.4\linewidth]
{\begin{tikzpicture}[scale=.85, xscale=.9]
\node(l) at (-2,0) {$\bullet$}; 

\node (m) at (0,0) {$\bullet$};

\node (r) at (2,0) {$\bullet$};

\draw (-2,0) --++ (-1,1);
\draw (-2,0) --++ (-1,.5);
\draw (-2,0) --++ (-1,0);
\draw (-2,0) --++ (-1,-.5);
\draw (-2,0) --++ (-1,-1);

\draw (-2,0) -- (2,0);

\draw (2,0) --++ (1,.5);
\draw (2,0) --++ (1,0);
\draw (2,0) --++ (1,-.5);

\node[anchor = south west] at (l) {$\mathsmaller{m_1 - 1}$};
\node[below] at (l) {$\mathsmaller{v_{1}}$};
\node[anchor = south] at (m) {$\mathsmaller{-1}$};
\node[below] at (m) {$\mathsmaller{v_{s+1}}$};
\node[anchor = south east] at (r) {$\mathsmaller{m_2 -1}$};
\node[below] at (r) {$\mathsmaller{v_{2}}$};

\node at (-4,0) {$\agraphp =$};

\begin{scope}[shift = {(0,-3)}]
\node(l) at (-2,0) {$\bullet$};

\node (r) at (2,0) {$\bullet$};

\draw (-2,0) --++ (-1,1);
\draw (-2,0) --++ (-1,.5);
\draw (-2,0) --++ (-1,0);
\draw (-2,0) --++ (-1,-.5);
\draw (-2,0) --++ (-1,-1);

\draw (-2,0) -- (2,0);

\draw (2,0) --++ (1,.5);
\draw (2,0) --++ (1,0);
\draw (2,0) --++ (1,-.5);

\node[anchor = south west] at (l) {$\mathsmaller{m_1}$};
\node[below] at (l) {$\mathsmaller{v_{1}}$};

\node[anchor = south east] at (r) {$\mathsmaller{m_2}$};
\node[below] at (r) {$\mathsmaller{v_{2}}$};

\node at (-4,0) {$\agraph  =$};
\end{scope}
\end{tikzpicture}
}
\subcaptionbox{Type \hypertarget{B}{(B)} move.
\label{fig:type b0}
}[.3\linewidth]
{\begin{tikzpicture}[scale=.9, xscale= .8]
\node (n) at (0,0) {$\bullet$};
\node (r) at (2,0) {$\bullet$};

\draw (0,0) --++ (-1.5,1);
\draw (0,0) --++ (-1.5,.5);
\draw (0,0) --++ (-1.5,0);
\draw (0,0) --++ (-1.5,-.5);
\draw (0,0) --++ (-1.5,-1);

\draw (0,0) --++ (2,0); 

\node[anchor = south west] at (n) {$\mathsmaller{m_1 -1}$};
\node[below] at (n) {$\mathsmaller{v_1}$};
\node[anchor = south] at (r) {$\mathsmaller{-1}$};
\node[below] at (r) {$\mathsmaller{v_{s+1}}$};
\node at (-2.45,0) {$\agraphp =$};

\begin{scope}[shift= {(0,-3)}]
\node (n) at (0,0) {$\bullet$};

\draw (0,0) --++ (-1.5,1);
\draw (0,0) --++ (-1.5,.5);
\draw (0,0) --++ (-1.5,0);
\draw (0,0) --++ (-1.5,-.5);
\draw (0,0) --++ (-1.5,-1);

\node[anchor = south west] at (n) {$\mathsmaller{m_1}$};
\node[below] at (n) {$\mathsmaller{v_1}$};
\node at (-2.4,0) {$\agraph  =$};
\end{scope}
\end{tikzpicture}
}
\subcaptionbox{Type \hypertarget{C}{(C)} move.
\label{fig:type C}
}[.2\linewidth]
{\begin{tikzpicture}[scale=.85]

\node at (-2.45,0) {$\Gamma' = \Gamma\  \sqcup $};

\node at (-.75,0) {$\bullet$};
\node[below] at (-.75,0) {$\mathsmaller{v_{s+1}}$};
\node[above] at (-.75,0) {$\mathsmaller{-1}$};

\begin{scope}[shift= {(0,-3)}]

\node at (-2.4,0) {$\Gamma$};
\node at (0,-1) {}; 

\end{scope}
\end{tikzpicture}
}
\caption{Two Neumann moves for negative definite plumbing graphs.}\label{fig:Neumann moves closed}
\end{figure}

For plumbing graphs $\Gamma$ and $\Gamma'$ related by one of the three moves in Figure \ref{fig:Neumann moves closed}, we use the symbol $'$ to denote data associated to $\Gamma'$, for instance $m', \delta'$, and $M'$. 
When writing in coordinates, we follow the convention that the ordering of the relevant vertices is as shown in Figure \ref{fig:Neumann moves closed}.

To each Neumann move we associate bijections 
\[
\alpha: \frac{ \delta + 2\Z^{s}}{2M\Z^s} \xrightarrow{} \frac{\delta'+ 2\Z^{s+1}}{2 M' \Z^{s+1}} \ \text{ and }\ \beta:
\frac{ m + 2\Z^{s}}{2M\Z^s} \xrightarrow{} \frac{m'+ 2\Z^{s+1}}{2 M' \Z^{s+1}}
\]
between the $\spinc$ structures of the corresponding plumbing graphs, in terms of both identifications \eqref{eq:spinc closed} and \eqref{eq:spinc with delta}. \\
\noindent
\textbf{Type \hyperlink{A}{(A)}}
\begin{align}
\label{eq:A0 closed}
\alpha : 
\alpha([a]) = [(a, 0)], \hspace{2em}\beta([k]) = [(k, 0) + (-1,-1, 0, \ldots, 0, 1)]
\end{align}
\textbf{Type \hyperlink{B}{(B)}}
\begin{align}
\label{eq:B0 closed}
    \alpha([a]) = [(a, 0) + (-1,0,\ldots,0,1)], \hspace{2em} \beta([k]) = [(k, 0) + (-1, 0,\ldots, 0,1)]
\end{align}
\textbf{Type \hyperlink{C}{(C)}}
\begin{align}
\label{eq:C}
    \alpha([a]) = [(a, 0)],\hspace{2em} \beta([k]) = [(k, -1)]
\end{align}
These maps fit into the commutative square   \eqref{eq:lattice square}.

\begin{equation}
\label{eq:lattice square}
\begin{tikzcd}
   \dfrac{ \delta + 2\Z^{s}}{2M\Z^s} \arrow[r,"\psi"] \ar[d, "\alpha"] & \dfrac{m+ 2\Z^{s}}{2 M \Z^{s}} \arrow[d, "\beta"]\\
   \dfrac{\delta'+ 2\Z^{s+1}}{2 M' \Z^{s+1}}\arrow[r,"\psi"]  & \dfrac{m'+ 2\Z^{s+1}}{2 M' \Z^{s+1}} 
\end{tikzcd}.
\end{equation}
The following explains the relevance of the Neumann moves.

\begin{thm}[{\cite{Neumann}}]
\label{thm:Neumann closed}
    Let $\agraph$ and $\agraph'$ be two negative definite plumbing trees. 
    Then $Y(\agraph)$ and $Y(\agraph')$ are diffeomorphic if and only if $\agraph$ and $\agraph'$ are related by a finite sequence of type \hyperlink{A}{(A)} and \hyperlink{B}{(B)} Neumann moves. 
\end{thm}

\begin{rem}
\label{rem:type C only appears for marked graphs}
    The above theorem restricts to negative definite trees and therefore uses only the type \hyperlink{A}{(A)} and \hyperlink{B}{(B)} moves. 
    In Section \ref{sec:Plumbed knot complements}, when discussing \emph{marked} plumbing graphs that yield manifolds with torus boundary, the type \hyperlink{C}{(C)} move will also be relevant. 
\end{rem}


\subsection{Plumbed knot complements}
\label{sec:Plumbed knot complements}
In this subsection, we describe an adaptation of the previous subsection to the setting of plumbed knot complements.    

A \textit{marked plumbing graph} \htarget{mgraph}$\mgraph$ is a plumbing graph with a distinguished, unweighted vertex $v_{0}$. 
We also require the graph to be a tree\footnote{One can generalize to forests, but this requires an extra  normalization in the main construction of this paper.}. 
For marked plumbing graphs, we will always index the vertices starting from $0$ rather than $1$, so that the $0$-th vertex $v_{0}$ is the marked one. 
When illustrating marked plumbing graphs, we use a hollow circle to represent the marked vertex. 

Given a marked plumbing graph $\mgraph$, we define the \textit{ambient plumbing graph} \htarget{agraph}$\agraph := \mgraph \setminus \{v_0\}$ to be the graph obtained from $\mgraph$ by deleting $v_{0}$ and all edges adjacent to $v_{0}$. 
Given an integer $m_{0}$, we also define the  \textit{surgered plumbing graph} \htarget{sgraph}$\sgraph$ to be the graph obtained from $\mgraph$ by giving $v_{0}$ the weight $m_0$. 
See Figure \ref{fig:marked, ambient, surgered} for an example. 

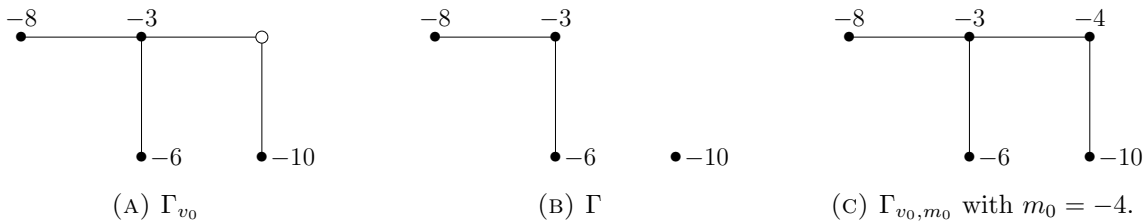
\begin{figure}[H]
\centering
\subcaptionbox{$\mgraph$
\label{fig:marked plumbing ex}}[.3\linewidth]
{\begin{tikzpicture}
\draw (0,0) circle [radius = .1];
\node (1) at (-2,0) {$\bullet$};
\node (2) at (-2,-2) {$\bullet$};
\node (3) at (-4,0) {$\bullet$};
\node (4) at (0,-2) {$\bullet$};

\draw (-4,0) -- (-.1,0);
\draw (-2,0) -- (-2,-2);
\draw (0,-.1) -- (0,-2);

\node[above] at (1) {$-3$};
\node[right] at (2) {$-6$};
\node[above] at (3) {$-8$};
\node[right] at (4) {$-10$};

\end{tikzpicture}
}
\subcaptionbox{$\agraph$
\label{fig:ambient plumbing ex}
}[.35\linewidth]
{\begin{tikzpicture}
\node (1) at (-2,0) {$\bullet$};
\node (2) at (-2,-2) {$\bullet$};
\node (3) at (-4,0) {$\bullet$};
\node (4) at (0,-2) {$\bullet$};

\draw (-4,0) -- (-2,0);
\draw (-2,0) -- (-2,-2);

\node[above] at (1) {$-3$};
\node[right] at (2) {$-6$};
\node[above] at (3) {$-8$};
\node[right] at (4) {$-10$};

\end{tikzpicture}
}
\subcaptionbox{$\sgraph$ with $m_{0} = -4$.
\label{fig:surgered plumbing ex}
}[.3\linewidth]
{\begin{tikzpicture}
\node (0) at (0,0) {$\bullet$};
\node (1) at (-2,0) {$\bullet$};
\node (2) at (-2,-2) {$\bullet$};
\node (3) at (-4,0) {$\bullet$};
\node (4) at (0,-2) {$\bullet$};

\draw (-4,0) -- (0,0);
\draw (-2,0) -- (-2,-2);
\draw (0,0) -- (0,-2);

\node[above] at (0) {$-4$};
\node[above] at (1) {$-3$};
\node[right] at (2) {$-6$};
\node[above] at (3) {$-8$};
\node[right] at (4) {$-10$};

\end{tikzpicture}
}
\caption{A marked plumbing and its corresponding ambient and surgered plumbing graphs.}\label{fig:marked, ambient, surgered}
\end{figure}

Fix a marked plumbing graph $\mgraph$ with $s+1$ vertices. Denote its degree vector by $\delta \in \Z^{s+1}$. 
We define three vectors $\onev, \adeg$, and  $m$ associated to the ambient plumbing graph $\agraph$.
First, set $\onev \in \Z^s$ to be the vector 
\begin{equation}
    \label{eq:onev vector}
    \onev_i =
    \begin{cases}
        1 & \text{ if } v_i \text{ is adjacent to } v_0,\\
        0 & \text{ otherwise.}
    \end{cases}
\end{equation}
Letting \htarget{adeg} $\adeg \in \Z^s$ denote the degree vector of the ambient plumbing graph $\agraph$, we have 
\[
\delta = (\delta_0, 0, \ldots, 0) + (0, \adeg + \onev).
\]
We also denote by $m = (m_1,\ldots, m_s)\in \Z^s$ the weight vector of ambient plumbing graph $\agraph$.

Consider the adjacency matrix $\mmatrix$ of $\mgraph$, where the diagonal entry corresponding to $v_0$ is left unspecified; that is, 
\[
\mmatrix =
    \begin{pmatrix}
    * &\hdual{\onev}\\
    \onev & \amatrix
    \end{pmatrix}
\]
where $M$ is the adjacency matrix of $\Gamma$. 
We say $\mgraph$ is \emph{negative definite} if $M$ is negative definite. 

Given a marked plumbing graph $\mgraph$, we associate two oriented 3-manifolds. 
As in the constructions in Section \ref{subsec:closed plumbed manifolds}, we begin by building a link $\mathcal{L} = \mathcal{L}(\mgraph)$, except now the component $L_{0}$ corresponding to $v_{0}$ does not have a framing. 
All other components have framing equal to the framing of the corresponding vertex of the plumbing. 
The \emph{ambient} $3$-manifold $\amfld$ is obtained by Dehn surgery on $\mathcal{L}\setminus L_0$, i.e., $Y= Y(\Gamma)$. 
Define a knot $\knot \subset \amfld$ to be the image of the unknot $L_{0}\subset S^3$ after performing Dehn surgery along $\mathcal{L}\setminus L_0$. 
We also define \htarget{complement} $\mmfld$ to be the complement of a tubular neighborhood of $\knot$ in $\amfld$ with a specified (unoriented) curve $\mu_{\knot}\subset \partial \mmfld$ given by a meridian of $\knot$. 

\begin{defn}
An oriented 3-manifold with torus boundary together with a specified curve $\gamma$ on its boundary is called a \emph{marked plumbed knot complement} if it is diffeomorphic to $\mmfld$ for some marked plumbing graph $\mgraph$ such that the diffeomorphism maps $\gamma$ to $\mu_{\knot}$. 
If $\mgraph$ can be chosen to be negative definite, we say $\mmfld$ is \textit{negative definite}.
\end{defn}

The vector $\onev$ described above has the following cohomological interpretation. 
Consider the 4-manifold $X = X(\Gamma)$ constructed from the ambient plumbing graph $\Gamma$ by attaching 2-handles to the 4-ball $D^4$ along the framed link $\mathcal{L}\setminus L_0$. 
Let $D^2_{0}$ be a properly embedded 2-disk in  $X$ obtained by taking a disk bounding $L_{0}$ and pushing it into $D^4$. 
In terms of the identification in \eqref{eq:H^2}, the Poincar\'{e} dual of this disk is precisely $\onev$ as defined in equation \eqref{eq:onev vector}. 

Next, we review \emph{relative spin$^c$ structures} on the manifold $\mmfld$ with torus boundary. 
We use Turaev's identification of relative $\spinc$ structures with \emph{smooth Euler structures} (see \cite{Turaev}). 
A smooth Euler structure on $\mmfld$ is an equivalence class of nowhere-vanishing vector field on $\mmfld$ that points outward along $\partial\mmfld$. 
The equivalence relation is given by declaring two such vector fields to be equivalent if there exists some $x\in \mmfld$ on which their restrictions to $\mmfld\setminus\{x\}$ are homotopic through nowhere-vanishing vector fields that point outward along the boundary. 
We will let $\spinc(\mmfld)$ denote the set of smooth Euler structures on $\mmfld$. 

Turaev \cite[Chapter VI]{Turaev} shows how a surgery presentation of a 3-manifold leads to an identification of $\spinc(\mmfld)$ with a certain lattice (or quotient lattice). 
Applying this to the surgery presentation of $\mmfld$ given by the framed link $\mathcal{L}$, one sees that $\spinc(\mmfld)$ is identified with the following set
\begin{equation}
\label{eq:relative spinc gluing}
    \dfrac{\delta + (1,\ldots, 1) + 2\Z^{s+1}}{2 \mmatrix (0 \times \Z^s)}.
\end{equation}
Note that the expression $\mmatrix (0\times \Z^s)$ is well-defined even though the entry labeled $*$ in $\mmatrix$ is unspecified because $\mmatrix (0,x) = (\hdual{\onev}x, \amatrix x) \in \Z^{s+1}$ does not depend on $*$. 
The elements of $\delta + (1,\ldots, 1) + 2\Z^{s+1}$ are called \emph{charges} in \cite[Section 2.2]{Turaev}.

The set $\spinc(\mmfld)$ comes equipped with a $\Z/2\Z$ conjugation action and an $H_{1}(\mmfld;\Z)$ action. 
In terms of \eqref{eq:relative spinc gluing}, these actions are given by 
\begin{align*}
        [b] &\mapsto [-b + (0,2,2, \ldots, 2)],\\
        [x] \cdot [b] &= [b+2x],
\end{align*}
respectively, for $b$ a charge and $x \in \Z^{s+1}$. 
Here $[x]$ is thought of as an element of $H_{1}(\mmfld;\Z)$ under the isomorphism $H_{1}(\mmfld;\Z)\cong \Z^{s+1} / \mmatrix(0 \times \Z^s)$ given by sending the $i$-th meridian of the link component $L_i$ of $\mathcal{L}$ to the coset of the standard basis vector $e_i$.

For consistency with the closed 3-manifold setting, we shift the numerator of \eqref{eq:relative spinc gluing} by $-(0,1,\ldots, 1)$ so that the conjugation action becomes multiplication by $-1$. 
With this shift in place, we make the following definition.  

\begin{defn}
\label{def:relative spinc} 
    The set of \emph{relative spin$^c$ structures} of the marked plumbing graph $\mgraph$, denoted $\spinc(\mgraph)$, is defined to be
    \begin{equation}\label{eq:rel spinc identification}
        \spinc(\mgraph) := \dfrac{\h{\delta} + 2\Z^{s+1}}{2 \mmatrix (0\times \Z^s)}.
    \end{equation}
    where $\h{\delta} = \delta+e_0$. 
\end{defn}
The conjugation and homology actions on $\spinc(\mgraph)$ are now given by
\begin{equation}
    \begin{aligned}
        \label{eq:conj and H_1 for marked graph}
        [b] &\mapsto [-b], \\
        [x] \cdot [b] &= [b+2x],
    \end{aligned}
\end{equation}
respectively, for $b\in \h{\delta} + 2\Z^{s+1}$ and $x\in \Z^{s+1}$. 

\begin{rem}
\label{rem:delta vs delta hat for rel spinc}
    In \cite[Equation (73)]{GM}, the set of relative $\spinc$ structures on $\mmfld$ is identified with 
    \begin{equation}
    \label{eq:GM rel spinc}
\dfrac{\delta+2\Z^{s+1}}{2\mmatrix(0 \times \Z^s)}.
    \end{equation}
    However, conjugation is not given by negating representatives in this identification. 
    For instance, the solid torus, which has no self-conjugate relative $\spinc$ structures, can be represented by the following marked plumbing
      
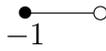
\begin{equation*}
        \begin{tikzpicture}
        \draw (0,0) -- (1,0);
            \node (a) at (0,0) {$\bullet$}; 
            \draw[fill=white] (1,0) circle (2.5pt);
            \node[below] (a) {$-1$};
        \end{tikzpicture}
    \end{equation*}
    while if using \eqref{eq:GM rel spinc}  we would have $[(-1,1)] = [(1,-1)]$. 
    One may verify that conjugation on \eqref{eq:GM rel spinc} is given by $[a] \mapsto [-a + (-2,0,\ldots, 0)]$. 
    Moreover, in \cite[Section 6.2]{GM}, the set of labels $[a]$ depends on the boundary parametrization (the framing $m_0$ of $v_0$), whereas the set of relative $\spinc$ structures does not. 
\end{rem} 

We now relate $\spinc(\mmfld)$ to $\spinc(\amfld)$, where $Y = Y(\mgraph \setminus \{v_0\})$. 
The gluing formula \cite[Ch. VI]{Turaev} applied to performing $\infty$-surgery on the component $L_0$ provides a map 
\[
\spinc(\mmfld) \times \spinc(S^1\times D^2) \to \spinc(Y).
\]
Choosing an orientation on the core of $S^1\times D^2$ (equivalently, an orientation of $L_0$) fixes an identification of $\spinc(S^1\times D^2)$ with the odd integers $1 + 2\Z$, where conjugation is given simply by negation. 
This provides a surjective map 
\[
\relmap_n : \spinc(\mmfld) \to \spinc(Y)
\]
for each $n\in 1+ 2\Z$.
\begin{rem}
    While $\relmap_n$ is defined for every odd integer $n$, for our main construction in Section \ref{sec:weighted graded root for plumbed knot complements} we will consider only $n=\pm 1$. 
    We note that $\relmap_{\pm 1}$ correspond to the two maps in \cite[Section 2.2]{OS-rational} given by picking an orientation on $\knot$.
\end{rem}

We describe $\relmap_n$ in terms of the coordinate identifications \eqref{eq:rel spinc identification} and \eqref{eq:spinc closed}. First, for each $x\in \Z^{s+1}$, let
\begin{equation}
    \label{eq:restriction to Gamma}
x\mapsto x\vert_{\agraph}
\end{equation}
denote the projection $\Z^{s+1} \twoheadrightarrow \Z^s$ given by forgetting the $v_0$-th entry.
Then, in coordinates, we have
\begin{equation}\label{eq: w_n coordinates}
    \relmap_n:\spinc(\mgraph)\to\spinc(\agraph),\hspace{2em} \relmap_n([b]) = [b\vert_{\agraph}+n(\lambda+Mu)].
\end{equation}
For those who prefer to use convention \eqref{eq:spinc with delta}, we define the following map,
\begin{equation}\label{eq: p_n}
    p_{n}:\spinc(\mgraph)\to\dfrac{\adeg + 2\Z^s}{2\amatrix \Z^s},\hspace{2em} p_n([b]) = [b\vert_{\agraph}+n\lambda].
\end{equation}
We then have the below commutative diagram.
\begin{center}
\begin{equation}
\label{eq:rel spinc to abs spinc map}
   \begin{tikzcd}[column sep = 0cm]
&    \spinc(\mgraph) \ar[dl, "p_{n}"'] \ar[dr, "\relmap_{n}"] & \\
\dfrac{\adeg + 2\Z^s}{2\amatrix \Z^s} \ar[rr, "\sim"] & & \spinc(\agraph)
\end{tikzcd}
\end{equation}
\end{center}

Next, we describe Neumann moves in the setting of marked plumbing graphs. 
The type \hyperlink{A}{(A)} and \hyperlink{B}{(B)} Neumann moves from Section \ref{subsec:closed plumbed manifolds} still apply. 
We also have two additional Neumann moves, called \hyperlink{A0}{(A0)} and \hyperlink{B0}{(B0)}, that involve the marked vertex $v_{0}$. 
They are described in Figure \ref{fig:Neumann moves}. The following is a consequence of \cite[Theorem 3.2]{Neumann}; see also  \cite[Section 1]{jackson2021invariance} and \cite[Section 2]{Niemi-Colvin}.
\begin{thm}[\cite{Neumann},  \cite{jackson2021invariance}, \cite{Niemi-Colvin}]

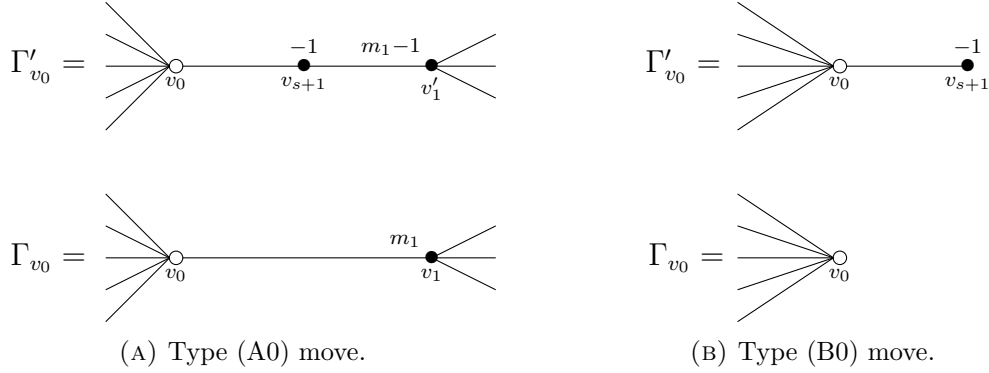
\begin{figure}
\centering
\subcaptionbox{Type \hypertarget{A0}{(A0)} move.
\label{fig:type a1}  }[.5\linewidth]
{\begin{tikzpicture}[scale=.85]
\draw (-2,0) circle (3pt);

\node (m) at (0,0) {$\bullet$};

\node (r) at (2,0) {$\bullet$};

\draw (-2.1,0) --++ (-1,1);
\draw (-2.1,0) --++ (-1,.5);
\draw (-2.1,0) --++ (-1,0);
\draw (-2.1,0) --++ (-1,-.5);
\draw (-2.1,0) --++ (-1,-1);

\draw (-1.9,0) -- (2,0);

\draw (2,0) --++ (1,.5);
\draw (2,0) --++ (1,0);
\draw (2,0) --++ (1,-.5);

\node[below] at (-2,0) {$\mathsmaller{v_{0}}$};
\node[anchor = south] at (m) {$\mathsmaller{-1}$};
\node[below] at (m) {$\mathsmaller{v_{s+1}}$};
\node[anchor = south east] at (r) {$\mathsmaller{m_{1} -1}$};
\node[below] at (r) {$\mathsmaller{v'_{1}}$};

\node at (-4,0) {$\mgraphp =$};

\begin{scope}[shift = {(0,-3)}]
\draw (-2,0) circle (3pt);

\node (r) at (2,0) {$\bullet$};

\draw (-2.1,0) --++ (-1,1);
\draw (-2.1,0) --++ (-1,.5);
\draw (-2.1,0) --++ (-1,0);
\draw (-2.1,0) --++ (-1,-.5);
\draw (-2.1,0) --++ (-1,-1);

\draw (-1.9,0) -- (2,0);

\draw (2,0) --++ (1,.5);
\draw (2,0) --++ (1,0);
\draw (2,0) --++ (1,-.5);

\node[below] at (-2,0) {$\mathsmaller{v_{0}}$};

\node[anchor = south east] at (r) {$\mathsmaller{m_1}$};
\node[below] at (r) {$\mathsmaller{v_{1}}$};

\node at (-4,0) {$\mgraph  =$};
\end{scope}
\end{tikzpicture}
}
\subcaptionbox{Type \hypertarget{B0}{(B0)} move.
\label{fig:type b1}}[.4\linewidth]
{\begin{tikzpicture}[scale=.85]
\draw (0,0) circle (3pt);
\node (r) at (2,0) {$\bullet$};

\draw (-.1,0) --++ (-1.5,1);
\draw (-.1,0) --++ (-1.5,.5);
\draw (-.1,0) --++ (-1.5,0);
\draw (-.1,0) --++ (-1.5,-.5);
\draw (-.1,0) --++ (-1.5,-1);

\draw (.1,0) --++ (1.9,0); 

\node[below] at (0,0) {$\mathsmaller{v_{0}}$};
\node[anchor = south] at (r) {$\mathsmaller{-1}$};
\node[below] at (r) {$\mathsmaller{v_{s+1}}$};
\node at (-2.45,0) {$\mgraphp =$};

\begin{scope}[shift= {(0,-3)}]
\draw (0,0) circle (3pt);

\draw (-.1,0) --++ (-1.5,1);
\draw (-.1,0) --++ (-1.5,.5);
\draw (-.1,0) --++ (-1.5,0);
\draw (-.1,0) --++ (-1.5,-.5);
\draw (-.1,0) --++ (-1.5,-1);

\node[below] at (0,0) {$\mathsmaller{v_0}$};
\node at (-2.4,0) {$\mgraph  =$};
\end{scope}
\end{tikzpicture}
}
\caption{Two  Neumann moves  involving the marked vertex for negative definite marked plumbing graphs.}\label{fig:Neumann moves}
\end{figure}

\label{thm:Neumann} 
Let $\mgraph$ and $\mgraph'$ be two negative definite marked plumbing graphs. 
Then the marked plumbed knot complements described by $\mgraph$ and $\mgraph'$ are diffeomorphic via a diffeomorphism identifying their specified boundary curves if and only if $\mgraph$ can be transformed into $\mgraph'$ by a finite sequence of the type \hyperlink{A}{(A)}, \hyperlink{A0}{(A0)}, \hyperlink{B}{(B)}, and \hyperlink{B0}{(B0)} moves. 
\end{thm}

The type \hyperlink{C}{(C)} move does not appear in the above theorem because we are restricting to the case that our marked plumbing graphs are trees, but it does appear when considering how the ambient plumbing graphs transform.
The \hyperlink{A}{(A)}, \hyperlink{A0}{(A0)}, \hyperlink{B}{(B)}, \hyperlink{B0}{(B0)} moves on $\mgraph$ result in the \hyperlink{A}{(A)}, \hyperlink{B}{(B)}, \hyperlink{B}{(B)}, \hyperlink{C}{(C)} on $\agraph$, respectively. For later use, we now record how the intersection forms of $\agraph$ and $\agraphp$ transform under Neumann moves. \\
\\
\begin{align}
    \begin{aligned}
    \label{eq:how matrices transform}
        &
        \amatrixp
\overset{\hyperlink{A}{(A)}}{=}
\SmallMatrix{
\amatrix &0\\
0&0
}
+
\SmallMatrix{
-1&-1&0&\cdots&0&\phantom{-}1\\
-1&-1&0&\cdots&0&\phantom{-}1\\
\phantom{-}0&\phantom{-}0&0&\cdots&0&\phantom{-}0\\
\phantom{-}\vdots&\phantom{-}\vdots&\vdots&\vdots&\vdots&\phantom{-}\vdots\\
\phantom{-}0&\phantom{-}0&0&\cdots&0&\phantom{-}0\\
\phantom{-}1&\phantom{-}1&0&\cdots&0&-1
}, 
&& 
\amatrixp
\overset{\hyperlink{A0}{(A0)}}{=}
\SmallMatrix{
\amatrix&0\\
0&0
}
+
\SmallMatrix{
-1&0&\cdots&0&\phantom{-}1\\
\phantom{-}0&0&\cdots&0&\phantom{-}0\\
\phantom{-}\vdots&\vdots&\vdots&\vdots&\phantom{-}\vdots\\
\phantom{-}0&0&\cdots&0&\phantom{-}0\\
\phantom{-}1&0&\cdots&0&-1\\
}
\\
& && \\
&
\amatrixp 
\overset{\hyperlink{B}{(B)}}{=}
\SmallMatrix{
\amatrix &0\\
0&0
}
+
\SmallMatrix{
-1&0&\cdots&0&\phantom{-}1\\
\phantom{-}0&0&\cdots&0&\phantom{-}0\\
\phantom{-}\vdots&\vdots&\vdots&\vdots&\phantom{-}\vdots\\
\phantom{-}0&0&\cdots&0&\phantom{-}0\\
\phantom{-}1&0&\cdots&0&-1\\
}, 
&&
\amatrixp
\overset{\hyperlink{B0}{(B0)}}{=}
\SmallMatrix{
\amatrix &0\\
0&0
}
+
\SmallMatrix{
0&\phantom{-}0\\
0&-1
}
    \end{aligned}
\end{align}

For a pair of marked plumbing graphs $\mgraph$ and $\mgraphp$ related by one of the four Neumann moves, we define maps 
\begin{align}\label{eqref:eq:alpha rel maps}
    \begin{aligned}[c]
         \alpha_{rel}:\spinc(\mgraph) \rightarrow  \spinc(\mgraph')
    \end{aligned}
    \hspace{2em}
    \begin{aligned}[c]
        &\operatorname{\hyperlink{A}{(A)}} &&  \alpha_{rel}([b])  = [(b,0)] \\
        & \operatorname{\hyperlink{A0}{(A0)}} && \alpha_{rel}([b])  = [(b,0)]  \\
        & \operatorname{\hyperlink{B}{(B)}} && \alpha_{rel}([b]) = [(b,0) + (0, 1, 0, \ldots, 0, -1)] \\
        & \operatorname{\hyperlink{B0}{(B0)}} && \alpha_{rel}([b]) = [(b,0) + (1,0, \ldots, 0,-1)]
    \end{aligned}
\end{align}
For each of the \hyperlink{A}{(A)}, \hyperlink{A0}{(A0)}, \hyperlink{B}{(B)}, \hyperlink{B0}{(B0)} moves on the marked plumbing, there are $\alpha$ and $\beta$ maps on the $\spinc$ structures of the corresponding ambient plumbing given by \eqref{eq:A0 closed}, \eqref{eq:B0 closed}, \eqref{eq:B0 closed}, \eqref{eq:C}, respectively. 
We summarize the relationship between these maps in the commutative diagram \eqref{eq:spinc maps Neumman moves}. 
All maps except $p_n$ and $\relmap_n$ are isomorphisms and commute with conjugation, and we have $p_n([-b]) = -p_{-n}([b])$ and $\relmap_n([-b]) = -\relmap_{-n}([b])$. 
\begin{equation}
\label{eq:spinc maps Neumman moves}
\begin{tikzcd}[column sep = tiny,row sep = 3ex]
    & & & \spinc(\mgraph') \ar[dddl, "p_n", swap] \ar[dddr, "\relmap_n", swap, pos = .3] & \\
   & & & &  \\
    & \spinc(\mgraph) \ar[dddl, "p_n", swap] \ar[rruu, "\alpha_{rel}"] & & & \\
    & & \dfrac{\adeg' + 2\Z^{s+1}}{2M\Z^{s+1}}  \ar[rr, "\psi"] & & \dfrac{m' +2\Z^{s+1}}{2M'\Z^{s+1}}\\
    & & & & \\
    \dfrac{\adeg +2\Z^{s}}{2M\Z^s} \ar[rruu, "\alpha"] \ar[rr, "\psi"] & & \dfrac{m +2\Z^{s}}{2M\Z^s}\arrow[from=uuul, crossing over, "\relmap_n", swap, pos = .3] \ar[rruu, "\beta"]& & 
\end{tikzcd}
\end{equation}


\section{Weighted graded roots for closed plumbed 3-manifolds}\label{sec:AJK}

In this section we summarize the main construction in \cite{AJK}, which takes the form of a \emph{weighted graded root}. 
\subsection{Graded roots} 
\label{subsec:graded roots}

We begin by recalling the definition of a \emph{graded root} from  \cite[Definition 3.2]{Nem_On_the}.

\begin{defn}
\label{def:graded root} A \emph{graded root} consists of 
\begin{itemize}
    \item an infinite tree $R$, with vertices and edges denoted $\V$ and $\Edges$ respectively, and
    \item a grading function $\chi : \V \to D$ where $D\subset \Q$ is of the form $D = n\Z + \Delta $ for some $n\in \Z$ and $\Delta\in \Q$. 
\end{itemize} 
We write an edge with endpoints $u$ and $v$ as $[u,v]\in \Edges$. 
The following properties must also be satisfied. 
\begin{enumerate}
    \item \label{graded root 1} $\chi(u) - \chi(v) = \pm n$ for any $[u,v] \in \Edges$. 
    \item \label{graded root 2} $\chi(u) < \max \{ \chi(v), \chi(w) \}$ for any $[u,v], [u,w] \in \Edges$ with $v\neq w$. 
    \item \label{graded root 3} Each preimage $\chi^{-1}(i)$ for $i\in D$ is finite.
    \item \label{graded root 4} $\chi$ is bounded above and $\vert \chi^{-1}(i) \vert =1$ for all $i\in D$ with  $i\ll 0$.
\end{enumerate}

An isomorphism of graded roots is an isomorphism of the underlying graphs that respects the grading. 
For $r\in \Q$, let $R\{r\}$ denote the graded root with the same underlying tree and whose grading function $\chi\{r\}$ is obtained from $\chi$ by shifting up by $r$; 
that is, $\chi \{r\}(v) = \chi(v) + r$. 
\end{defn}

The graded roots considered in \cite{Nem_On_the, AJK} differ slightly from the above, in that item \eqref{graded root 2} is replaced with $\chi(u) > \min \{ \chi(v), \chi(w) \}$ for any $[u,v], [u,w] \in \Edges$ with $v\neq w$, and item \eqref{graded root 4} is replaced with the condition that $\chi$ is bounded below and $\vert \chi^{-1}(i) \vert = 1$ for $i\in D$ with $i\gg 0$. 
When the distinction is needed, we refer to graded roots in Definition \ref{def:graded root} as \emph{downward pointing} and those with items \eqref{graded root 2} and \eqref{graded root 4} modified as described above as \emph{upward pointing}. 
One can be transformed into the other by negating the grading function $\chi$. 

In the present paper, we work with downward pointing graded roots whose grading function $\chi$ takes values in $2\Z + \Delta$ for some $\Delta\in \Q$. 
The goal of this section is in part to clarify the relationship between these conventions appearing in the literature; 
see also \cite[Section 2.3]{Dai-Manolescu}. 

Let $\agraph$ be a negative definite plumbing tree with $s$ vertices labeled $v_1, \ldots, v_s$, weight vector $m = (m_1, \ldots, m_s)\in \Z^s$, degree vector $\delta = (\delta_1, \ldots, \delta_s)\in \Z^s$, adjacency matrix $\amatrix$, and corresponding $3$-manifold $\amfld$. 
For $K\in \Char(\Gamma)$, set
\begin{equation}
\label{eq:h_U}
h_U(K) = \frac{K^2 +s}{4},
\end{equation}
where $K^2$ is computed as in equation \eqref{eq:square}.

Let $k \in m + 2\Z^s$ be a representative for a $\spinc$ structure $[k]$ on $Y$. 
For $h \in 2\Z + h_U(k)$, define the \emph{superlevel} set
\begin{equation}
    \label{eq:sublevel set}
    \superlevel_h(\Gamma,[k]) = \{ K   \in [k] \mid h_U(K) \geq h\}.
\end{equation}
These form the $0$-cells of a $1$-dimensional CW complex,  denoted $\superlevelone_h(\Gamma,[k])$, in which two $0$-cells $K, K'$ are connected by an edge if $K-K' = \pm 2 \amatrix {e_{i}}$ for some $1\leq i \leq s$. 
Let $\pi_0(\superlevelone_h(\Gamma,[k]))$ denote the connected components of this CW complex. 
Since $M$ is negative definite, $h_U$ may be viewed as a negative definite quadratic on $[k]$; in particular, $h_U$ has a maximal nonempty superlevel set. 

\begin{rem}
If $k' \in [k]$ is another representative, then $h_U(k') - h_U(k) \in 2\Z$, so that $2\Z + h_U(k) = 2\Z + h_U(k')$.
\end{rem}

\begin{defn}
\label{def:graded root Gamma}
The (downward pointing) graded root associated to $(\agraph, [k])$, denoted $\root(\agraph,[k])$, is the graph such that:
\begin{itemize}
    \item Vertices of $\root(\agraph,[k])$ are connected components of the superlevel sets over all $h \in 2\Z + h_U(k)$: 
    \[
    \mathcal{V}(\root(\agraph,[k])) = \bigcup_{h \in 2\Z + h_U(k)} \pi_0(\superlevelone_h(\agraph,[k])).
    \]
    The grading of a vertex $C \in  \pi_0(\superlevelone_h(\agraph,[k]))$ is defined to be $h$. 
    \item Two vertices of the graded root, represented by connected components $C \subset \superlevelone_h(\agraph;[k])$ and $C' \subset \superlevelone_{h'}(\agraph;[k])$, are connected by an edge in  $\root(\agraph,[k])$ if $h' = h - 2$ and $C\subset C'$. 
\end{itemize}
\end{defn}

\begin{rem}
    Elsewhere in the literature  \cite{Dai-Manolescu} the graded root as described above also includes an overall grading shift of $-2$. 
\end{rem}

We now summarize the translation between the above downward pointing graded root and the upward pointing one appearing in \cite{Nem_On_the,AJK}. For the latter, one starts with a $\spinc$ representative $k\in \Char(\agraph)$ and considers the function 
\[
\chi^{}_k : \Z^s\to \Z
\]
given by $\chi^{}_k(x) = -\frac{1}{2}(\hdual{k}x + \hdual{x}\amatrix x)$. 
Since $\amatrix$ is negative definite, $\chi^{}_k$ is a positive definite quadratic. 
One then considers \emph{sub}level sets $\chi^{-1}_k((-\infty,i])$ for $i\in \Z$, each of which is given the structure of a $1$-dimensional CW complex with vertices $\{x \in \Z^s \mid \chi^{}_k(x)\leq i\}$ and two vertices $x, x' \in \chi^{-1}_k((-\infty,i])$ connected by an edge if $x-x' = \pm {e_i}$ for some $1\leq i \leq s$. 
The graded root in \cite{Nem_On_the} is then defined to have vertices given by connected components of all these sublevel sets, with two components $C \subset \chi^{-1}_k((-\infty,i])$ and $C'\subset \chi^{-1}_k((-\infty,i'])$ connected by an edge if $i'= i + 1$ and $C \subset C'$. 
The grading of a vertex $C \subset \chi^{-1}_k((-\infty,i])$ is defined to be $2i$. 
We denote this graph by $\root^*(\agraph, k)$. 
Since $\chi^{}_k$ is negative definite and one considers sublevel sets, $\root^*(\agraph, k)$ has a minimal grading and is an upward pointing graded root. 
As explained in \cite[Proposition 4.4]{Nem_On_the}, if $k' = k + 2My$ is another representative of $[k]$, then there is an isomorphism 
\[
\root^*(\agraph,k') \cong \root^*(\agraph,k)\{-\chi^{}_k(y)\}.
\]
We normalize so that the minimal grading is 
\begin{equation}
\label{eq:graded root d invt}
-\max_{K\in [k]} h_U(K)
\end{equation}
and denote the resulting upward pointing graded root by $\root^*(\agraph, [k])$.   
Let $-\root^*(\Gamma, [k])$ denote the downward pointing graded root obtained by negating all the gradings in $\root^*(\Gamma, [k])$.

In contrast, the present paper follows the conventions in \cite{OSS,Niemi-Colvin}. 
Rather than picking a representative of a $\spinc$ structure and working with the lattice $\Z^s$, one works with the entire set $[k]\subset \Char(\agraph)$ as the lattice. 
For a fixed choice of representative $k\in [k]$, there is an identification $\Z^s \leftrightarrow [k]$ where a point $x\in \Z^s$ corresponds to the characteristic vector $K = k + 2Mx$. 
With this translation $x\leftrightarrow K= k + 2Mx$, an edge connecting $x$ and $x' = x \pm {e_i}$ in $\Z^s$ becomes an edge connecting $K = k + 2\amatrix x$ and $K' = K \pm 2\amatrix {e_i}$ in $[k]$. 
A straightforward computation reveals 
\[
-2\chi^{}_{k}(x) = h_U(K) - h_U(k),
\]
so that there is a isomorphism $\root(\agraph,[k]) \cong -\root^*(\agraph, [k])$.

We end this subsection by summarizing the relationship between graded roots and Heegaard Floer homology. 
Let $\Z[U]$ denote the graded polynomial ring with $U$ in degree $-2$. 
A (downward or upward pointing) graded root $(R,\chi)$ determines a graded $\Z[U]$-module $\mathbb{H}(R,\chi)$ as follows. 
Suppose $\chi$ takes values in $n\Z + \Delta$. 
As an abelian group, $\mathbb{H}(R,\chi)$ is freely generated by vertices of $R$, with gradings given by $\chi$. 
For $v\in \V(R)$ with $\chi(v) = i$, let $\{u_1, \ldots, u_\ell\}$ denote the set of vertices in  $\chi^{-1}(i-\vert n\vert)$ and which are connected to $v$ by an edge, and set $U \cdot v = u_1 + \cdots + u_\ell$ (note that if $R$ is downward pointing, then $\ell= 1$). 
For a downward (resp. upward) pointing graded root $(R,\chi)$, $\mathbb{H}(R,\chi)$ is precisely the homological degree zero part of lattice homology (resp. cohomology). 

\begin{thm}[\cite{Nem_Lattice_cohomology}] If $\Gamma$ is almost rational, then there are isomorphisms of graded $\Z[U]$-modules
\begin{align*}
    \mathbb{H}(\root^*(\Gamma,[k])) &\cong HF^+(-Y(\Gamma),[k]),\\
    \mathbb{H}({\root}(\Gamma,[k])\{-2\}) &\cong HF^-(Y(\Gamma),[k]),
\end{align*}
   where in the first isomorphism, $-Y(\agraph)$ denotes $Y(\agraph)$ with reversed orientation. 
\end{thm}

\begin{rem}
    Although in the present paper we focus on (bi)graded roots, which encode the homological degree zero part of (knot) lattice homology, a version of the above isomorphisms is known to hold in much greater generality. 
    There is a completed version of lattice  homology, built over $\F[[U]]$ where $\F=\Z/2\Z$, which can be defined for not necessarily negative definite plumbing trees \cite{OSS_spectral_sequence}. 
    Zemke \cite{Zem} has established the equivalence between this completed lattice homology and the corresponding completed version of $HF^-$. 
\end{rem}

\subsection{BPS $q$-series for closed plumbed 3-manifolds}
\label{sec:Zhat closed}

In this section, we review the BPS $q$-series for negative definite plumbings from \cite{GPPV}; see also \cite[Section 4.3]{GM}. 

Let $\agraph$ be a negative definite plumbing tree with $s$ vertices. 
We follow the conventions established in Section \ref{subsec:closed plumbed manifolds}. 
Let $a\in \delta + 2\Z^s$ be a representative of a $\spinc$ structure $[a]$ on $Y(\agraph)$, using the convention \eqref{eq:spinc with delta}. 
Define 
\begin{equation}
\label{eq:Zhat closed}
    \Zhat_a(q) := q^{- \frac{3s + \sum_v m_v}{4}}  \cdot 
    v.p.
    \oint\limits_{\vert z_v \vert =1}  \prod_{v\in \V(\Gamma)} \frac{dz_{v}}{2\pi i z_{v}}\left(z_{v}-z_v^{-1}\right)^{2-\delta_v} \cdot \Theta_a^{-M}(z),
\end{equation}
where 
\begin{equation}
\label{eq:theta function}
\Theta_a^{-M}(z) := \sum_{\ell \in a + 2M\Z^s} q^{-\frac{\hdual{\ell} M^{-1}\ell}{4}} \prod_{v\in \V(\Gamma)} z_v^{\ell_v}.
\end{equation}
In \eqref{eq:Zhat closed}, $v.p.$ indicates \emph{principal value}, the average of the integrals over $\vert z_v \vert = 1 + \epsilon$  and  $\vert z_v \vert = 1 -\epsilon$ for small $\epsilon>0$. 
Concretely, the integral is
\[
 \frac{1}{2}
  \left[\ 
    \oint\limits_{\vert z_v \vert =1-\eps}  \prod_{v\in \V(\Gamma)} \frac{dz_{v}}{2\pi i z_{v}}\left(z_{v}-z_v^{-1}\right)^{2-\delta_v}  \Theta_a^{-M}(z)
    +
    \oint\limits_{\vert z_v \vert =1+\eps}  \prod_{v\in \V(\Gamma)} \frac{dz_{v}}{2\pi i z_{v}}\left(z_{v}-z_v^{-1}\right)^{2-\delta_v}  \Theta_a^{-M}(z)
    \right]
\]
where  for $\delta_v \geq 3$ the term $\left(z_{v}-z_v^{-1}\right)^{2-\delta_v}$ is expanded as
\begin{equation}
\label{eq:two expansions zhat}
   \left( -\sum\limits_{i\geq 0} z_v^{2i+1}\right)^{\delta_v -2} \text{ if }  \vert z_v \vert <1 \text{ and } \left( \sum\limits_{i\geq 0} z_v^{-(2i+1)}\right)^{\delta_v -2}   \text { if }  \vert z_v \vert>1. 
\end{equation}

Applying $\displaystyle \oint\limits_{\vert z \vert = 1}  \frac{dz}{2\pi i z}$ to a Laurent series in $z$ or in $z^{-1}$ returns the constant term of the series. 
Consequently, the integral in \eqref{eq:Zhat closed} may be computed by taking one-half the sum of the two expansions in \eqref{eq:two expansions zhat} (which is a bi-infinite series in the variables $z_v$ for $v\in \V(\agraph)$), multiplying with $\Theta_a^{-M}(z)$, and recording the constant term. 
That $\amatrix$ is negative definite guarantees that the result is a well-defined Laurent series in $q$. 
For a further discussion  we refer the reader to \cite[Section 7]{AJK}. 
From now on we will omit $v.p.$ and the domain of integration from the notation.

Unifying the graded root and $\Zhat$ required an identification of the lattices used to define each theory. 
Namely, in \eqref{eq:Zhat closed}, the sum is over $a + 2\amatrix \Z^s = [a]$, whereas lattice homology is defined in terms of a sub-lattice of characteristic vectors $\Char(\agraph)$. 
In equation \eqref{eq:delta vs m} we identify $[a]$ with $[a+Mu] \in \spinc(\agraph)$. 
At the level of lattices, after translating between characteristic vectors and $\Z^s$ as described in Section \ref{subsec:graded roots}, the identification used in \cite{AJK} is 
\begin{align*}
    a + 2\amatrix \Z^s & \longleftrightarrow a + \amatrix u + 2\amatrix \Z^s \\
    \ell &\longleftrightarrow K = \ell + Mu.
\end{align*}
This is not canonical: any odd integer $n$ provides an identification $\ell \leftrightarrow K = \ell+ n Mu$. 
Note that, at the level of $\spinc$ structures, $[a+Mu] = [a + nMu]$ for any odd $n$. 
However, for the main constructions of this paper to be invariant under Neumann moves, this odd integer must be $\pm 1$. 
To emphasize this restriction, we denote by $\eps$, rather than $n$, a fixed choice of $\pm 1$.  The lattices are then identified via
\begin{align}
\begin{aligned}
\label{eq:identifying lattices closed}
    a + 2\amatrix \Z^s & \longleftrightarrow a + \amatrix u + 2\amatrix \Z^s \\
    \ell &\longleftrightarrow K = \ell + \eps Mu.
    \end{aligned}
\end{align}

With the identification \eqref{eq:identifying lattices closed} at hand, if we set $k = a+ Mu$, we can rewrite \eqref{eq:Zhat closed} as 
\begin{equation}
    \label{eq:Zhat closed in terms of K}
   \widehat{Z}_{a}(q) =    q^{-\frac{3s+\sum m_{v}}{4}} \sum_{K \in [k]} \h{W}_{\agraph}(K) 
q^{-\frac{(K-\eps \amatrix {\one})^2}{4}}
\end{equation}
The coefficient $\h{W}_{\agraph}(K)$, which depends on the choice of $\eps$, will be discussed in the next subsection.

\subsection{The weighted graded root for closed plumbed $3$-manifolds}

The main construction of \cite{AJK} introduces additional weights on the graded root. 
These weights depend on a choice of an \emph{admissible family of functions}, which we review now.

\begin{defn}[{\cite[Definition 4.1]{AJK}}] \label{def:admissible family}
Let $\Ring$ be a commutative ring. A family of functions $W = \{W_n : \Z \to \Ring\}_{n\geq 0}$ is called \emph{admissible} if 
\begin{enumerate}[label= (AD\arabic*)]
    \item\label{item:AD1} $W_2(0) = 1$ and $W_2(i) = 0$ for all $i\neq 0$. 
    \item \label{item:AD2} For all $n\geq 1$ and $i\in \Z$, 
    \[ W_n(i+1) - W_n(i-1) = W_{n-1}(i).\]
\end{enumerate}
\end{defn}

\begin{rem}
In \cite{AJK} an admissible family was denoted by $F$. 
Here we use a different notation to avoid confusion with the Gukov-Manolescu series $F_K$. 
\end{rem}
\begin{rem}
We will actually only use the value of $W_n$ at even numbers (resp. odd numbers) when $n$ is even (resp. odd), so we could have defined an admissible family as a family of functions $W = \{W_n : 2\Z+n \to \Ring\}_{n\geq 0}$ satisfying the two conditions above, and this would not change any of the discussions in this paper. 
\end{rem}

As noted in \cite[Equation (15)]{AJK}, conditions \ref{item:AD1} and \ref{item:AD2} determine $W_1$ and $W_0$:

\begin{equation}
\label{eq:W_1 and W_0}
W_1(i) =
\begin{cases}
1 & \text{ if } i=-1, \\
-1 & \text{ if } i=1, \\
0 & \text{ otherwise.} 
\end{cases}
\hskip3em 
W_0(i) =
\begin{cases}
1 & \text{ if } i=\pm 2, \\
-2 & \text{ if } i=0, \\
0 & \text{ otherwise.} 
\end{cases}
\end{equation}

Let us briefly discuss the above definition. 
In computing $\Zhat$ (for both closed plumbed manifolds and for plumbed knot complements) one encounters terms of the form $(z -z^{-1})^{2-n}$, $n\geq 0$. 
When $n>2$, such terms are expanded as a bi-infinite power series, and for $\Zhat$ the principal value dictates how to perform the expansion (see \cite[Section 7]{AJK}). 
Such an expansion is not unique but is essentially controlled by a choice of admissible family, as follows. 

Fix an admissible family of functions $W = \{W_n : \Z \to \Ring \}_{n\geq 0}$. 
Let $\BI$ denote the set of bi-infinite power series in a variable $z$, 
\[
\BI = \left\{ \sum_{j\in \Z} c_j z^j \mid c_j \in \Ring\right\}.
\]
In general, one cannot multiply two elements of $\BI$. 
However, $\BI$ is naturally a module over the ring of 
Laurent \emph{polynomials} $\Ring[z,z^{-1}]$. 
For $n\geq 0$, set 
\begin{equation}
\label{eq:expansion}
 (z -z^{-1})^{2-n} = \sum_{j\in \Z} W_n(-j) z^j \in \BI.
\end{equation}
The defining properties of an admissible family of functions implies that the above definition is coherent, in the following sense. 
First, property \ref{item:AD1} and equation \eqref{eq:W_1 and W_0} imply that if $0\leq n\leq 2$ then \eqref{eq:expansion} agrees with the usual expansion of $(z -z^{-1})^{2-n}$ as a Laurent polynomial. 
Moreover, property \ref{item:AD2} implies that if $0 \leq n'\leq n$ then
\[
(z-z^{-1})^{n'} \cdot (z -z^{-1})^{2-n} = (z -z^{-1})^{2-(n-n')}.
\]
Note that the first term on the left-hand side of the above equality is a Laurent polynomial, 
while the second term as well as the right-hand side are, in general, elements of $\BI$. 
In particular, if $n\geq 2$, then $(z-z^{-1})^{n-2} \cdot (z -z^{-1})^{2-n} = 1$, so that $(z -z^{-1})^{2-n}$, interpreted as in \eqref{eq:expansion}, provides an ``inverse''  to $(z-z^{-1})^{n-2}$ as an element of $\BI$.

For $x\in \Z^s$, let $x_i$ denote its $i$-th coordinate. Given an admissible family $W$ and a fixed $\eps \in \{\pm 1\}$, we define $W_{\agraph}: \Char(\agraph) \to \Ring$ by
\begin{equation}
\label{eq:W_Gamma weighted graded root}
W_{\agraph}(K) = \prod_{i=1}^s W_{\delta_i}((K-\eps \amatrix{\one})_i)
\end{equation}
where ${\one} = (1,\ldots, 1) \in \Z^s$  as defined in \eqref{eq:\one}.
Recall that for a $\spinc$ structure $[k]\in \spinc(\agraph)$, a vertex of the graded root $\root(\agraph, [k])$ corresponds to a connected 
component $C$ of some superlevel set. 
Its \emph{weight} is defined to be 
\begin{equation}
    \label{eq:weights in weighted graded root}
    W_{\agraph, [k]}(C; q,t) = q^{-\frac{3s+\sum m_{v}}{4}} t^{-\frac{ \eps \hdual{\one}  \amatrix {\one}}{2}} \sum_{K \in C\cap [k]} W_{\agraph}(K) q^{-\frac{(K-\eps \amatrix {\one})^2}{4}} t^{\frac{\hdual{K} {\one}}{2}}
\end{equation}
where on the left-hand side we include $[k]$ in the subscript to indicate that we are working within the sub-lattice of $\Char(\agraph)$ determined by the chosen $\spinc$ structure $[k]$, 
and in the sum on the right-hand side we write $K\in C\cap [k]$ to emphasize that the sum is over vertices ($0$-cells) of $C$. 
Note that $C$ has finitely many vertices since it is compact. 
We do not include the choice of $\eps$ in the notation in the left-hand sides of \eqref{eq:W_Gamma weighted graded root}
and \eqref{eq:weights in weighted graded root} to avoid clutter. The graded root equipped with these weights is called the \emph{weighted graded root} and denoted by 
\[
\root_\eps(\Gamma,[k],W).
\]

Let us explain the origin of the above weights. In \eqref{eq:Zhat closed in terms of K} the coefficients $\h{W}_{\agraph}(K)$ can be computed via \eqref{eq:W_Gamma weighted graded root} using the admissible family 
\begin{equation}
\label{eq:W hat}
\h{W} = \left\{ \h{W}_n: \Z \to \Q\right\}
\end{equation}
as defined in \cite[Definition 7.1]{AJK} (denoted by $\h{F}$ therein). Each $\h{W}_n$ is obtained from expanding $\left(z-z^{-1}\right)^{2-n}$ as a bi-infinite power series in a specific way. 
Namely, this expansion is one-half the sum of the two expansions in \eqref{eq:two expansions zhat}; see also \cite[Equations (33) and (34)]{AJK}. 
Using the identification of lattices $\ell \leftrightarrow K = \ell + \eps \amatrix u$ as in \eqref{eq:identifying lattices closed}, the contribution of $\ell$ to the integral in \eqref{eq:Zhat closed} is precisely 
\[
\h{W}_{\agraph}(K) q^{-\frac{3s+\sum m_{v}}{4}-\frac{(K-\eps \amatrix {\one})^2}{4}},
\]
where $[k] = [a+Mu]$.

The variable $t$ in \eqref{eq:weights in weighted graded root} was introduced in \cite{AJK}. 
When the admissible family is $\h{W}$, this gives a two-variable\footnote{To clarify, the result is a Laurent series in $q$ whose coefficients are Laurent polynomials in $t$.} refinement of $\Zhat$, denoted $\Zhathat$ in \cite[Section 7.3]{AJK}. 
This two-variable series can also be defined by modifying the integrand in \eqref{eq:Zhat closed}:
\begin{equation}
\label{eq:Zhathat closed}
    \Zhathat_a(q,t) := q^{- \frac{3s + \sum_v m_v}{4}}  \cdot
    \oint\limits \prod_{v\in \V(\Gamma)} \frac{dz_{v}}{2\pi i z_{v}}\left(t^{-1/2}z_v-t^{1/2}z_v^{-1}\right)^{2-\delta_v} \cdot \Theta_a^{-M}(z).
\end{equation}
Negative powers of $t^{-1/2}z_v-t^{1/2}z_v^{-1}$ are interpreted according to \eqref{eq:expansion} via the substitution $z \mapsto t^{-1/2}z_v$ and with respect to the admissible family $\h{W}$.
The integral is interpreted as recording the coefficient of the constant term of the integrand, as discussed in Section \ref{sec:Zhat closed}.  
Therefore, the contribution to the $t$-power for a fixed $K\in [k]$ is precisely 
\[
\dfrac{- \eps \hdual{u}\amatrix u + \hdual{K}u }{2}.
\]

More generally, for any admissible family $W$, consider
\begin{equation}
\label{eq:two variable series closed}
    q^{- \frac{3s + \sum_v m_v}{4}}  \cdot
     \oint  \prod_{v\in \V(\Gamma)} \frac{dz_{v}}{2\pi i z_{v}}\left(t^{-1/2}z_v-t^{1/2}z_v^{-1}\right)^{2-\delta_v} \cdot \Theta_a^{-M}(z).
\end{equation}
As for $\Zhathat$, in the above formula negative powers of $t^{-1/2}z_v-t^{1/2}z_v^{-1}$ are expanded according to \eqref{eq:expansion} with respect to $W$, and the integral records the constant term. 
The result is precisely the two-variable series defined in \cite[Section 6]{AJK}. 
Equivalently, as in \cite[Remark 6.5]{AJK}, \eqref{eq:two variable series closed} is equal to
\begin{equation}
    q^{- \frac{3s + \sum_v m_v}{4}} t^{-\frac{\eps \hdual{u}\amatrix u}{2}} \sum_{K\in [k]} W_{\agraph}(K) q^{-\frac{(K-\eps \amatrix {\one})^2}{4}} t^{\frac{\hdual{K} {\one}}{2}}.
\end{equation}
Note that this two-variable series does not depend on the choice of $\eps$. 

When $\eps=1$, the downward pointing weighted grading root $\root_1(\agraph, [h], W)$ is obtained from the upward pointing one introduced in \cite{AJK} by negating gradings. 
To see this, using the translation $x \leftrightarrow K= k + 2\amatrix x$ described in Section \ref{subsec:graded roots}, one can see that for $\eps=1$ we have 
\begin{align}
\begin{aligned}
\label{eq:comparing chi and h}
-\frac{3s+\sum m_{v} +(K-\amatrix  {\one})^2}{4} &= - \frac{ 3s + \sum m_v + (k-M  {\one})^2 }{4} + 2\chi^{}_k(x) + \hdual{x}\amatrix  {\one}, \\
\frac{\hdual{K} {\one} -  \hdual{\one}\amatrix  {\one}}{2} &= \frac{\hdual{k}{\one} -  \hdual{\one}\amatrix  {\one}}{2} + \hdual{x}M {\one}.
\end{aligned}
\end{align}
Therefore the contribution of $K$ to the weight \eqref{eq:weights in weighted graded root} is equal to the contribution of $x$ in \cite{AJK}; 
see in particular \cite[Equation (13), Notation 5.1, and Definition 5.2]{AJK}. 

\begin{thm}[{\cite[Theorem 5.9]{AJK}}] 
\label{thm:weighted graded root invariance}
Let $W$ be an admissible family of functions. Suppose $\agraph$ and $\agraphp$ are negative definite plumbing trees related by a type \hyperlink{A}{(A)} or type \hyperlink{B}{(B)} Neumann move. 
Let $\beta : \spinc(\agraph) \to \spinc(\agraphp)$ denote the corresponding bijection as in \eqref{eq:A0 closed} and \eqref{eq:B0 closed}, and let $[k]\in \spinc(\agraph)$. 
Then the weighted graded roots $\root_\eps(\agraph, [k], W)$ and $\root_\eps(\agraphp, [\beta(k)], W)$ are isomorphic. 
\end{thm}

For $\eps =1$ Theorem \ref{thm:weighted graded root invariance} follows from the above discussion and \cite[Theorem 5.9]{AJK}. 
Since the proofs of invariance in the two cases $\eps=\pm 1$ are essentially identical, rather than treating the case $\eps=-1$ separately we will provide a proof for a general $\eps$.

To begin, we define how the lattices will transform under Neumann moves. 
These maps are lifts of the maps $\beta$ in \eqref{eq:A0 closed}, \eqref{eq:B0 closed}, \eqref{eq:C} and will also be used later in Section \ref{sec:invariance}. 
To that end, we  define maps $\beta_{\pm}: m + 2\Z^s\to m' + 2\Z^{s+1}$ as follows:
\\
\\
\textbf{Type \hyperlink{A}{(A)}}
\begin{align}
\label{eq:A0 closed lattice}
\beta_{\pm}(k) = (k, 0) \pm(-1,-1, 0, \ldots, 0, 1)
\end{align}
\textbf{Type \hyperlink{B}{(B)}}
\begin{align}
\label{eq:B0 closed lattice}
\beta_{\pm}(k) = (k, 0) \pm(-1, 0,\ldots, 0,1)
\end{align}
\textbf{Type \hyperlink{C}{(C)}}
\begin{align}
\label{eq:C lattice}
\beta_{\pm}(k) = (k, \pm 1)
\end{align}
We note that while we record the maps $\beta_\pm$ for the Type \hyperlink{C}{(C)} move, as in Remark \ref{rem:type C only appears for marked graphs} it  will only be relevant in Section \ref{sec:BPS q series for plumbed knot complements} when considering marked plumbing graphs.

\begin{lem}
    \label{lem:beta1 and beta-1 induce isos on graded roots}
    For negative definite plumbing trees $\Gamma$ and $\Gamma'$ related by a type \hyperlink{A}{(A)} or \hyperlink{B}{(B)} Neumann move, the corresponding map $\beta_\pm$ from \eqref{eq:A0 closed lattice} and \eqref{eq:B0 closed lattice} induces an isomorphism of graded roots $\root(\agraph, [k]) \cong \root(\agraphp, [k'])$. 
\end{lem}

\begin{proof}
Let $K\in \Char(\agraph)$ and set $K' = \beta_\pm(K)$. We first verify that in all cases, $h_U(K') = h_U(K)$, so that $\beta_{\pm}$ restricts to a map $\superlevel_h(\agraph, [k]) \to \superlevel_h(\agraphp, [k'])$, for all $h\in h_U(k) + 2\Z$.
Note that this is equivalent to $\hdual{(K')} (\amatrixp)^{-1} K' = \hdual{K} \amatrix^{-1} K - 1$.

For the  type \hyperlink{A}{(A)} move, if $\amatrix^{-1} K = x= (x_1, x_2, \ldots, x_s)$, then it follows from \eqref{eq:how matrices transform} that $\amatrixp (x, x_1+x_2 \mp 1) = K'$. Then 
\[
\hdual{(K')} (\amatrixp)^{-1} K' = \hdual{[ (K,0) \pm (-1,-1,0,\ldots,0, 1)]}  (x, x_1+x_2\mp1) = \hdual{K} \amatrix^{-1} K -1. 
\]
For the type \hyperlink{B}{(B)} move, if $\amatrix^{-1} K = x= (x_1, \ldots, x_s)$, then from \eqref{eq:how matrices transform} we see that $\amatrixp (x, x_1\mp 1) = K'$. Then 
\[
\hdual{(K')} (\amatrixp)^{-1} K' = \hdual{[ (K,0) \pm (-1,0,\ldots,0, 1)]}  (x, x_1\mp1 ) = \hdual{K} \amatrix^{-1} K -1,
\]
which verifies $h_U(K') = h_U(K)$.

After performing the translation explained in Section \ref{subsec:graded roots}, one of $\beta_{+}$ or $\beta_{-}$ is equal to the map in \cite[Proposition 4.6]{Nem_On_the}, \cite[Proposition 3.4.2]{Nem_Lattice_cohomology}, which is shown to induce an isomorphism of graded roots. 
However, $\beta_+(K)$ and $\beta_{-}(K)$ are connected by an edge in $\superlevel_h(\agraphp, [k'])$ since $\beta_-(K) = \beta_{+}(K) + 2 \amatrixp e_{s+1}$. 
Therefore $\beta_+$ and $\beta_{-}$  induce (equal) isomorphisms of graded roots. 

\end{proof}

\begin{proof}[Proof of Theorem \ref{thm:weighted graded root invariance}]

Fix $h\in h_U(k) + 2\Z$. 
For a connected component $C \subset \superlevelone_h(\agraph, [k])$, let $C' \subset \superlevelone_h(\agraph, [k])$ be the corresponding component of $\superlevelone_h(\agraphp, [k'])$ under the isomorphism in Lemma \ref{lem:beta1 and beta-1 induce isos on graded roots}. 
Precisely, $C'$ is the connected component which contains $\beta_\pm(K)$ for any $K\in C$. 
For the type \hyperlink{A}{(A)} move we will use the map $\beta_\eps$ from \eqref{eq:A0 closed lattice} corresponding to our fixed $\eps\in \{\pm 1\}$, while for the type \hyperlink{B}{(B)} move we will use both maps $\beta_\pm$ from \eqref{eq:B0 closed lattice}.

Let us address the type \hyperlink{A}{(A)} move. 
For $K'\in \Char(\agraphp)$, $ K'-\eps M' u = K' - \eps (Mu,0) - \eps (-1,-1,0,\ldots, 0, 1)$. 
Since $\delta'_{s+1}=2$, property \ref{item:AD1} implies that $W_{\agraphp}(K') = 0$ if $K'_{s+1}\neq \eps$. 
If $K'_{s+1}= \eps$, then by setting $K = (K'_1 +\eps, K'_2 + \eps, K'_3, \ldots, K'_s) \in \Char(\agraph)$ we have $\beta_{\eps}(K) = K'$. 
Therefore only vertices in the image of $\beta_\eps$ contribute to $W_{\agraphp, [k']}(C';q,t)$. 
The contributions of $K$ and $K'$ are equal because 
\begin{align*}
    W_{\agraphp}(K') = \prod_{i=1}^{s+1} W_{\delta_i'}((K'-\eps M'u)_i) &= \prod_{i=1}^{s} W_{\delta_i}((K-\eps Mu)_i) \cdot W_{2}(0) = W_{\agraph}(K), \\
    \hdual{(K')}u &= \hdual{K}u - \eps, \\
    \eps \hdual{u} (M')u  &= \eps \hdual{u} M u -\eps, \\
    3(s+1) + \sum_{v\in \V(\agraphp)} m'_v &= 3s + \sum_{v\in \V(\agraph)} m_v, \\
    (K'-\eps M' u)^2 & = (K-\eps Mu)^2,
\end{align*}
which completes the proof in the type \hyperlink{A}{(A)} case. 

Let us now address the type \hyperlink{B}{(B)} move. 
For $K'\in \Char(\agraphp)$, $K' - \eps M'u =
K' -\eps (Mu,0)$. Since $\delta_{s+1}' = 1$, from \eqref{eq:W_1 and W_0} we see that $W_{\agraphp}(K')= 0$ if $K'_{s+1}\neq \pm 1$. 
If $K'_{s+1} = \pm 1$, then $K'$ is in the image of $\beta_{\pm }$:  
\[
K' = \begin{cases}
\beta_+(K'_1 +1, K'_2, \ldots, K'_s) & \text{ if } K'_{s+1} = 1, \\
\beta_{-}(K'_1 -1, K'_2, \ldots, K'_s) & \text{ if } K'_{s+1} = -1 ,
\end{cases}
\]
so only vertices in the image of $\beta_\pm$ contribute to the weight. 
For $K\in \superlevel_h(\agraph, [k])$, we set $\til{K} = K -\eps M u$. Then, using \ref{item:AD2} and \eqref{eq:W_1 and W_0}, we have 
\begin{align*}
&\quad\;  
W_{\agraphp}(\beta_{-}(K)) + W_{\agraphp}(\beta_{+}(K)) \\
&= \prod_{i=1}^{s+1} W_{\delta'_i}(((\til{K},0) + (1,0,\ldots, 0, -1))_i) + \prod_{i=1}^{s+1} W_{\delta'_i}(((\til{K},0) + (-1,0,\ldots, 0, 1))_i) \\
&= W_{\delta_1 +1} (\til{K}_1 +1 ) W_1(-1)\prod_{i=2}^{s} W_{\delta_i}(\til{K}_i) + W_{\delta_1 +1} (\til{K}_1 - 1 ) W_1(1) \prod_{i=2}^{s} W_{\delta_i}(\til{K}_i) \\
&  = \left[ W_{\delta_1 +1} (\til{K}_1 +1 )  - W_{\delta_1 +1} (\til{K}_1 - 1 )  \right] \prod_{i=2}^{s} W_{\delta_i}(\til{K}_i) \\
&= W_{\agraph}(K).
 \end{align*}
We also have 
\begin{align*}
\left(\beta_{\pm}(K) - \eps M'u\right)^2 &= (K-\eps Mu)^2  - 1, \\
3(s+1) + \sum\limits_{v\in \V(\Gamma')} m'_v &= 3s + \sum_{v\in \V(\Gamma)} m_v + 1,  \\
\hdual{(\beta_{\pm }(K))}u  - \eps\hdual{u} M' u &= \hdual{K}u - \eps\hdual{u}Mu,
\end{align*}
and it follows that $W_{\agraph,[k]}(C; q,t)  = W_{\agraphp, [k']}(C';q,t)$. 
\end{proof}

\subsection{Conjugation of $\spinc$ structures revisited}
\label{sec:spinc conjugation revisited}

In this section we analyze the effect of $\spinc$ conjugation on the weighted graded root. 
We begin by recalling the following property of an admissible family of fuctions $W$, introduced in \cite[Section 8]{AJK}:
\begin{equation}
\label{eq:AD3}
    W_n(-i) = (-1)^n W_n(i) \text{ for all } n\geq 0 \text{ and } i\in \Z. \tag{AD3} 
\end{equation}
While $\h{W}$ satisfies \eqref{eq:AD3}, not all admissible families do (see, for example, the admissible families $\h{F}^{\pm}$ from \cite[Definition 7.2]{AJK}). 
As discussed in \cite[Example 8.4]{AJK}, the weighted graded root is not invariant under $\spinc$ conjugation. 
However, we have the following result, which may be viewed as a refinement of \cite[Proposition 8.1]{AJK}.

\begin{prop}
\label{prop:epsilon vs -epsilon}
    Let $W$ be an admissible family of functions which satisfies \eqref{eq:AD3}. 
    For any negative definite plumbing tree $\agraph$ and $\spinc$ structure $[k]\in \spinc(\agraph)$, $\root_{-\eps}(\agraph, [-k], W)$ is obtained from $\root_\eps(\agraph, [k], W)$ by the change of variables $t\mapsto t^{-1}$. 
\end{prop}

\begin{proof}
    Consider the involution $\iota$ of $\Char(\agraph)$ given by $\iota(K) = -K$. 
    We have $h_U(K) = h_U(\iota(K))$, so that $\iota$ restricts to a map of superlevel sets $\iota: \superlevel_h(\agraph, [k]) \to \superlevel_h(\agraph, [-k])$, which is evidently an isomorphism of $1$-dimensional CW complexes. 
    Thus $\iota$ induces an isomorphism of graded roots $\iota_*: \root(\agraph, [k]) \xrightarrow{\sim} \root(\agraph, [-k])$, and we will show $\iota_*$ respects the weights. 

    We have 
    \[
\prod_{v\in \V(\agraph)} W_{\delta_v} (( -K +\eps \amatrix u)_v) = (-1)^{\sum_v \delta_v} \prod_{v\in \V(\agraph)} W_{\delta_v} (( K-\eps \amatrix u)_v) = \prod_{v\in \V(\agraph)} W_{\delta_v} (( K-\eps \amatrix u)_v)
    \]
    where the first equality follows from property \eqref{eq:AD3}, and the second equality follows from the fact that the sum of degrees in any graph is even. 
    The left-most term above is the weight of $-K$, equation \eqref{eq:W_Gamma weighted graded root}, for $-\eps$, while the right-most term is the weight of $K$ for $\eps$. 
    To address the powers of $q$ and $t$, we have
    \begin{align*}
    (-K + \eps \amatrix {\one})^2 & = (K-\eps \amatrix {\one})^2, \\
       \frac{\eps \hdual{u}\amatrix u + \hdual{(-K)}u }{2} & =  -\frac{- \eps \hdual{u}\amatrix u + \hdual{K}u }{2}
    \end{align*}
    which completes the proof. 
\end{proof}

The above result implies that if $[k]$ is self-conjugate, then the corresponding weighted graded roots for $\eps$ and $-\eps$ differ by replacing $t$ and $t^{-1}$.


\section{Weighted bigraded roots for plumbed knot complements}\label{sec:weighted_bigraded_roots}

\subsection{Knot lattice homology and bigraded roots}

In this subsection we review part of the construction in \cite{Niemi-Colvin}, following the notation established in Section \ref{sec:Plumbed knot complements}.

Let $\mgraph$ be a negative definite marked plumbing graph with $\vert \V(\mgraph)\vert = s+1$. 
As usual, we set $\agraph = \mgraph \setminus \{v_0\}$ and $\amfld = Y(\agraph)$. 
Denote by $m\in \Z^s$ and by $\amatrix$ the weight vector and adjacency matrix of $\agraph$, respectively. 
Let $k \in \Char(\agraph) = m + 2\Z^{s}$ be a representative of a $\spinc$ structure $[k] \in \spinc(\agraph)$.

Recall  the function $h_U : [k] \to 2\Z + h_U(k)$ from \eqref{eq:h_U}. Given $K \in [k]$, define
\begin{align}
  \label{eq:h_V}  h_{V}(K) &= h_{U}(K+2 {\onev}),  \\
    \label{eq:A} A(K) &= \frac{h_U(K)-h_V(K)}{2},
    \end{align}
where $\onev$ is as defined in \eqref{eq:onev vector}. The quantity $A(K)$ is called the \emph{Alexander grading} of $K$.

Let us discuss a useful alternative perspective on the Alexander grading. 
For $m_0\in \Z$, recall that $\sgraph$ denotes the plumbing tree obtained from $\mgraph$ by giving $v_0$ the framing $m_0$, with corresponding intersection form on its associated $4$-manifold $X(\sgraph)$ denoted by $\smatrix$. 
Pick any $m_0$ so that $\sgraph$ is negative definite. 
A computation yields
\begin{equation}
\label{eq:inverse of e_0}
    \smatrix^{-1} e_{0}  =\frac{1}{m_0- \hdual{\onev}\amatrix ^{-1} \onev }\begin{pmatrix}
        1\\
        -\amatrix^{-1}\onev
            \end{pmatrix}.
\end{equation}
Following \cite[Equation (3.1)]{OSS}, let \htarget{Sigma}
\[
\Sigma \in H_2(X(\sgraph);\Q) \cong \Q^{s+1}
\]
denote the  element whose $v_0$-th entry is $1$ and which satisfies $\hdual{e_i} \smatrix \Sigma =0$ for $1\leq i \leq s$. 
This element exists and is unique. 
In fact, it follows from \eqref{eq:inverse of e_0} that
\begin{equation}
    \label{eq:Sigma}
    \Sigma = \dfrac{\smatrix^{-1} e_0}{\hdual{e_0}\smatrix^{-1}e_0} = (1, -\amatrix^{-1}\onev). 
\end{equation}

The Alexander grading of $K\in \Char(\agraph)$ in \cite[Definition 3.2]{OSS} is defined to be
\[
\frac{1}{2} ( L_K(\Sigma) + \Sigma^2)
\]
where $L_K = (-m_0, K) \in \Char(\sgraph)$. 
In terms of our coordinates, when writing $L\in \Z^{s+1}$ and $\Sigma\in \Q^{s+1}$, $L(\Sigma)$ is  $\hdual{L}\Sigma$ and the homology pairing $\Sigma^2$ of $\Sigma$ with itself is given by $\hdual{\Sigma} \smatrix \Sigma$. Note, $\Sigma^2 < 0$ since $\smatrix$ is negative definite. 
From \eqref{eq:Sigma} we see that 
\begin{equation}
    \label{eq:Sigma^2}
    \Sigma^2 = m_0 - \hdual{\onev}\amatrix^{-1} \onev,
\end{equation}
and it follows that $ \frac{1}{2} ( L_K(\Sigma) + \Sigma^2) = A(K)$. 

For $(i,j)\in (2\Z+h_{U}({k}))\times(2\Z+h_{V}({k}))$, we define the following sets:
\begin{align*}
    \superlevel^{U}_{i}(\mgraph,[k])&= \{K\in [{k}]\mid h_{U}(K)\geq i\},\\
    \superlevel^{V}_j (\mgraph,[k]) &= \{K\in [k]\mid h_{V}(K)\geq j\},\\
    \superlevel_{i,j}(\mgraph,[k]) &= \superlevel^U_{i}(\mgraph,[k])\cap \superlevel^V_{j}(\mgraph,[k]). 
\end{align*}
We give each of these sets the structure of a 1-skeleton by declaring the elements of the above sets to be the $0$-cells, two of which $K_{1}, K_{2}$ share an edge if and only if $K_{1}-K_{2}=\pm 2M  {e_i}$ for some $1\leq i\leq s$. 
We denote these $1$-skeleta by $\superlevelone{}^{U}_{i}(\mgraph,[k]), \superlevelone{}^{V}_j (\mgraph,[k])$, and $\superlevelone_{i,j}(\mgraph,[k])$, and denote their connected components by $\pi_{0}(\superlevelone^U_{i}( \mgraph,[k])), \pi_0(\superlevelone^V_{j}( \mgraph,[k])), \pi_{0}(\superlevelone_{i,j}(\mgraph,[k]))$.

\begin{rem}
We note that $\superlevelone_{i,j}(\mgraph,[k])$ are the $1$-skeleta of the double filtration studied in \cite{Niemi-Colvin}, see in particular \cite[Section 5.2]{Niemi-Colvin}.
\end{rem}

\begin{defn}
\label{def:three graded roots}
We define three graphs associated to this data:
\begin{itemize}
    \item The \emph{bigraded root} $\rootbi(\mgraph,[k])$. Its set of vertices is  
    \[
    \bigcup_{ \substack{ i \in 2\Z+h_{U}({k})\\ j\in 2\Z+h_{V}({k})} }
   \pi_{0}(\superlevelone_{i,j}(\mgraph,[k])).
    \]
    Two vertices corresponding to connected components $C_{i,j} \subset \superlevelone_{i,j}(\mgraph,[k])$ and $C_{i', j'} \subset \superlevelone_{i',j'}(\mgraph,[k])$ are connected by an edge in $\rootbi(\mgraph,[k])$ if either $i'=i+2, j' = j$ and $C_{i',j'}\subset C_{i,j}$, or if $i'=i, j'=j+2$ and $C_{i',j'}\subset C_{i,j}$. The vertex $C_{i,j}$ lies in bigrading $(i,j)$. 
    \item The \emph{$U$-graded root} $\root^U(\mgraph,[k])$ whose set of vertices is 
    \[
    \bigcup\limits_{i \in 2\Z+h_{U}({k})} \pi_{0}(\superlevelone{}^U_{i}( \mgraph,[k])),
    \]
    and two vertices $C_{i} \subset \superlevelone^U_{i}( \mgraph,[k])$ and $C_{i'} \subset \superlevelone{}^U_{i'}( \mgraph,[k])$ are connected by an edge if $i'=i+2$ and $C_{i'}\subset C_i$. 
    The grading of $C_{i}$ is defined to be $i$. 
    \item The \emph{$V$-graded root} $\root^V(\mgraph,[k])$ is defined in an entirely analogous way to the $U$-graded root, but with the $U$ replaced with $V$. 
\end{itemize}
Note that $\root^U(\mgraph,[k]) = \root(\agraph,[k])$ and that $\root^V(\mgraph, [k]) = \root(\agraph, [k+2\onev])$. 
\end{defn}

\begin{exmp}
\label{ex:bigraded root unknot in S3}
Consider the graph $\mgraph$ shown below. 
\begin{equation*}
    \begin{tikzpicture}
    \draw (0,0) -- (1,0);
        \node (a) at (0,0) {$\bullet$}; 
        \draw[fill=white] (1,0) circle (2.5pt);
        \node[below] (a) {$-1$};
    \end{tikzpicture}
\end{equation*}
The ambient $3$-manifold $Y$ is $S^3$, and the image of the unknot corresponding to the marked vertex is the unknot in $S^3$. 
There is one $\spinc$ structure on $Y$, which we denote $\mathfrak{s}_0$.  For $K\in \Char(\agraph) = 1 +2\Z$, 
\[
h_U(K) = \frac{K^2 + 1}{4}, \hskip2em h_V(K) = \frac{(K+2)^2 +1}{4},
\]
The maximal nonempty superlevel set $\superlevelone_{i,j}(\mgraph,\mathfrak{s}_0)$ is at bigrading $(i,j) = (0,0)$, and since in this case the lattice is $1$-dimensional, by convexity we see that $\superlevelone_{i,j}(\mgraph,\mathfrak{s}_0)$ consists of one connected component for each $(i,j) \in 2\Z_{\leq 0} \times 2\Z_{\leq 0}$. 
The bigraded root is  shown in \eqref{eq:bigraded root unknot}. 
    \small
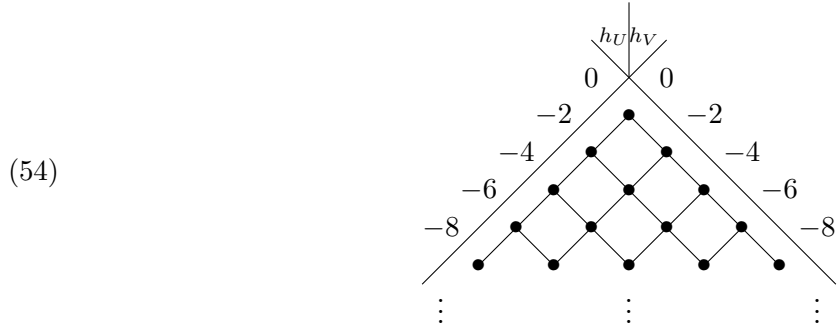
    \begin{equation}
    \begin{aligned}\label{eq:bigraded root unknot}
        \begin{tikzpicture}[scale=0.5]
\node at (0,0) {$\bullet$};
\node at (-1,-1) {$\bullet$};
\node at (1,-1) {$\bullet$};
\node at (-2,-2) {$\bullet$};
\node at (0,-2) {$\bullet$};
\node at (2,-2) {$\bullet$};
\node at (-3,-3) {$\bullet$};
\node at (-1,-3) {$\bullet$};
\node at (1,-3) {$\bullet$};
\node at (3,-3) {$\bullet$};
\node at (-4,-4) {$\bullet$};
\node at (-2,-4) {$\bullet$};
\node at (0,-4) {$\bullet$};
\node at (2,-4) {$\bullet$};
\node at (4,-4) {$\bullet$};

\draw (0,0) -- (-1,-1); 
\draw (0,0) -- (1,-1);
\draw (-2,-2)--(-1,-1)--(0,-2)--(1,-1)--(2,-2);
\draw (-3,-3)--(-2,-2)--(-1,-3)--(0,-2)--(1,-3)--(2,-2)--(3,-3);
\draw (-4,-4)--(-3,-3)--(-2,-4)--(-1,-3)--(0,-4)--(1,-3)--(2,-4)--(3,-3)--(4,-4);


\node at (-5, -5) {\reflectbox{$\vdots$}};
\node at (5, -5) {$\vdots$};
\node at (0, -5) {$\vdots$};

\node at (1, 1) {$0$};
\node at (2, 0) {$-2$};
\node at (3, -1) {$-4$};
\node at (4, -2) {$-6$};
\node at (5, -3) {$-8$};

\node at (-1, 1) {$0$};
\node at (-2, 0) {$-2$};
\node at (-3, -1) {$-4$};
\node at (-4, -2) {$-6$};
\node at (-5, -3) {$-8$};

\draw (-1,2)--(5.5,-4.5);
\draw (1,2)--(-5.5,-4.5);
\draw (0,3)--(0,1);

\node at (-0.4,2.1) {$\mathsmaller{h_U}$};
\node at (0.4,2.1) {$\mathsmaller{h_V}$};

\end{tikzpicture}
    \end{aligned}
    \end{equation}
    \normalsize
\end{exmp}

\begin{rem}
\label{rem:recovering graded root from bigraded root} 
Both the $U$-graded root and the $V$-graded root can  be obtained from the bigraded root as follows. 
For a given $i \in 2\Z + h_U(k)$, by picking $j \ll 0$ we have $\superlevelone{}^U_{i}( \mgraph,[k]) \subset \superlevelone{}^V_{j}( \mgraph,[k])$, so that 
\[
\superlevelone_{i,\ell}( \mgraph,[k]) = \superlevelone{}^U_{i}(\mgraph,[k])
\]
for all $\ell \leq j$. 
The vertices at height $i$ in $\root^U(\mgraph,[k]) = \root(\agraph,[k])$ can then be read off from  $\rootbi( \mgraph,[k])$ at bigrading $(i,j)$. 
Similarly, if $i'=i-2$, then by potentially decreasing $j$ further, edges between vertices at height $i$ and $i'$ of $\root(\agraph,[k])$ can be determined from edges in the $U$-direction of $\rootbi( \mgraph,[k])$ in bigradings $(i,j)$ and $(i',j)$. 
The graded root $\root(\agraph,[k+2\onev]) = \root^V(\mgraph, [k])$ can analogously be recovered from $\rootbi(\mgraph, [k])$. 
We refer to the above procedure of obtaining either the $U$-graded root or the $V$-graded root as \emph{collapsing} the bigraded root.

\end{rem}

\begin{defn}
\label{def:coordinates}
    The \emph{coordinate} of a node $\eta$ of $\rootbi(\mgraph, [k])$ in bigrading $(i,j)$ is the pair $(\eta_1, \eta_2)$ where $\eta_1$ (resp. $\eta_2$) is the unique node of $\root(\agraph,[k])$ (resp.\ of $\root(\agraph, [k+2\onev])$) in grading $i$ (resp. $j$) which corresponds to $\eta_1$ after collapsing to the $U$-graded (resp. $V$-graded) root. 
\end{defn}

We stress that the notion of coordinates is finer than the notion of bigrading, since in general, nodes in $\rootbi(\mgraph, [k])$ in the same bigrading may have different coordinates. 

\begin{rem}
    For simplicity, the examples provided in the present paper have the property that each $(i,j)$-superlevel set has at most one connected component, so that there is at most one node at each bigrading.
    Consequently, the bigraded root can be depicted in the plane. 
    Moreover, whenever the ambient manifold is $S^3$ (which is the case in our examples) all the nodes in a given bigrading have the same coordinate. 
\end{rem}

The following theorem is a special case of \cite[Theorem 1.2]{Niemi-Colvin}.

\begin{thm}{\cite{Niemi-Colvin}}\label{thm:bigraded root invariance}
Suppose $\mgraph$ and $\mgraphp$ are related by a Neumann move. Let $\beta : \spinc(\agraph) \to \spinc(\agraphp)$ be the corresponding bijection as defined in \eqref{eq:A0 closed}, \eqref{eq:B0 closed}, and \eqref{eq:C}, and let $[k] \in \spinc(\agraph)$ be a $\spinc$ structure. 
Then the bigraded roots $\rootbi(\mgraph,[k])$ and $\rootbi (\mgraphp,\beta([k]))$ are isomorphic.
\end{thm}

\subsection{Surgery formula for (bi)graded roots}
\label{sec:surgery formula for bigradd roots}
In this section, we explain how to determine the graded root of a surgered manifold from the bigraded root of a plumbed knot complement. 
We note that Ozsv{\'a}th-Stipsicz-Szab{\'o} \cite{OSS} established a surgery formula using the algebraic rather than superlevel set approach to knot lattice homology, the latter of which was introduced in \cite{Niemi-Colvin}. 
The results in this section may be viewed as a degree zero analogue of their surgery formula for the superlevel set approach to (knot) lattice homology. 
 
Let $\mgraph$ be a negative definite marked plumbing graph, with ambient manifold $\amfld$ and plumbed knot complement $\mmfld$. Pick a framing $m_0$ on $v_0$ such that the surgered plumbing graph $\sgraph$ is negative definite. 

For $L\in \Char(\sgraph)$, define the Alexander grading\footnote{Technically, Alexander grading is defined for $K \in \Char(\Gamma)$, so here we are abusing the terminology and call $a(L) = \frac{L(\Sigma)+\Sigma^2}{2}$ the Alexander grading of $L \in \Char(\sgraph)$.} of $L$ to be 
\begin{equation}
    a(L) := \frac{L(\Sigma)+\Sigma^2}{2},
\end{equation}
where $\Sigma$ is as in \eqref{eq:Sigma}. 
Given a $\spinc$ structure $\mathfrak{t} \in \spinc(\sgraph)$, set 
\[
\Char(\sgraph, \mathfrak{t}) := \{ L \in \Char(\sgraph) \mid [L] = \mathfrak{t} \},
\]
and let 
\[
\coset(\mathfrak{t}) := \left\{a(L) \;\vert\; L \in \Char(\sgraph, \mathfrak{t})\right\} \subset \mathbb{Q}.
\]
We observe that 
 \begin{equation}
 \label{eq:edges in a slice}
            a(L + 2\smatrix e_i) =
            \begin{cases}
             a(L) + \Sigma^2  &  \text{ if } i=0,\\
              a(L) & \text{ if } 1\leq i \leq s.
            \end{cases}
\end{equation}        
It follows that $\coset(\mathfrak{t}) = a(L) + \Sigma^2 \Z$
for any choice of $L \in \Char(\sgraph, \mathfrak{t})$. 
For $a\in \coset(\mathfrak{t})$, we set $\mathfrak{t}_a \in \spinc(\agraph)$ to be the $\spinc$ structure represented by $L\vert_{\agraph}$, for any  $L\in \Char(\sgraph, \mathfrak{t})$ with $a(L) = a$. 
Lemma \ref{lem:restricting spinc structure surgery} shows that $\mathfrak{t}_a$ is independent of the choice of $L$. Equation \eqref{eq:edges in a slice}, together with $\smatrix e_0 = (m_0, \lambda)$, implies that
\[
\mathfrak{t}_{a+\Sigma^2} = [2\onev]\cdot \mathfrak{t}_a, 
\]
where $\cdot$ denotes the homology action 
\eqref{eq:conj and H_1 for marked graph}. 

For $a, h \in \Q$, we  define 
\begin{align}
\begin{aligned}
\label{eq:h[a]}
\shift(a) &:= -\frac{1}{\Sigma^2}\left(a-\frac{\Sigma^2}{2}\right)^2 -\frac{1}{4}, \\
h[a] &:= h + \shift(a). 
\end{aligned}
\end{align}
The rest of this subsection is dedicated to the proof of the following proposition.

\begin{prop}
\label{prop:bigraded root algorithm}
The graded root of $(\sgraph,\mathfrak{t})$ is determined by the bigraded roots of $\mgraph$, according to the following algorithm:
\begin{enumerate}
    \item 
    Consider the graded graph $\bigsqcup\limits_{a \in \coset(\mathfrak{t})}\root(\agraph, \mathfrak{t}_a)\{-\shift(a)\}$, where $\{d\}$ denotes up upwards grading shift by $d$. 
    \item For each pair of nodes $\eta_1$ of $\root(\agraph, \mathfrak{t}_a)\{-\shift(a)\}$ and $\eta_2$ of $\root(\agraph, \mathfrak{t}_{a+\Sigma^2})\{-\shift(a+\Sigma^2)\}$ which are in the same grading, we identify $\eta_1$ and $\eta_2$ if there is a node in the bigraded root $\rootbi(\mgraph, \mathfrak{t}_{a})$ with  coordinate $(\eta_1, \eta_2)$ (see Definition \ref{def:coordinates}). 
    After all of these identifications, we remove multiple edges connecting the same pair of vertices. 
\end{enumerate}
\end{prop}

Examples \ref{ex:-1 surgery on trefoil bigraded root} and \ref{ex:-2 surgery on trefoil bigraded root} demonstrate the above algorithm. 
In Theorem \ref{thm:p surgery} we extend the surgery formula to weighted graded roots.

For a $\spinc$ structure $\mathfrak{t} \in \spinc(\sgraph)$, we partition $\superlevel_h(\sgraph,\mathfrak{t})$ according to Alexander grading,
\[
\superlevel_h(\sgraph,\mathfrak{t}) = \bigsqcup_{a \in \coset(\mathfrak{t})} \superlevel_{h}^a(\sgraph,\mathfrak{t}),
\]
where
\[
\superlevel_{h}^a(\sgraph,\mathfrak{t}) :=
\left\{
L \in \Char(\sgraph,\mathfrak{t}) \mid
\frac{L^2+(s+1)}{4} \geq h,\; a(L) = a
\right\}.
\]

Each $\superlevel_{h}^a(\sgraph,\mathfrak{t})$ forms the $0$-cells of a $1$-dimensional CW complex, denoted $\superlevelone{}^a_h(\sgraph,\mathfrak{t})$. Two vertices $L, L' \in \superlevelone{}^a_h(\sgraph,\mathfrak{t})$ are connected by an edge if $L-L' = \pm 2 \smatrix e_i$ for some $1\leq i \leq s$ (note  that \eqref{eq:edges in a slice} prevents edges in the $2 \smatrix e_0$ direction).
We would like to express each  $\superlevelone{}^a_h(\sgraph,\mathfrak{t})$ in terms of superlevel sets of $\Gamma$. 
This is accomplished in Proposition \ref{prop:slice is isomorphic to superlevel of ambient mfld} below. 
For this purpose, the following lemma expressing $L^2$ in terms of $L\vert_\Gamma$ is useful. Recall that $v^*\in H^2(X;\Z)$ denotes the image of the Poincar\'{e} dual of $v$ under the map $H^2(X, \partial X;\Z)\to H^2(X;\Z)$.

\begin{lem}
\label{lem:L^2 and K^2}
For any $L \in \Char(\sgraph,\mathfrak{t})$ with $ a(L)= a$, 
\[
L^2 = K^2 + \frac{1}{\Sigma^2}(\Sigma^2-2a)^2,
\]
where $K = L\vert_\Gamma$. 
\end{lem}
\begin{proof}
For any $1\leq i \leq s$, let $i'$ denote the unique vertex adjacent to $v_0$ that is in the same connected component as $v_i$ in $\Gamma$. 
As a consequence of equations \eqref{eq:inverse of e_0} and \eqref{eq:Sigma^2},
\[
(M_{v_0,m_0}^{-1})_{00} = \frac{1}{\Sigma^2}, \hskip1em (M_{v_0,m_0}^{-1})_{i0} = -\frac{1}{\Sigma^2}(M^{-1})_{ii'},
\]
\[
(M_{v_0,m_0}^{-1})_{ij} = (M^{-1})_{ij} + \frac{1}{\Sigma^2}(M^{-1})_{ii'}(M^{-1})_{jj'}
\]
for any $1\leq i,j\leq s$. Therefore, we can express $L^2$ in terms of $K$ in the following way: 
\begin{align*}
    L^2 &= \sum_{0\leq i,j\leq s}L(v_i)L(v_j) v_i^*\cdot v_j^*\\
    &= \sum_{1\leq i,j\leq s}K(v_i)K(v_j) v_i^*\cdot v_j^* + 2L(v_0)\sum_{1\leq i\leq s}K(v_i) v_0^* \cdot v_i^* + L(v_0)^2 (v_0^*)^2\\
    &= K^2 + \frac{1}{\Sigma^2}\left( \sum_{1\leq i\leq s}K(v_i)(M^{-1})_{ii'} \right)^2 - \frac{2L(v_0)}{\Sigma^2}\left( \sum_{1\leq i\leq s}K(v_i)(M^{-1})_{ii'} \right) + \frac{L(v_0)^2}{\Sigma^2}\\
    &= K^2 + \frac{1}{\Sigma^2}\left( \sum_{1\leq i\leq s}K(v_i)(M^{-1})_{ii'}-L(v_0) \right)^2,
\end{align*}
where $v_i^* \cdot v_j^* = (M_{v_0,m_0}^{-1})_{ij}$ denotes the dual intersection pairing. 
Plugging in 
\[
L(v_0) = 2a - \Sigma^2 - K(\Sigma-v_0) = 2a - \Sigma^2 + K M^{-1}  {\onev}
\]
and
\[
\sum_{1\leq i\leq s}K(v_i)(M^{-1})_{ii'} = K M^{-1}  {\onev}
\]
into the last expression, we get the desired equation:
\[
L^2 = K^2 + \frac{1}{\Sigma^2}(\Sigma^2-2a)^2. 
\]
\end{proof}

To relate $\superlevelone{}^a_h(\sgraph,\mathfrak{t})$, the slice of $\superlevelone_h(\sgraph, \mathfrak{t})$ with fixed Alexander grading $a$, to superlevel sets of $\Gamma$, we need to specify which $\spinc$ structure of $\Gamma$ we are considering. 
Recall that the pair $(\mathfrak{t},a)$ for $a\in \coset(\mathfrak{t})$ determines a  $\spinc$ structure on $\agraph$, denoted by $\mathfrak{t}_a$, which is  represented by $L\vert_\agraph \in \Char(\agraph)$ for a choice of $L\in \Char(\sgraph, \mathfrak{t})$ satisfying $a(L) = a$. 
The following demonstrates that $\mathfrak{t}_a$ is independent of the choice of $L$. 

\begin{lem}[{\cite[Section 5]{OSS}}]
    \label{lem:restricting spinc structure surgery}
    Let $L,L' \in \Char(\mgraph, \mathfrak{t})$ with $a(L) = a(L') = a$. 
    Then $L\vert_{\agraph}$ and  $L'\vert_{\agraph}$ represent the same $\spinc$ structure on the ambient manifold. 
\end{lem} 

\begin{proof}
   We have $L-L' = 2\smatrix x$ for some $x\in \Z^{s+1}$ and $(L-L')(\Sigma)=0$. 
   This implies that $x_0=0$, so that $L\vert_{\agraph}  - L'\vert_{\agraph}  = 2 \amatrix (x\vert_{\agraph} )$. 
\end{proof}

\begin{prop}
\label{prop:slice is isomorphic to superlevel of ambient mfld}
For any $a\in \coset(\mathfrak{t})$, there is an isomorphism of $1$-dimensional CW complexes
\begin{align*}
\superlevelone{}^a_h(\sgraph,\mathfrak{t}) \cong
\superlevelone_{h[a]}(\Gamma,\mathfrak{t}_a)
\end{align*}
given on $0$-cells by $L\mapsto L\vert_{\Gamma}$, where $h[a]$ is as defined in \eqref{eq:h[a]}. 
It follows that the connected components of $\superlevelone{}^a_h(\sgraph,\mathfrak{t})$  correspond to the nodes of the graded root $\root(\Gamma,\mathfrak{t}_a)$ at height $h[a]$.
\end{prop}
\begin{proof}
    Lemma \ref{lem:L^2 and K^2} implies that if $L\in \superlevel_{h}^a(\sgraph,\mathfrak{t})$, then $L\vert_\Gamma \in \superlevel_{h[a]}(\Gamma,\mathfrak{t}_a)$. Given $K\in \superlevel_{h[a]}(\Gamma,\mathfrak{t}_a)$, we will show that there is precisely one $L = (L_0, K) \in  \superlevel_{h}^a(\sgraph,\mathfrak{t})$. Indeed,
    \[
a= \frac{1}{2}(L(\Sigma) + \Sigma^2) = \frac{1}{2}( L_0 + m_0 + 2A(K))
    \]
    forces $L_0 = 2(a-A(K)) - m_0$. We need to verify that this choice of $L$ is in $\superlevel_{h}^a(\sgraph,\mathfrak{t})$. 
    
    First, let $L'\in \superlevel_{h}^a(\sgraph,\mathfrak{t})$ and set $K'= L'\vert_{\agraph} $. Lemma \ref{lem:restricting spinc structure surgery} implies that $K = K' + 2\amatrix x$ for some $x\in \Z^s$. We have 
    \begin{align*}
2(a- A(K)) = L'_0 - 2\hdual{x}\onev + m_0,
    \end{align*}
    which shows $L_0 \equiv m_0 \bmod 2$, so that $L \in \Char(\sgraph)$. Lemma \ref{lem:restricting spinc structure surgery}  gives $[L] = [L'] = \mathfrak{t}$, and Lemma \ref{lem:L^2 and K^2} shows $h_U(L) \geq h$, which completes the proof. 
\end{proof}

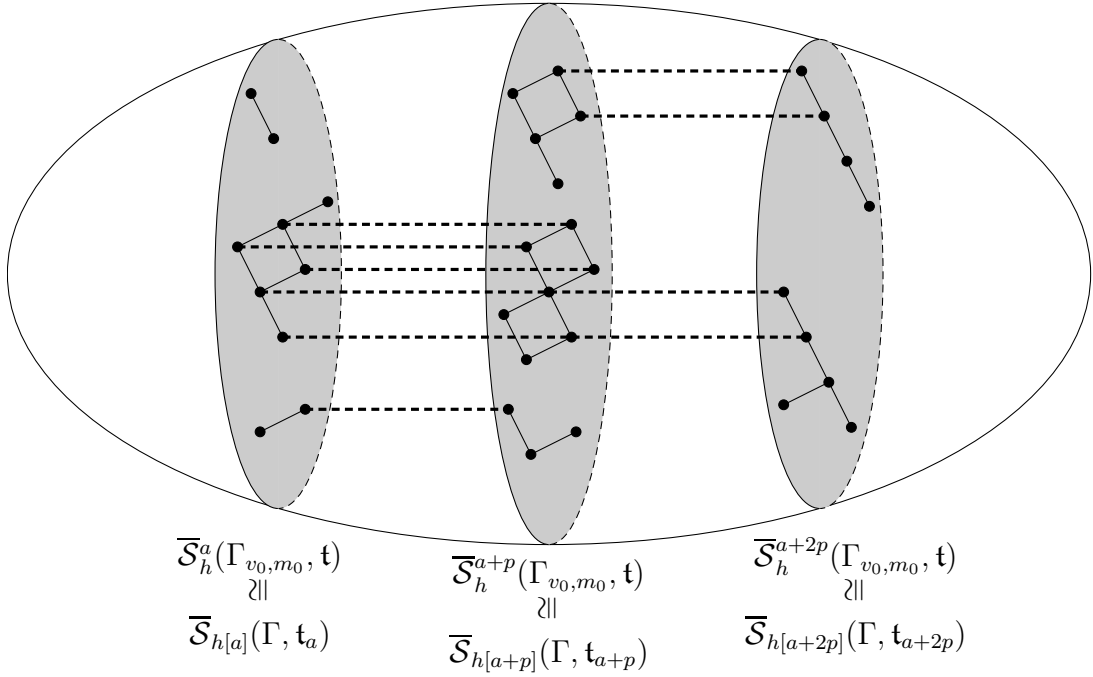
\begin{figure}
    \centering

\begin{tikzpicture}[scale=1.2]
\draw (0,0) ellipse (6 and 3);

\draw[densely dashed,fill = gray!40] (-3,2.6) arc (90:-90:.7 and 2.6);
\draw[fill = gray!40] (-3,2.6) arc (90:270:.7 and 2.6);
\draw[gray!40] (-3,2.6) -- (-3,-2.6);

\draw[densely dashed,fill = gray!40] (0,3) arc (90:-90:.7 and 3);
\draw[fill = gray!40] (0,3) arc (90:270:.7 and 3);
\draw[gray!40] (0,3) -- (0,-3);

\draw[densely dashed,fill = gray!40] (3,2.6) arc (90:-90:.7 and 2.6);
\draw[fill = gray!40] (3,2.6) arc (90:270:.7 and 2.6);
\draw[gray!40] (3,2.6) -- (3,-2.6);

\node[below] at (-3.2,-2.8) {$\overline{\mathcal{S}}{}^a_h(\Gamma_{v_0,m_0},\mathfrak{t})$}; 

\node[below] at (0,-3) {$\overline{\mathcal{S}}{}^{a+p}_{h}(\Gamma_{v_0,m_0},\mathfrak{t})$}; 

\node[below] at (3.4,-2.8) {$\overline{\mathcal{S}}{}^{a+2p}_{h}(\Gamma_{v_0,m_0},\mathfrak{t})$};


\node[rotate=90] at (-3.2,-3.5) {$\cong$}; 
\node at (-3.2,-4) {$\overline{\mathcal{S}}_{h[a]}(\Gamma,\mathfrak{t}_a)$}; 

\node[rotate=90] at (0,-3.7) {$\cong$}; 
\node at (0,-4.2) {$\overline{\mathcal{S}}_{h[a+p]}(\Gamma,\mathfrak{t}_{a+p})$};  

\node[rotate=90] at (3.4,-3.5) {$\cong$}; 
\node at (3.4,-4) {$\overline{\mathcal{S}}_{h[a+2p]}(\Gamma,\mathfrak{t}_{a+2p})$};

\coordinate (L1) at (-3.2,-.2);
\coordinate (L2) at (-3.2+.5,-.2+.25);
\coordinate (L3) at (-3.2+.25,-.2-.5);
\coordinate (L4) at (-3.2-.25,-.2+.5);
\coordinate (L5) at (-3.2-.25+.5,-.2+.5+.25);
\coordinate (L6) at (-3.2-.25+.5+.5,-.2+.5+.25+.25);

\draw[fill=black]  (L1) circle (.3ex);
\draw[fill=black] (L2) circle (.3ex);
\draw[fill=black]  (L3) circle (.3ex);
\draw[fill=black]  (L4) circle (.3ex);
\draw[fill=black]  (L5) circle (.3ex);
\draw[fill=black]  (L6) circle (.3ex);

\draw (L1) -- (L2); 
\draw (L1) -- (L3);
\draw (L1) -- (L4);
\draw (L2) -- (L5);
\draw (L4) -- (L6);


\coordinate (L7) at  (-3.2,-2+.25);
\coordinate (L8) at  (-3.2+.5,-2+.25+.25);

\draw[fill=black]  (L7) circle (.3ex);
\draw[fill=black]  (L8) circle (.3ex);
\draw (L7) -- (L8); 


\coordinate (L9) at (-3.3,2);
\coordinate (L10) at (-3.3+.25,2-.5);
\draw[fill=black] (L9) circle (.3ex);
\draw[fill=black] (L10) circle (.3ex);
\draw (L9) -- (L10);


\begin{scope}[shift={(3,0)}]
\coordinate (M1) at (-3.2+.2,-.2);
\coordinate (M2) at (-3.2+.2+.5,-.2+.25);
\coordinate (M3) at (-3.2+.2+.25,-.2-.5);
\coordinate (M4) at (-3.2+.2-.25,-.2+.5);
\coordinate (M5) at (-3.2+.2-.25+.5,-.2+.5+.25);
\coordinate (M6) at (-3.2+.2-.5,-.2-.25);
\coordinate (M71) at (-3.2+.2+.25-.5,-.2-.5-.25);

\draw[fill=black]  (M1) circle (.3ex);
\draw[fill=black] (M2) circle (.3ex);
\draw[fill=black]  (M3) circle (.3ex);
\draw[fill=black]  (M4) circle (.3ex);
\draw[fill=black]  (M5) circle (.3ex);
\draw[fill=black]  (M6) circle (.3ex);
\draw[fill=black]  (M71) circle (.3ex);

\draw (M1) -- (M2); 
\draw (M1) -- (M3);
\draw (M1) -- (M4);
\draw (M2) -- (M5);
\draw (M4) -- (M5);
\draw (M1) -- (M6);
\draw (M6) -- (M71);
\draw (M71)--(M3);

\coordinate (M7) at  (-3.2,-2);
\coordinate (M8) at  (-3.2+.5,-2+.25);
\coordinate (M81) at (-3.2-.25,-2+.5);

\draw[fill=black]  (M81) circle (.3ex);

\draw[fill=black]  (M7) circle (.3ex);
\draw[fill=black]  (M8) circle (.3ex);
\draw (M7) -- (M8); 
\draw (M7) -- (M81);

\coordinate (M9) at (-3.4,2);
\coordinate (M10) at (-3.4+.25,2-.5);
\coordinate (M11) at (-3.4+.25+.5,2-.5+.25);
\coordinate (M12) at (-3.4+.5,2+.25);
\coordinate (M13) at (-3.4+.25+.25,2-.5-.5);
\draw[fill=black] (M9) circle (.3ex);
\draw[fill=black] (M10) circle (.3ex);
\draw[fill=black] (M11) circle (.3ex);
\draw[fill=black] (M12) circle (.3ex);
\draw[fill=black] (M13) circle (.3ex);
\draw (M9) -- (M10);
\draw (M9) -- (M12);
\draw (M11) -- (M12);
\draw (M10) -- (M11);
\draw (M10) -- (M13);
\end{scope}


\begin{scope}[shift={(6,0)}]
\coordinate (R1) at (-3.4,-.5+.3);
\coordinate (R2) at (-3.4+.25,-.5+.3-.5);
\coordinate (R3) at (-3.4+.25+.25,-.5+.3-.5-.5);
\coordinate (R4) at (-3.4+.25+.25+.25,-.5+.3-.5-.5-.5);
\coordinate (R5) at (-3.4+.25+.25-.5,-.5+.3-.5-.5-.25);
\draw[fill=black] (R1) circle (.3ex);
\draw[fill=black] (R2) circle (.3ex);
\draw[fill=black] (R3) circle (.3ex);
\draw[fill=black] (R4) circle (.3ex);
\draw[fill=black] (R5) circle (.3ex);
\draw (R1) -- (R4);
\draw (R3) -- (R5);


\coordinate (R9) at (-3.2,2.25);
\coordinate (R10) at (-3.2+.25,2.25-.5);
\coordinate (R11) at (-3.2+.25+.25,2.25-.5-.5);
\coordinate (R12) at (-3.2+.25+.25+.25,2.25-.5-.5-.5);
\draw[fill=black] (R9) circle (.3ex);
\draw[fill=black] (R10) circle (.3ex);
\draw[fill=black] (R11) circle (.3ex);
\draw[fill=black] (R12) circle (.3ex);
\draw (R9) -- (R12);
\end{scope}

\draw[densely dashed,very thick] (L1) -- (M1);

\draw[densely dashed,very thick] (L8) -- (M81);

\draw[densely dashed,very thick] (L3) -- (M3);

\draw[densely dashed,very thick] (M12) -- (R9);
\draw[densely dashed,very thick] (M11) -- (R10);
\draw[densely dashed,very thick] (M3) -- (R2);

\draw[densely dashed,very thick] (L2) -- (M2);
\draw[densely dashed,very thick] (M1) -- (R1);

\draw[densely dashed,very thick] (L4) -- (M4);
\draw[densely dashed,very thick] (L5) -- (M5);
\end{tikzpicture}
    \caption{A schematic depiction of slicing the superlevel sets of $\sgraph$ by the Alexander grading, where we have set $p = \Sigma^2$. 
    The dashed horizontal lines represent edges in the $v_0$ direction. Note that $\mathfrak{t}_{a}$ and $\mathfrak{t}_{a+rp}$, for $r\in \Z$, may be different $\spinc$ structures on $\agraph$.} 
    \label{fig:superlevel set surgery}
\end{figure}

\begin{proof}[Proof of Proposition \ref{prop:bigraded root algorithm}]
With Proposition \ref{prop:slice is isomorphic to superlevel of ambient mfld} at hand, together with \eqref{eq:edges in a slice}, we see that in order to determine the connected components of the whole superlevel set $\superlevelone_{h}(\sgraph,\mathfrak{t})$ in terms of superlevel sets of the ambient plumbing, we need to know when there are edges in the $2v_0^*\vert_{\Gamma}=2\onev$ direction.
Equivalently, we need to know when 
\[
K \in \superlevel_{h[a]}(\Gamma,\mathfrak{t}_a)
\]
representing a connected component $C_1$
and
\[
K' \in \superlevel_{h\left[a+\Sigma^2\right]}(\Gamma,\mathfrak{t}_{a+\Sigma^2})
\]
representing a connected component $C_2$
are connected by an edge in the $2v_0^*\vert_{\Gamma}$ direction, so that
\[
K' = K + 2v_0^*\vert_{\Gamma} = K + 2\onev.
\]
See Figure \ref{fig:superlevel set surgery} for a schematic. 

This information is captured exactly by the bigraded root. 
More precisely, such $K$ must satisfy
$h_U(K) \geq h[a]$ and $h_V(K) \geq h[a+\Sigma^2]$, and therefore represents a connected component of $\superlevelone_{h_1, h_2}(\mgraph, \mathfrak{t})$ where $h_1 = h[a]$, $h_2=h[a+\Sigma^2]$. 
Such connected components can be read off from the bigraded root. 
That is, we can check if there is any node of the bigraded root representing a connected component of $\superlevelone_{h_1, h_2}(\mgraph, \mathfrak{t})$ for which its images under inclusion to $\superlevelone{}^U_{h_1}(\mgraph,\mathfrak{t})$ and $\superlevelone{}^V_{h_2}(\mgraph,\mathfrak{t})$ lie in $C_1$ and $C_2 - 2v_0^*\vert_\Gamma$, respectively. 
Thus the bigraded root encodes precisely the information needed to perform surgery and recover the full graded root according to the algorithm in Proposition \ref{prop:bigraded root algorithm}. 
\end{proof}

Below, we illustrate how the surgery of (bi)graded roots work in practice through examples. First, we set $\sf = \hdual{\onev}\amatrix^{-1}\onev\in \Q$. \htarget{sf} If $\knot$ is nullhomologous then $\sf \in \Z$, and moreover from \eqref{eq:inverse of e_0} we see that $\sf$ is the Seifert framing of $\knot$ since $M_{v_0, \sf}$ is not invertible.
We also set 
\[
p := \Sigma^2 = m_0 - \hdual{\onev}\amatrix^{-1}\onev = m_0 - \sf 
\in \Q.
\]
If $\knot$ is nullhomologous, then performing $p$ surgery on $\knot \subset \amfld$ yields precisely $\smfld$. 

\begin{exmp}[$-1$ surgery]
\label{ex:-1 surgery on trefoil bigraded root}
Let $\knot$ be an algebraic knot in $S^3$.
We will describe how to obtain the graded root for the $-1$ surgery (i.e. $\Sigma^2 = -1$) of $\knot$, from the bigraded root of $\knot$. 

The ambient $3$-manifold is $S^3$, and its only graded root is given by an infinite linear graph with one node at each non-positive even degree. 
Take $\mathbb{Z}$ copies of the graded root, and arrange it so that the height of the $a$-th copy is shifted by
\[
h-h[a] = -\shift(a) = \frac{1}{4} + \frac{1}{\Sigma^2}\left(a-\frac{\Sigma^2}{2}\right)^2 = -a(a+1). 
\]
That is, for the $a$-th copy of the graded root, nodes which would normally be called of height $h[a]$ are now placed at height $h$. 
See Figure \ref{fig:-1_surgery}. 
\begin{figure}[h!]
    \centering
    \begin{tikzpicture}[scale=0.5]
\node at (0,0) {$\bullet$};
\node at (0,-2) {$\bullet$};
\node at (0,-4) {$\bullet$};
\node at (0,-6) {$\bullet$};
\node at (0,-8) {$\bullet$};
\node at (0,-10) {$\bullet$};
\node at (0,-12) {$\bullet$};

\node at (-2,0) {$\bullet$};
\node at (-2,-2) {$\bullet$};
\node at (-2,-4) {$\bullet$};
\node at (-2,-6) {$\bullet$};
\node at (-2,-8) {$\bullet$};
\node at (-2,-10) {$\bullet$};
\node at (-2,-12) {$\bullet$};

\node at (2,-2) {$\bullet$};
\node at (2,-4) {$\bullet$};
\node at (2,-6) {$\bullet$};
\node at (2,-8) {$\bullet$};
\node at (2,-10) {$\bullet$};
\node at (2,-12) {$\bullet$};

\node at (-4,-2) {$\bullet$};
\node at (-4,-4) {$\bullet$};
\node at (-4,-6) {$\bullet$};
\node at (-4,-8) {$\bullet$};
\node at (-4,-10) {$\bullet$};
\node at (-4,-12) {$\bullet$};

\node at (4,-6) {$\bullet$};
\node at (4,-8) {$\bullet$};
\node at (4,-10) {$\bullet$};
\node at (4,-12) {$\bullet$};

\node at (-6,-6) {$\bullet$};
\node at (-6,-8) {$\bullet$};
\node at (-6,-10) {$\bullet$};
\node at (-6,-12) {$\bullet$};

\node at (6,-12) {$\bullet$};

\node at (-8,-12) {$\bullet$};

\draw (0,0)--(0,-2)--(0,-4)--(0,-6)--(0,-8)--(0,-10)--(0,-12)--(0,-13);
\draw (-2,0)--(-2,-2)--(-2,-4)--(-2,-6)--(-2,-8)--(-2,-10)--(-2,-12)--(-2,-13);
\draw (2,-2)--(2,-4)--(2,-6)--(2,-8)--(2,-10)--(2,-12)--(2,-13);
\draw (-4,-2)--(-4,-4)--(-4,-6)--(-4,-8)--(-4,-10)--(-4,-12)--(-4,-13);
\draw (4,-6)--(4,-8)--(4,-10)--(4,-12)--(4,-13);
\draw (-6,-6)--(-6,-8)--(-6,-10)--(-6,-12)--(-6,-13);
\draw (6,-12)--(6,-13);
\draw (-8,-12)--(-8,-13);

\node at (0,2) {$0$};
\node at (-2,2) {$-1$};
\node at (2,2) {$1$};
\node at (-4,2) {$-2$};
\node at (4,2) {$2$};
\node at (-6,2) {$-3$};
\node at (6,2) {$3$};
\node at (-8,2) {$-4$};

\node at (-10,0) {$0$};
\node at (-10,-2) {$-2$};
\node at (-10,-4) {$-4$};
\node at (-10,-6) {$-6$};
\node at (-10,-8) {$-8$};
\node at (-10,-10) {$-10$};
\node at (-10,-12) {$-12$};

\draw (-11,1)--(7,1);
\draw (-9,3)--(-9,-13);
\draw (-9,1)--(-11,3);

\node at (-9.6,2.4) {$a$};
\node at (-10.4,1.6) {$h$};

\end{tikzpicture}
    \caption{$\mathbb{Z}$ copies of the graded root arranged for $-1$ surgery}
    \label{fig:-1_surgery}
\end{figure}
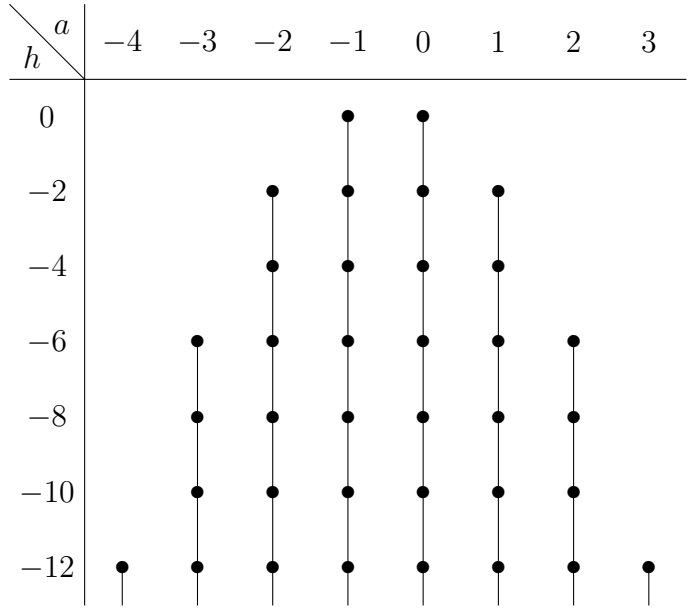

The graded root of the surgered manifold will be obtained by collapsing the horizontal $a$-axis. 
That is, for each height $h$, we will replace the nodes at height $h$ with the actual nodes corresponding to the connected components for the surgered manifold. 
For this, we need to know when a node at $(a,h)$ is connected by an edge in $2v_0^*$ direction to a node at $(a-1,h)$. 

To illustrate this, we now specialize to the (right-handed) trefoil knot, whose bigraded root is given in Figure \ref{fig:trefoil_bigraded_root}. 
\begin{figure}
    \centering
    \begin{tikzpicture}[scale=0.7]
\begin{scope}[shift={(-12,0)}]
    \coordinate (1) at (-2,2) ;
\coordinate (2) at (0,0) ;
\coordinate (3) at (2,2) ;
\node at (1) {$\bullet$};
\node at (2) {$\bullet$};
\node at (3) {$\bullet$};

\draw (0,-2.828) circle (3pt);

\draw (1) -- (2);
\draw (2) -- (3);
\draw (0,0) -- (0,-2.728);

\node[above] at (1) {$-2$};
\node[anchor=north east] at (2) {$-1$};
\node[above] at (3) {$-3$};

\end{scope}


\node at (-1,-1) {$\bullet$};
\node at (1,-1) {$\bullet$};
\node at (-2,-2) {$\bullet$};
\node at (0,-2) {$\bullet$};
\node at (2,-2) {$\bullet$};
\node at (-3,-3) {$\bullet$};
\node at (-1,-3) {$\bullet$};
\node at (1,-3) {$\bullet$};
\node at (3,-3) {$\bullet$};
\node at (-4,-4) {$\bullet$};
\node at (-2,-4) {$\bullet$};
\node at (0,-4) {$\bullet$};
\node at (2,-4) {$\bullet$};
\node at (4,-4) {$\bullet$};

\draw (-2,-2)--(-1,-1)--(0,-2)--(1,-1)--(2,-2);
\draw (-3,-3)--(-2,-2)--(-1,-3)--(0,-2)--(1,-3)--(2,-2)--(3,-3);
\draw (-4,-4)--(-3,-3)--(-2,-4)--(-1,-3)--(0,-4)--(1,-3)--(2,-4)--(3,-3)--(4,-4);


\node at (-5, -5) {\reflectbox{$\vdots$}};
\node at (5, -5) {$\vdots$};
\node at (0, -5) {$\vdots$};

\node at (1, 1) {$0$};
\node at (2, 0) {$-2$};
\node at (3, -1) {$-4$};
\node at (4, -2) {$-6$};
\node at (5, -3) {$-8$};

\node at (-1, 1) {$0$};
\node at (-2, 0) {$-2$};
\node at (-3, -1) {$-4$};
\node at (-4, -2) {$-6$};
\node at (-5, -3) {$-8$};

\draw (-1,2)--(5.5,-4.5);
\draw (1,2)--(-5.5,-4.5);
\draw (0,3)--(0,1);

\node at (-0.4,2.1) {$h_U$};
\node at (0.4,2.1) {$h_V$};

\end{tikzpicture}
    \caption{Left: a plumbing diagram for the trefoil knot. Right: the bigraded root for the trefoil knot. }
    \label{fig:trefoil_bigraded_root}
\end{figure}
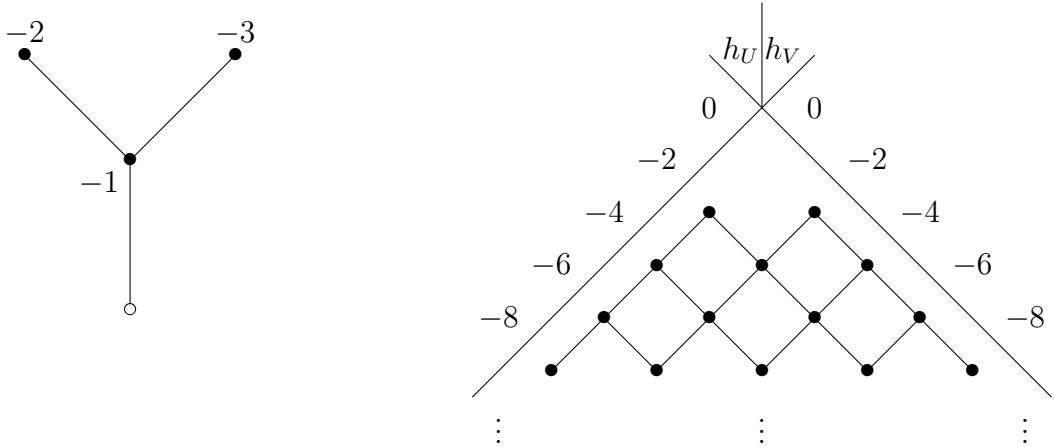
As discussed in Definition \ref{def:coordinates}, each node of a bigraded root carries a pair of coordinates valued in the graded root of the ambient $3$-manifold, which in this case, is the same as a pair of non-positive even numbers. 
As drawn in Figure \ref{fig:trefoil_bigraded_root}, the bigraded root of the trefoil has a single node at every coordinate $(h_U,h_V) \in (2\mathbb{Z}_{\leq 0})^{2}$, except for $(0,0)$, where it doesn't have any node. 

From our earlier discussion, the nodes at $(a,h)$ and $(a-1,h)$ in Figure \ref{fig:-1_surgery} are connected by an edge if and only if there is a node in the bigraded root at coordinate $(h[a],h[a-1])$.\footnote{When the ambient $3$-manifold is not $S^3$, we need to check if there is a node in the bigraded root at the coordinate specified by a pair of nodes of the graded root.} 
Since the bigraded root of the trefoil knot has nodes at every coordinate except at $(0,0)$, this means that the only missing edge is between the nodes at $(a,h) = (0,0)$ and $(a,h) = (-1,0)$. 
\begin{figure}
    \centering
    \begin{tikzpicture}[scale=0.5]
\node at (0,0) {$\bullet$};
\node at (0,-2) {$\bullet$};
\node at (0,-4) {$\bullet$};
\node at (0,-6) {$\bullet$};
\node at (0,-8) {$\bullet$};
\node at (0,-10) {$\bullet$};
\node at (0,-12) {$\bullet$};

\node at (-2,0) {$\bullet$};
\node at (-2,-2) {$\bullet$};
\node at (-2,-4) {$\bullet$};
\node at (-2,-6) {$\bullet$};
\node at (-2,-8) {$\bullet$};
\node at (-2,-10) {$\bullet$};
\node at (-2,-12) {$\bullet$};

\node at (2,-2) {$\bullet$};
\node at (2,-4) {$\bullet$};
\node at (2,-6) {$\bullet$};
\node at (2,-8) {$\bullet$};
\node at (2,-10) {$\bullet$};
\node at (2,-12) {$\bullet$};

\node at (-4,-2) {$\bullet$};
\node at (-4,-4) {$\bullet$};
\node at (-4,-6) {$\bullet$};
\node at (-4,-8) {$\bullet$};
\node at (-4,-10) {$\bullet$};
\node at (-4,-12) {$\bullet$};

\node at (4,-6) {$\bullet$};
\node at (4,-8) {$\bullet$};
\node at (4,-10) {$\bullet$};
\node at (4,-12) {$\bullet$};

\node at (-6,-6) {$\bullet$};
\node at (-6,-8) {$\bullet$};
\node at (-6,-10) {$\bullet$};
\node at (-6,-12) {$\bullet$};

\node at (6,-12) {$\bullet$};

\node at (-8,-12) {$\bullet$};

\draw (0,0)--(0,-2)--(0,-4)--(0,-6)--(0,-8)--(0,-10)--(0,-12)--(0,-13);
\draw (-2,0)--(-2,-2)--(-2,-4)--(-2,-6)--(-2,-8)--(-2,-10)--(-2,-12)--(-2,-13);
\draw (2,-2)--(2,-4)--(2,-6)--(2,-8)--(2,-10)--(2,-12)--(2,-13);
\draw (-4,-2)--(-4,-4)--(-4,-6)--(-4,-8)--(-4,-10)--(-4,-12)--(-4,-13);
\draw (4,-6)--(4,-8)--(4,-10)--(4,-12)--(4,-13);
\draw (-6,-6)--(-6,-8)--(-6,-10)--(-6,-12)--(-6,-13);
\draw (6,-12)--(6,-13);
\draw (-8,-12)--(-8,-13);

\node at (0,2) {$0$};
\node at (-2,2) {$-1$};
\node at (2,2) {$1$};
\node at (-4,2) {$-2$};
\node at (4,2) {$2$};
\node at (-6,2) {$-3$};
\node at (6,2) {$3$};
\node at (-8,2) {$-4$};

\node at (-10,0) {$0$};
\node at (-10,-2) {$-2$};
\node at (-10,-4) {$-4$};
\node at (-10,-6) {$-6$};
\node at (-10,-8) {$-8$};
\node at (-10,-10) {$-10$};
\node at (-10,-12) {$-12$};

\draw (-11,1)--(7,1);
\draw (-9,3)--(-9,-13);
\draw (-9,1)--(-11,3);

\node at (-9.6,2.4) {$a$};
\node at (-10.4,1.6) {$h$};

\draw [dotted] (-4,-2)--(2,-2);
\draw [dotted] (-4,-4)--(2,-4);
\draw [dotted] (-6,-6)--(4,-6);
\draw [dotted] (-6,-8)--(4,-8);
\draw [dotted] (-6,-10)--(4,-10);
\draw [dotted] (-8,-12)--(6,-12);

\draw[ultra thick, ->] (8,-6)--(10,-6);
\end{tikzpicture}

\begin{tikzpicture}[scale=0.5]
\node at (-3,0) {$0$};
\node at (-3,-2) {$-2$};
\node at (-3,-4) {$-4$};
\node at (-3,-6) {$-6$};
\node at (-3,-8) {$-8$};
\node at (-3,-10) {$-10$};
\node at (-3,-12) {$-12$};

\node at (-3,2) {$h$};

\node at (1,0) {$\bullet$};
\node at (-1,0) {$\bullet$};
\node at (0,-2) {$\bullet$};
\node at (0,-4) {$\bullet$};
\node at (0,-6) {$\bullet$};
\node at (0,-8) {$\bullet$};
\node at (0,-10) {$\bullet$};
\node at (0,-12) {$\bullet$};
\draw (1,0)--(0,-2)--(0,-4)--(0,-6)--(0,-8)--(0,-10)--(0,-12)--(0,-13);
\draw (-1,0)--(0,-2);

\end{tikzpicture}
    \caption{$-1$ surgery on the trefoil. 
    The dashed horizontal lines indicate which vertices are identified in step (2) of Proposition \ref{prop:bigraded root algorithm}.}
    \label{fig:trefoil_-1_surgery}
\end{figure}
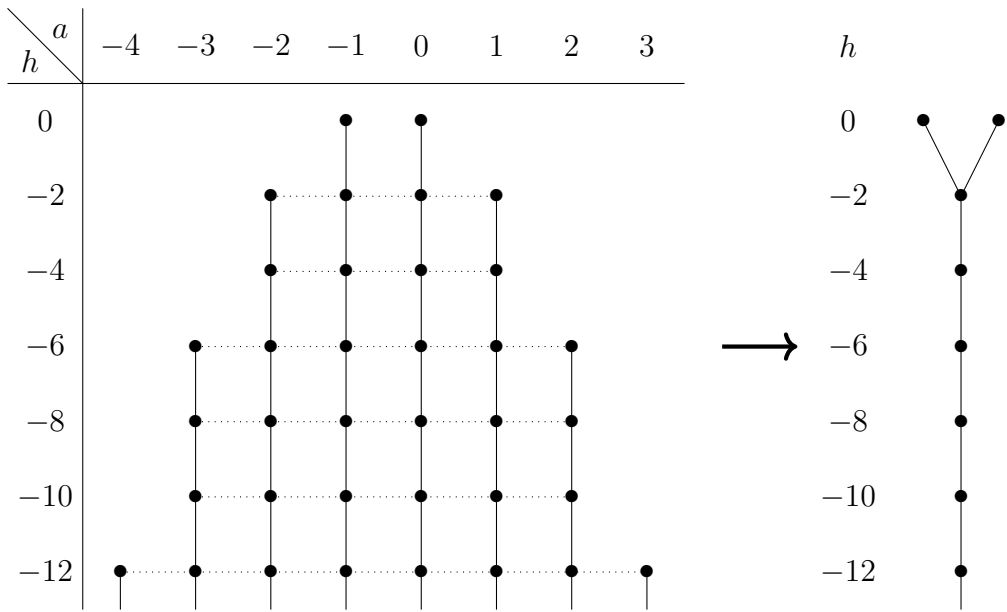
See the left part of Figure \ref{fig:trefoil_-1_surgery}. 
Collapsing the horizontal $a$-coordinate by taking the connected components, we obtain the graded root of $S^{3}_{-1}(\mathbf{3}_1) = \Sigma(2,3,7)$ (the right part of Figure \ref{fig:trefoil_-1_surgery}). 

For an arbitrary algebraic knot in $S^3$, the graded root for the $-1$ surgery can be obtained in the same way; 
the dictionary between the nodes of the bigraded root and edges connecting nodes from different $a$'s in the $-1$ surgery is summarized in Figure \ref{fig:-1_surgery_dictionary}. 
\begin{figure}
    \centering
    \begin{tikzpicture}[scale=0.5]

\fill[fill=yellow, fill opacity=0.3](-2.2,-0.2) rectangle (0.2,0.2);
\fill[fill=yellow, fill opacity=0.3](-4.2,-2.2) rectangle (2.2,-1.8);
\fill[fill=yellow, fill opacity=0.3](-4.2,-4.2) rectangle (2.2,-3.8);
\fill[fill=yellow, fill opacity=0.3](-6.2,-6.2) rectangle (4.2,-5.8);
\fill[fill=yellow, fill opacity=0.3](-6.2,-8.2) rectangle (4.2,-7.8);
\fill[fill=yellow, fill opacity=0.3](-6.2,-10.2) rectangle (4.2,-9.8);
\fill[fill=yellow, fill opacity=0.3](-8.2,-12.2) rectangle (6.2,-11.8);

\node at (0,0) {$\bullet$};
\node at (0,-2) {$\bullet$};
\node at (0,-4) {$\bullet$};
\node at (0,-6) {$\bullet$};
\node at (0,-8) {$\bullet$};
\node at (0,-10) {$\bullet$};
\node at (0,-12) {$\bullet$};

\node at (-2,0) {$\bullet$};
\node at (-2,-2) {$\bullet$};
\node at (-2,-4) {$\bullet$};
\node at (-2,-6) {$\bullet$};
\node at (-2,-8) {$\bullet$};
\node at (-2,-10) {$\bullet$};
\node at (-2,-12) {$\bullet$};

\node at (2,-2) {$\bullet$};
\node at (2,-4) {$\bullet$};
\node at (2,-6) {$\bullet$};
\node at (2,-8) {$\bullet$};
\node at (2,-10) {$\bullet$};
\node at (2,-12) {$\bullet$};

\node at (-4,-2) {$\bullet$};
\node at (-4,-4) {$\bullet$};
\node at (-4,-6) {$\bullet$};
\node at (-4,-8) {$\bullet$};
\node at (-4,-10) {$\bullet$};
\node at (-4,-12) {$\bullet$};

\node at (4,-6) {$\bullet$};
\node at (4,-8) {$\bullet$};
\node at (4,-10) {$\bullet$};
\node at (4,-12) {$\bullet$};

\node at (-6,-6) {$\bullet$};
\node at (-6,-8) {$\bullet$};
\node at (-6,-10) {$\bullet$};
\node at (-6,-12) {$\bullet$};

\node at (6,-12) {$\bullet$};

\node at (-8,-12) {$\bullet$};

\draw (0,0)--(0,-2)--(0,-4)--(0,-6)--(0,-8)--(0,-10)--(0,-12)--(0,-13);
\draw (-2,0)--(-2,-2)--(-2,-4)--(-2,-6)--(-2,-8)--(-2,-10)--(-2,-12)--(-2,-13);
\draw (2,-2)--(2,-4)--(2,-6)--(2,-8)--(2,-10)--(2,-12)--(2,-13);
\draw (-4,-2)--(-4,-4)--(-4,-6)--(-4,-8)--(-4,-10)--(-4,-12)--(-4,-13);
\draw (4,-6)--(4,-8)--(4,-10)--(4,-12)--(4,-13);
\draw (-6,-6)--(-6,-8)--(-6,-10)--(-6,-12)--(-6,-13);
\draw (6,-12)--(6,-13);
\draw (-8,-12)--(-8,-13);

\node at (0,2) {$0$};
\node at (-2,2) {$-1$};
\node at (2,2) {$1$};
\node at (-4,2) {$-2$};
\node at (4,2) {$2$};
\node at (-6,2) {$-3$};
\node at (6,2) {$3$};
\node at (-8,2) {$-4$};

\node at (-10,0) {$0$};
\node at (-10,-2) {$-2$};
\node at (-10,-4) {$-4$};
\node at (-10,-6) {$-6$};
\node at (-10,-8) {$-8$};
\node at (-10,-10) {$-10$};
\node at (-10,-12) {$-12$};

\draw (-11,1)--(7,1);
\draw (-9,3)--(-9,-13);
\draw (-9,1)--(-11,3);

\node at (-9.6,2.4) {$a$};
\node at (-10.4,1.6) {$h$};

\node at (-1,0) {\textcolor{green}{$a$}};

\node at (-3,-2) {\textcolor{green}{$b$}};
\node at (-1,-2) {\textcolor{green}{$c$}};
\node at (1,-2) {\textcolor{green}{$d$}};

\node at (-3,-4) {\textcolor{green}{$e$}};
\node at (-1,-4) {\textcolor{green}{$f$}};
\node at (1,-4) {\textcolor{green}{$g$}};

\node at (-5,-6) {\textcolor{green}{$h$}};
\node at (-3,-6) {\textcolor{green}{$i$}};
\node at (-1,-6) {\textcolor{green}{$j$}};
\node at (1,-6) {\textcolor{green}{$k$}};
\node at (3,-6) {\textcolor{green}{$l$}};

\node at (-5,-8) {\textcolor{green}{$m$}};
\node at (-3,-8) {\textcolor{green}{$n$}};
\node at (-1,-8) {\textcolor{green}{$o$}};
\node at (1,-8) {\textcolor{green}{$p$}};
\node at (3,-8) {\textcolor{green}{$q$}};

\node at (-5,-10) {\textcolor{green}{$r$}};
\node at (-3,-10) {\textcolor{green}{$s$}};
\node at (-1,-10) {\textcolor{green}{$t$}};
\node at (1,-10) {\textcolor{green}{$u$}};
\node at (3,-10) {\textcolor{green}{$v$}};

\node at (-7,-12) {\textcolor{green}{$w$}};
\node at (-5,-12) {\textcolor{green}{$x$}};
\node at (-3,-12) {\textcolor{green}{$y$}};
\node at (-1,-12) {\textcolor{green}{$z$}};
\node at (1,-12) {\textcolor{green}{$\alpha$}};
\node at (3,-12) {\textcolor{green}{$\beta$}};
\node at (5,-12) {\textcolor{green}{$\gamma$}};

\draw[ultra thick, <->] (7.5,-6)--(9,-6);


\begin{scope}[scale=1.3, shift={(13,0)}]
\draw[draw=yellow, draw opacity=0.3, ultra thick](0,0)--(0,0);
\draw[draw=yellow, draw opacity=0.3, ultra thick](-1,-1)--(0,-2)--(1,-1);
\draw[draw=yellow, draw opacity=0.3, ultra thick](-1,-3)--(0,-4)--(1,-3);
\draw[draw=yellow, draw opacity=0.3, ultra thick](-2,-2)--(-1,-5)--(0,-6)--(1,-5)--(2,-2);
\draw[draw=yellow, draw opacity=0.3, ultra thick](-2,-4)--(-1,-7)--(0,-8)--(1,-7)--(2,-4);
\draw[draw=yellow, draw opacity=0.3, ultra thick](-2,-6)--(-1,-9)--(0,-10)--(1,-9)--(2,-6);
\draw[draw=yellow, draw opacity=0.3, ultra thick](-3,-3)--(-2,-8)--(-1,-11)--(0,-12)--(1,-11)--(2,-8)--(3,-3);

\node at (1, 1) {$0$};
\node at (2, 0) {$-2$};
\node at (3, -1) {$-4$};
\node at (4, -2) {$-6$};
\node at (5, -3) {$-8$};

\node at (-1, 1) {$0$};
\node at (-2, 0) {$-2$};
\node at (-3, -1) {$-4$};
\node at (-4, -2) {$-6$};
\node at (-5, -3) {$-8$};

\draw (-1,2)--(5.5,-4.5);
\draw (1,2)--(-5.5,-4.5);
\draw (0,3)--(0,1);

\node at (-0.4,2.1) {$h_U$};
\node at (0.4,2.1) {$h_V$};

\node at (0,0) {\textcolor{green}{$a$}};

\node at (-1,-1) {\textcolor{green}{$b$}};
\node at (0,-2) {\textcolor{green}{$c$}};
\node at (1,-1) {\textcolor{green}{$d$}};

\node at (-1,-3) {\textcolor{green}{$e$}};
\node at (0,-4) {\textcolor{green}{$f$}};
\node at (1,-3) {\textcolor{green}{$g$}};

\node at (-2,-2) {\textcolor{green}{$h$}};
\node at (-1,-5) {\textcolor{green}{$i$}};
\node at (0,-6) {\textcolor{green}{$j$}};
\node at (1,-5) {\textcolor{green}{$k$}};
\node at (2,-2) {\textcolor{green}{$l$}};

\node at (-2,-4) {\textcolor{green}{$m$}};
\node at (-1,-7) {\textcolor{green}{$n$}};
\node at (0,-8) {\textcolor{green}{$o$}};
\node at (1,-7) {\textcolor{green}{$p$}};
\node at (2,-4) {\textcolor{green}{$q$}};

\node at (-2,-6) {\textcolor{green}{$r$}};
\node at (-1,-9) {\textcolor{green}{$s$}};
\node at (0,-10) {\textcolor{green}{$t$}};
\node at (1,-9) {\textcolor{green}{$u$}};
\node at (2,-6) {\textcolor{green}{$v$}};

\node at (-3,-3) {\textcolor{green}{$w$}};
\node at (-2,-8) {\textcolor{green}{$x$}};
\node at (-1,-11) {\textcolor{green}{$y$}};
\node at (0,-12) {\textcolor{green}{$z$}};
\node at (1,-11) {\textcolor{green}{$\alpha$}};
\node at (2,-8) {\textcolor{green}{$\beta$}};
\node at (3,-3) {\textcolor{green}{$\gamma$}};

\end{scope}
\end{tikzpicture}
    \caption{$-1$ surgery dictionary}
    \label{fig:-1_surgery_dictionary}
\end{figure}
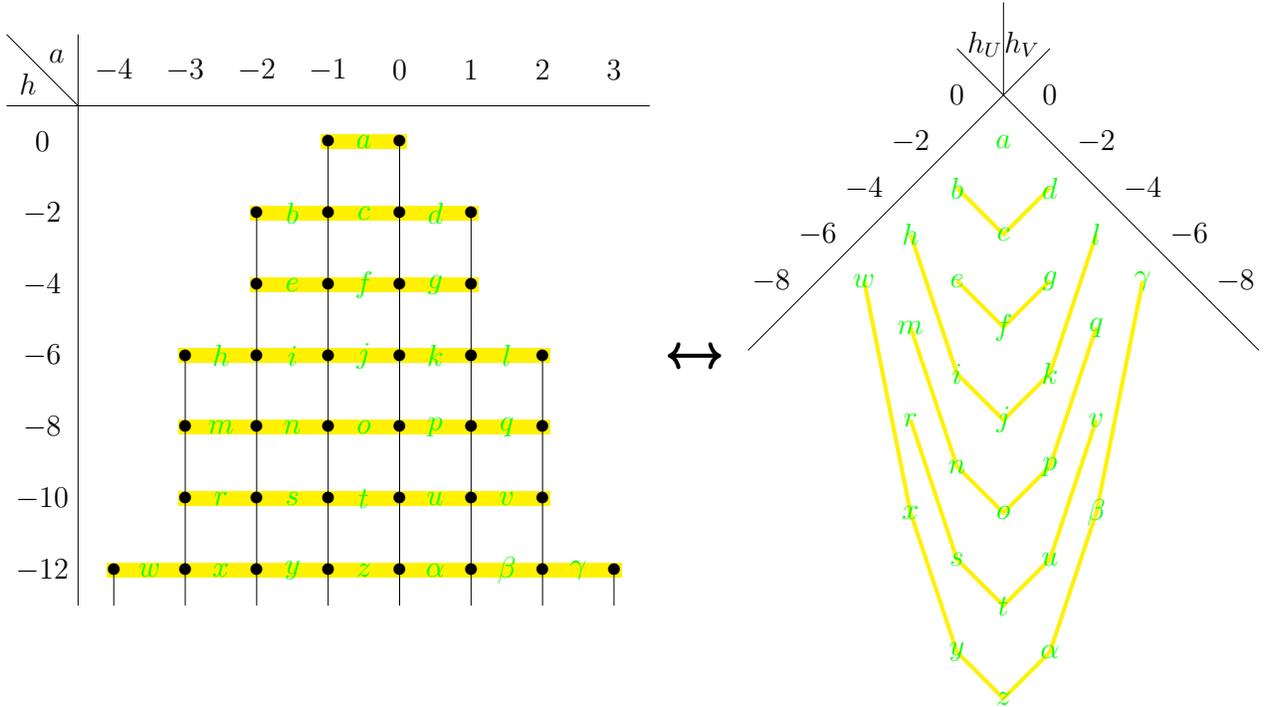

\end{exmp}

\begin{exmp}[$-2$ surgery]
\label{ex:-2 surgery on trefoil bigraded root}
We illustrate the surgery of (bi)graded roots through one more example: the $-2$ surgery (i.e. $\Sigma^2 = -2$). 
As in the previous example, let $\knot$ be an algebraic knot in $S^3$. We can proceed in the same way. 

Take $\mathbb{Z}$ copies of the graded root, but now $a$ even and $a$ odd represent two different spin$^c$ structures, $\mathfrak{t}_0$ and $\mathfrak{t}_1$, of the surgered manifold. 
For either of the spin$^c$ structures, arrange the copies of the graded root so that the height of the $a$-th copy is shifted by
\[
h-h[a] = -\shift(a) = \frac{1}{4} + \frac{1}{\Sigma^2}\left(a-\frac{\Sigma^2}{2}\right)^2 = -\frac{1}{4}-\frac{1}{2}a(a+2). 
\]
Then, to compute the graded root of the surgered manifold, we need to determine if the nodes at coordinates $(a,h)$ and $(a-2,h)$ are connected by an edge or not. 
As before, this can be directly read off from the bigraded root of the knot by looking at the coordinate $(h[a],h[a-2])$ of the bigraded root and see if there is a node or not. 
We summarize the dictionary in Figure \ref{fig:-2_surgery_dictionary_t0} and \ref{fig:-2_surgery_dictionary_t1} below. 
\begin{figure}
    \centering
    \begin{tikzpicture}[scale=0.5]

\fill[fill=yellow, fill opacity=0.3](-2.2,-0.2) rectangle (0.2,0.2);
\fill[fill=yellow, fill opacity=0.3](-2.2,-2.2) rectangle (0.2,-1.8);
\fill[fill=yellow, fill opacity=0.3](-4.2,-4.2) rectangle (2.2,-3.8);
\fill[fill=yellow, fill opacity=0.3](-4.2,-6.2) rectangle (2.2,-5.8);
\fill[fill=yellow, fill opacity=0.3](-4.2,-8.2) rectangle (2.2,-7.8);
\fill[fill=yellow, fill opacity=0.3](-4.2,-10.2) rectangle (2.2,-9.8);
\fill[fill=yellow, fill opacity=0.3](-6.2,-12.2) rectangle (4.2,-11.8);

\node at (0,0) {$\bullet$};
\node at (0,-2) {$\bullet$};
\node at (0,-4) {$\bullet$};
\node at (0,-6) {$\bullet$};
\node at (0,-8) {$\bullet$};
\node at (0,-10) {$\bullet$};
\node at (0,-12) {$\bullet$};

\node at (-2,0) {$\bullet$};
\node at (-2,-2) {$\bullet$};
\node at (-2,-4) {$\bullet$};
\node at (-2,-6) {$\bullet$};
\node at (-2,-8) {$\bullet$};
\node at (-2,-10) {$\bullet$};
\node at (-2,-12) {$\bullet$};

\node at (2,-4) {$\bullet$};
\node at (2,-6) {$\bullet$};
\node at (2,-8) {$\bullet$};
\node at (2,-10) {$\bullet$};
\node at (2,-12) {$\bullet$};

\node at (-4,-4) {$\bullet$};
\node at (-4,-6) {$\bullet$};
\node at (-4,-8) {$\bullet$};
\node at (-4,-10) {$\bullet$};
\node at (-4,-12) {$\bullet$};

\node at (4,-12) {$\bullet$};

\node at (-6,-12) {$\bullet$};

\draw (0,0)--(0,-2)--(0,-4)--(0,-6)--(0,-8)--(0,-10)--(0,-12)--(0,-13);
\draw (-2,0)--(-2,-2)--(-2,-4)--(-2,-6)--(-2,-8)--(-2,-10)--(-2,-12)--(-2,-13);
\draw (2,-4)--(2,-6)--(2,-8)--(2,-10)--(2,-12)--(2,-13);
\draw (-4,-4)--(-4,-6)--(-4,-8)--(-4,-10)--(-4,-12)--(-4,-13);
\draw (4,-12)--(4,-13);
\draw (-6,-12)--(-6,-13);

\node at (0,2) {$0$};
\node at (-2,2) {$-2$};
\node at (2,2) {$2$};
\node at (-4,2) {$-4$};
\node at (4,2) {$4$};
\node at (-6,2) {$-6$};

\node at (-8.6,0) {$-\frac{1}{4}$};
\node at (-8.6,-2) {$-\frac{1}{4}-2$};
\node at (-8.6,-4) {$-\frac{1}{4}-4$};
\node at (-8.6,-6) {$-\frac{1}{4}-6$};
\node at (-8.6,-8) {$-\frac{1}{4}-8$};
\node at (-8.6,-10) {$-\frac{1}{4}-10$};
\node at (-8.6,-12) {$-\frac{1}{4}-12$};

\draw (-9,1)--(5,1);
\draw (-7,3)--(-7,-13);
\draw (-7,1)--(-9,3);

\node at (-7.6,2.4) {$a$};
\node at (-8.4,1.6) {$h$};

\node at (-1,0) {\textcolor{green}{$a$}};

\node at (-1,-2) {\textcolor{green}{$b$}};

\node at (-3,-4) {\textcolor{green}{$c$}};
\node at (-1,-4) {\textcolor{green}{$d$}};
\node at (1,-4) {\textcolor{green}{$e$}};

\node at (-3,-6) {\textcolor{green}{$f$}};
\node at (-1,-6) {\textcolor{green}{$g$}};
\node at (1,-6) {\textcolor{green}{$h$}};

\node at (-3,-8) {\textcolor{green}{$i$}};
\node at (-1,-8) {\textcolor{green}{$j$}};
\node at (1,-8) {\textcolor{green}{$k$}};

\node at (-3,-10) {\textcolor{green}{$l$}};
\node at (-1,-10) {\textcolor{green}{$m$}};
\node at (1,-10) {\textcolor{green}{$n$}};

\node at (-5,-12) {\textcolor{green}{$o$}};
\node at (-3,-12) {\textcolor{green}{$p$}};
\node at (-1,-12) {\textcolor{green}{$q$}};
\node at (1,-12) {\textcolor{green}{$r$}};
\node at (3,-12) {\textcolor{green}{$s$}};

\draw[ultra thick, <->] (5.5,-6)--(7,-6);


\begin{scope}[scale=1.35, shift = {(11.5,0)}]

\draw[draw=yellow, draw opacity=0.3, ultra thick](0,0)--(0,0);
\draw[draw=yellow, draw opacity=0.3, ultra thick](0,-2)--(0,-2);
\draw[draw=yellow, draw opacity=0.3, ultra thick](-2,-2)--(0,-4)--(2,-2);
\draw[draw=yellow, draw opacity=0.3, ultra thick](-2,-4)--(0,-6)--(2,-4);
\draw[draw=yellow, draw opacity=0.3, ultra thick](-2,-6)--(0,-8)--(2,-6);
\draw[draw=yellow, draw opacity=0.3, ultra thick](-2,-8)--(0,-10)--(2,-8);
\draw[draw=yellow, draw opacity=0.3, ultra thick](-4,-4)--(-2,-10)--(0,-12)--(2,-10)--(4,-4);

\node at (1, 1) {$0$};
\node at (2, 0) {$-2$};
\node at (3, -1) {$-4$};
\node at (4, -2) {$-6$};
\node at (5, -3) {$-8$};

\node at (-1, 1) {$0$};
\node at (-2, 0) {$-2$};
\node at (-3, -1) {$-4$};
\node at (-4, -2) {$-6$};
\node at (-5, -3) {$-8$};

\draw (-1,2)--(5.5,-4.5);
\draw (1,2)--(-5.5,-4.5);
\draw (0,3)--(0,1);

\node at (-0.4,2.1) {$h_U$};
\node at (0.4,2.1) {$h_V$};

\node at (0,0) {\textcolor{green}{$a$}};

\node at (0,-2) {\textcolor{green}{$b$}};

\node at (-2,-2) {\textcolor{green}{$c$}};
\node at (0,-4) {\textcolor{green}{$d$}};
\node at (2,-2) {\textcolor{green}{$e$}};

\node at (-2,-4) {\textcolor{green}{$f$}};
\node at (0,-6) {\textcolor{green}{$g$}};
\node at (2,-4) {\textcolor{green}{$h$}};

\node at (-2,-6) {\textcolor{green}{$i$}};
\node at (0,-8) {\textcolor{green}{$j$}};
\node at (2,-6) {\textcolor{green}{$k$}};

\node at (-2,-8) {\textcolor{green}{$l$}};
\node at (0,-10) {\textcolor{green}{$m$}};
\node at (2,-8) {\textcolor{green}{$n$}};

\node at (-4,-4) {\textcolor{green}{$o$}};
\node at (-2,-10) {\textcolor{green}{$p$}};
\node at (0,-12) {\textcolor{green}{$q$}};
\node at (2,-10) {\textcolor{green}{$r$}};
\node at (4,-4) {\textcolor{green}{$s$}};

\end{scope}
\end{tikzpicture}
    \caption{$-2$ surgery dictionary for $\mathfrak{t}_0$}
    \label{fig:-2_surgery_dictionary_t0}
\end{figure}
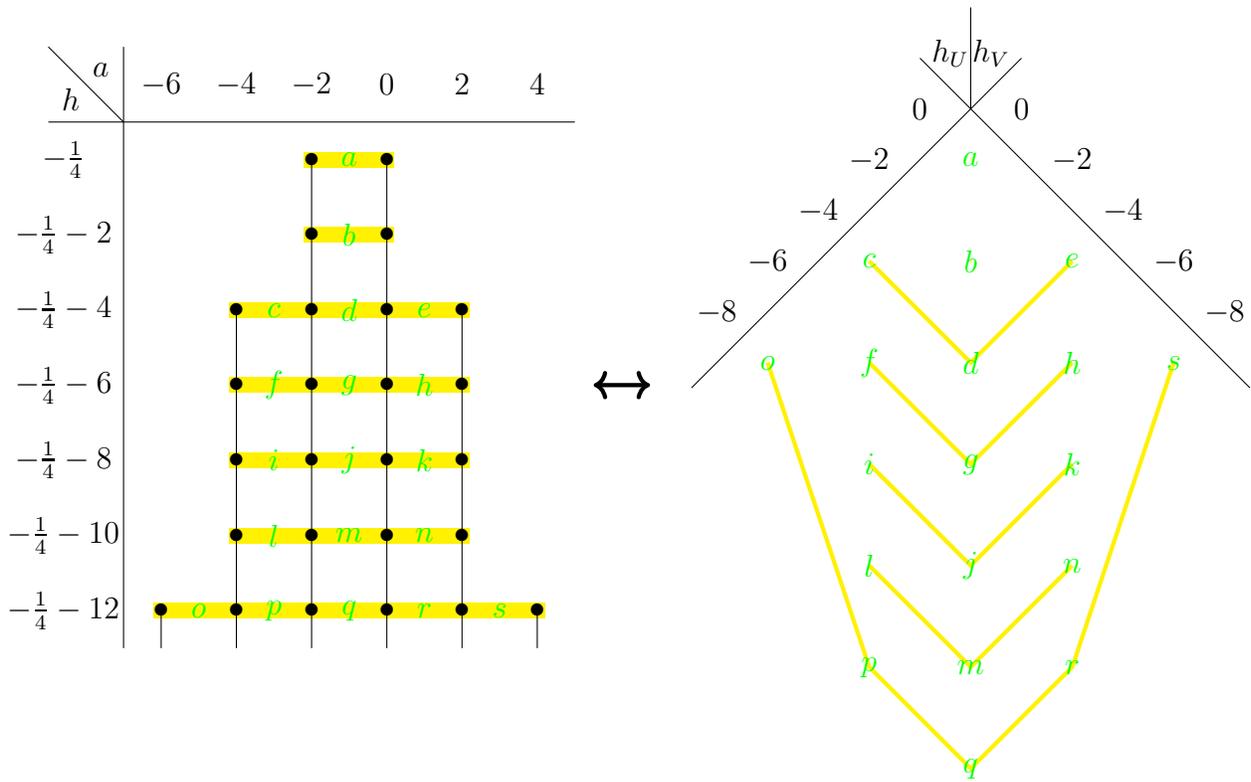
\begin{figure}
    \centering
    \begin{tikzpicture}[scale=0.45]
\fill[fill=yellow, fill opacity=0.3](-4.2,-2.2) rectangle (0.2,-1.8);
\fill[fill=yellow, fill opacity=0.3](-4.2,-4.2) rectangle (0.2,-3.8);
\fill[fill=yellow, fill opacity=0.3](-4.2,-6.2) rectangle (0.2,-5.8);
\fill[fill=yellow, fill opacity=0.3](-6.2,-8.2) rectangle (2.2,-7.8);
\fill[fill=yellow, fill opacity=0.3](-6.2,-10.2) rectangle (2.2,-9.8);
\fill[fill=yellow, fill opacity=0.3](-6.2,-12.2) rectangle (2.2,-11.8);

\node at (0,-2) {$\bullet$};
\node at (0,-4) {$\bullet$};
\node at (0,-6) {$\bullet$};
\node at (0,-8) {$\bullet$};
\node at (0,-10) {$\bullet$};
\node at (0,-12) {$\bullet$};

\node at (-2,0) {$\bullet$};
\node at (-2,-2) {$\bullet$};
\node at (-2,-4) {$\bullet$};
\node at (-2,-6) {$\bullet$};
\node at (-2,-8) {$\bullet$};
\node at (-2,-10) {$\bullet$};
\node at (-2,-12) {$\bullet$};

\node at (2,-8) {$\bullet$};
\node at (2,-10) {$\bullet$};
\node at (2,-12) {$\bullet$};

\node at (-4,-2) {$\bullet$};
\node at (-4,-4) {$\bullet$};
\node at (-4,-6) {$\bullet$};
\node at (-4,-8) {$\bullet$};
\node at (-4,-10) {$\bullet$};
\node at (-4,-12) {$\bullet$};

\node at (-6,-8) {$\bullet$};
\node at (-6,-10) {$\bullet$};
\node at (-6,-12) {$\bullet$};

\draw (0,-2)--(0,-4)--(0,-6)--(0,-8)--(0,-10)--(0,-12)--(0,-13);
\draw (-2,0)--(-2,-2)--(-2,-4)--(-2,-6)--(-2,-8)--(-2,-10)--(-2,-12)--(-2,-13);
\draw (2,-8)--(2,-10)--(2,-12)--(2,-13);
\draw (-4,-2)--(-4,-4)--(-4,-6)--(-4,-8)--(-4,-10)--(-4,-12)--(-4,-13);
\draw (-6,-8)--(-6,-10)--(-6,-12)--(-6,-13);

\node at (0,2) {$1$};
\node at (-2,2) {$-1$};
\node at (2,2) {$3$};
\node at (-4,2) {$-3$};
\node at (-6,2) {$-5$};

\node at (-8.5,0) {$\frac{1}{4}$};
\node at (-8.5,-2) {$\frac{1}{4}-2$};
\node at (-8.5,-4) {$\frac{1}{4}-4$};
\node at (-8.5,-6) {$\frac{1}{4}-6$};
\node at (-8.5,-8) {$\frac{1}{4}-8$};
\node at (-8.5,-10) {$\frac{1}{4}-10$};
\node at (-8.5,-12) {$\frac{1}{4}-12$};

\draw (-9,1)--(3,1);
\draw (-7,3)--(-7,-13);
\draw (-7,1)--(-9,3);

\node at (-7.6,2.4) {$a$};
\node at (-8.4,1.6) {$h$};

\node at (-3,-2) {\textcolor{green}{$a$}};
\node at (-1,-2) {\textcolor{green}{$b$}};

\node at (-3,-4) {\textcolor{green}{$c$}};
\node at (-1,-4) {\textcolor{green}{$d$}};

\node at (-3,-6) {\textcolor{green}{$e$}};
\node at (-1,-6) {\textcolor{green}{$f$}};

\node at (-5,-8) {\textcolor{green}{$g$}};
\node at (-3,-8) {\textcolor{green}{$h$}};
\node at (-1,-8) {\textcolor{green}{$i$}};
\node at (1,-8) {\textcolor{green}{$j$}};

\node at (-5,-10) {\textcolor{green}{$k$}};
\node at (-3,-10) {\textcolor{green}{$l$}};
\node at (-1,-10) {\textcolor{green}{$m$}};
\node at (1,-10) {\textcolor{green}{$n$}};

\node at (-5,-12) {\textcolor{green}{$o$}};
\node at (-3,-12) {\textcolor{green}{$p$}};
\node at (-1,-12) {\textcolor{green}{$q$}};
\node at (1,-12) {\textcolor{green}{$r$}};

\draw[ultra thick, <->] (4,-6)--(6,-6);


\begin{scope}[scale=1.4, shift={(11,0)}]

\draw[draw=yellow, draw opacity=0.3, ultra thick](-1,-1)--(1,-1);
\draw[draw=yellow, draw opacity=0.3, ultra thick](-1,-3)--(1,-3);
\draw[draw=yellow, draw opacity=0.3, ultra thick](-1,-5)--(1,-5);
\draw[draw=yellow, draw opacity=0.3, ultra thick](-3,-3)--(-1,-7)--(1,-7)--(3,-3);
\draw[draw=yellow, draw opacity=0.3, ultra thick](-3,-5)--(-1,-9)--(1,-9)--(3,-5);
\draw[draw=yellow, draw opacity=0.3, ultra thick](-3,-7)--(-1,-11)--(1,-11)--(3,-7);

\node at (1, 1) {$0$};
\node at (2, 0) {$-2$};
\node at (3, -1) {$-4$};
\node at (4, -2) {$-6$};
\node at (5, -3) {$-8$};

\node at (-1, 1) {$0$};
\node at (-2, 0) {$-2$};
\node at (-3, -1) {$-4$};
\node at (-4, -2) {$-6$};
\node at (-5, -3) {$-8$};

\draw (-1,2)--(5.5,-4.5);
\draw (1,2)--(-5.5,-4.5);
\draw (0,3)--(0,1);

\node at (-0.4,2.1) {$h_U$};
\node at (0.4,2.1) {$h_V$};

\node at (-1,-1) {\textcolor{green}{$a$}};
\node at (1,-1) {\textcolor{green}{$b$}};

\node at (-1,-3) {\textcolor{green}{$c$}};
\node at (1,-3) {\textcolor{green}{$d$}};

\node at (-1,-5) {\textcolor{green}{$e$}};
\node at (1,-5) {\textcolor{green}{$f$}};

\node at (-3,-3) {\textcolor{green}{$g$}};
\node at (-1,-7) {\textcolor{green}{$h$}};
\node at (1,-7) {\textcolor{green}{$i$}};
\node at (3,-3) {\textcolor{green}{$j$}};

\node at (-3,-5) {\textcolor{green}{$k$}};
\node at (-1,-9) {\textcolor{green}{$l$}};
\node at (1,-9) {\textcolor{green}{$m$}};
\node at (3,-5) {\textcolor{green}{$n$}};

\node at (-3,-7) {\textcolor{green}{$o$}};
\node at (-1,-11) {\textcolor{green}{$p$}};
\node at (1,-11) {\textcolor{green}{$q$}};
\node at (3,-7) {\textcolor{green}{$r$}};

\end{scope}
\end{tikzpicture}
    \caption{$-2$ surgery dictionary for $\mathfrak{t}_1$}
    \label{fig:-2_surgery_dictionary_t1}
\end{figure}
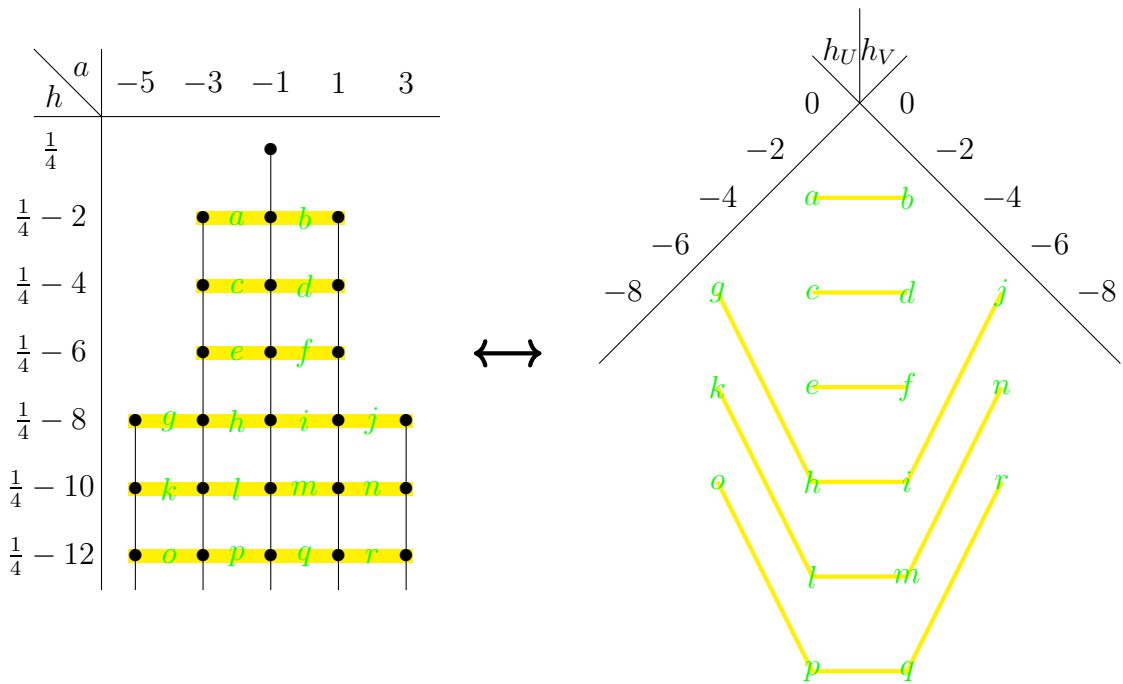
\end{exmp}

\begin{rem}
The examples given above can be easily generalized to any $-p$ surgery. 
For $\Sigma^2 = -p$ surgery, the $a$-th copy of the graded root is shifted by
\[
h - h[a] = \frac{1}{4} - \frac{1}{p}\left(a+\frac{p}{2}\right)^2. 
\]
The nodes at coordinates $(a,h)$ and $(a-p,h)$ are connected by an edge if and only if there is a node in the bigraded root at $(h_U,h_V) = (h[a],h[a-p])$.

Note that 
\[
\frac{h[a]-h[a-p]}{2} = a,
\]
so if a spin$^c$ structure $\mathfrak{t}$ of the surgered manifold corresponds to $a$'s with $a \equiv b \bmod p$, then only the nodes in the bigraded root whose Alexander grading $A = \dfrac{h_U - h_V}{2}$ is $b \bmod p$ are used in the surgery. 
\end{rem}

\subsection{BPS $q$-series for plumbed knot complements}
\label{sec:BPS q series for plumbed knot complements}
In this subsection we discuss (a renormalization of) the Gukov-Manolescu \cite{GM} series invariant of negative definite plumbed knot complements. 
For a negative definite marked plumbing graph $\Gamma_{v_0}$ with $s+1$ vertices
and a choice of relative $\spinc$ structure 
\[
[b] \in \frac{\h{\delta} + 2\mathbb{Z}^{s+1}}{2\mmatrix(0\times \mathbb{Z}^s)} 
\cong \spinc(\mmfld), 
\]
the corresponding BPS $q$-series is given by
\begin{align}
\begin{aligned}
\label{eq:Zhatplumbedknotcomplement}
\widehat{Z}_{[b]}(q,t) &:= \oint \frac{dz_0}{2\pi i z_0}
(t^{-\frac{1}{2}}z_0-t^{\frac{1}{2}}z_0^{-1})^{1-\delta_0}
\oint \prod_{v\neq v_0} \frac{d z_v}{2\pi i z_v} \left(t^{-\frac{1}{2}}z_v - t^{\frac{1}{2}}z_v^{-1}\right)^{2-\delta_v} \\
&\quad\quad\quad \times 
q^{-\frac{3s+\sum_{v\neq v_0} m_v + \hdual{\onev} M^{-1} \onev}{4}}
\sum_{\ell \in b\vert_{\Gamma} + 2M\mathbb{Z}^{s}}
q^{-\frac{\hdual{\ell} M^{-1}\ell}{4}}
z_0^{\hdual{\onev} M^{-1}(\ell-b\vert_{\Gamma}) + b_0}
\prod_{v\neq v_0}z_v^{\ell_v}.
\end{aligned}
\end{align}

Note, our expression is slightly different from the expression given in \cite[Section 6]{GM}, which depends on a triple of labels, $([a], n_0, \zeta_0)$, where
$[a] \in \frac{\delta + 2\mathbb{Z}^{s+1}}{2M_{v_{0}}(0\times \mathbb{Z}^s)}$, $n_0\in \mathbb{Z}$, and $\zeta_0 \in 1+2\mathbb{Z}$ is the exponent of $z_0$.\footnote{In \cite{GM}  the label $[a]$ is called the ``relative $\spinc$ structure'', which differs from our conventions. See Remark \ref{rem:delta vs delta hat for rel spinc}.} 
A framing $m_0$ on $v_0$ is also fixed in \cite{GM}; we denote the corresponding adjacency matrix by $\smatrix$. 
Under conjugation of the relative $\spinc$ structure, the triple transforms in the following way:
\[
([a], n_0, \zeta_0) \mapsto (-[a], -n_0, -\zeta_0).
\]
In fact, the triple of labels determines a relative $\spinc$ structure by
\[
[(a,\zeta_0,n_0)] := [a - \zeta_0 e_0 + 2n_0(\smatrix e_0)] 
\in \frac{\h{\delta} + 2\mathbb{Z}^{s+1}}{2M_{v_{0}}(0\times \mathbb{Z}^s)} 
\cong \spinc(\mmfld).
\]
Different triples representing the same relative $\spinc$ structure are related by the symmetry described in \cite[Section 6.6]{GM}, but the resulting BPS $q$-series differ by some overall monomial factor if we use the expression given in \cite{GM}. 
In \eqref{eq:Zhatplumbedknotcomplement}, we have fixed the overall normalization so that it depends only on the relative $\spinc$ structure $[b]$, not on the triple of labels representing $[b]$.

\begin{rem}
It is possible to leave the $z_0$ variable unintegrated:
\begin{align}
\begin{aligned}
\label{eq:Zhatz0naive}
\widehat{Z}_{[b]}(z_0, q,t) &:= 
(t^{-\frac{1}{2}}z_0-t^{\frac{1}{2}}z_0^{-1})^{1-\delta_0}
\oint \prod_{v\neq v_0} \frac{d z_v}{2\pi i z_v} \left(t^{-\frac{1}{2}}z_v - t^{\frac{1}{2}}z_v^{-1}\right)^{2-\delta_v} \\
&\quad\quad\quad \times 
q^{-\frac{3s+\sum_{v\neq v_0} m_v + \hdual{\onev} M^{-1} \onev}{4}}
\sum_{\ell \in b\vert_{\Gamma} + 2M\mathbb{Z}^{s}}
q^{-\frac{\hdual{\ell} M^{-1}\ell}{4}}
z_0^{\hdual{\onev} M^{-1}(\ell-b\vert_{\Gamma}) + b_0}
\prod_{v\neq v_0}z_v^{\ell_v}.
\end{aligned}
\end{align}
This expression is again a well-defined invariant, but it does not contain any more information than \eqref{eq:Zhatplumbedknotcomplement}. 
This is because
\begin{align*}
\widehat{Z}_{[b]}(z_0, q, t) = \sum_{k} \left(\oint \frac{dz_0}{2\pi i z_0} \widehat{Z}_{[b]}(z_0, q, t)z_0^{-k} \right) z_0^{k} = \sum_{k} \widehat{Z}_{[b-ke_0]}(q,t) z_0^{k}. 
\end{align*}
Note, when $Y$ is an integer homology sphere so that the set of relative $\spinc$ structures is $\spinc(\mmfld) \cong 1 + 2\mathbb{Z}$ and the meridian acts on it by shifting by $2$, then up to overall power of $z_0$, this is the generating series
\[
\sum_{[b]\in 1+2\mathbb{Z}} \widehat{Z}_{[b]}(q,t) z_0^{[b]}.
\]
\end{rem}

In our main construction of weighted bigraded roots, 
we will actually slightly shift the exponent of $z_0$ and use the following simpler form of BPS $q$-series 
\begin{align}
\begin{aligned}
\label{eq:Zhatz0}
\widehat{Z}_{[b\vert_{\Gamma}]}(z_0, q,t) &:= 
(t^{-\frac{1}{2}}z_0-t^{\frac{1}{2}}z_0^{-1})^{1-\delta_0}
\oint \prod_{v\neq v_0} \frac{d z_v}{2\pi i z_v} \left(t^{-\frac{1}{2}}z_v - t^{\frac{1}{2}}z_v^{-1}\right)^{2-\delta_v} \\
&\quad\quad\quad \times 
q^{-\frac{3s+\sum_{v\neq v_0} m_v + \hdual{\onev} M^{-1} \onev}{4}}
\sum_{\ell \in b\vert_{\Gamma} + 2M\mathbb{Z}^{s}}
q^{-\frac{\hdual{\ell} M^{-1}\ell}{4}}
z_0^{\hdual{\onev} M^{-1}\ell}
\prod_{v\neq v_0}z_v^{\ell_v}
\end{aligned}
\end{align}
which depends only on 
\[
[b\vert_{\Gamma}] \in \frac{\adeg + \onev + 2\mathbb{Z}^s}{2M\mathbb{Z}^s}. 
\]

Recall the map $\relmap_n$ from \eqref{eq:rel spinc to abs spinc map}. 
While it is defined for all odd integers $n$, as in Section \ref{sec:Zhat closed}, only the two choices $\relmap_{\pm 1}$ will transform properly with respect to Neumann moves, in our construction of weighted bigraded roots. 
This is apparent in the proof of Theorem \ref{thm:invariance}. 
As in Section \ref{sec:Zhat closed}, we write $\eps$ in place of $n$ to emphasize this restriction.

With a choice of $\eps \in \{\pm 1\}$, the label $[b\vert_\Gamma]$  can be identified with
\[
\relmap_\eps([b]) = [b\vert_{\Gamma} + \eps(\onev + Mu)] \in \frac{m + 2\mathbb{Z}^s}{2M\mathbb{Z}^s} \cong \spinc(Y),
\]
the $\spinc$ structure on the ambient $3$-manifold obtained by gluing (via $\infty$-surgery) the relative $\spinc$ structure $[b]$ on $Y_{v_0}$ with the relative $\spinc$ structure on the solid torus determined by the choice $\eps \in \{\pm 1\} \subset 1+2\mathbb{Z} \cong \spinc(S^1\times D^2)$. 
In other words, $\widehat{Z}_{[b\vert_{\Gamma}]}(z_0, q,t)$ depends only on the $\spinc$ structure on the ambient manifold determined by $[b]$ and $\eps \in \{\pm 1\}$. 

It is straightforward to check that the expression \eqref{eq:Zhatz0naive} is independent of the choice of representative $b$ of $[b]$, and  that both \eqref{eq:Zhatz0naive} and \eqref{eq:Zhatz0} are invariant under Neumann moves. 
This will follow from the stronger statement in Theorem \ref{thm:invariance}. 

To begin our unification of $\Zhat$ and knot lattice homology, just as in Section \ref{sec:Zhat closed}, we need to identify the lattices which are used to define each theory. 
Namely, the sum in the definition of $\Zhat$ in \eqref{eq:Zhatz0} is over the lattice $b\vert_{\agraph} + 2 \amatrix \Z^s$ while the (bi)graded root is defined in terms of the lattice $\relmap_\eps([b]) = b\vert_{\agraph}  + \eps(\onev + Mu) + 2 \amatrix \Z^s  \subset \Char(\agraph)$. 
With the identification 
\begin{align}
\begin{aligned}
    \label{eq:identifying the lattics}
   b\vert_{\agraph}  + 2 \amatrix \Z^s &\longleftrightarrow b\vert_{\agraph}  + \eps(\onev + Mu) + 2 \amatrix \Z^s \\
    \ell & \longleftrightarrow K = \ell + \eps(\onev + \amatrix \one),
    \end{aligned}
\end{align}
we can rewrite $\widehat{Z}_{[b\vert_{\Gamma}]}(z_0, q,t)$ from \eqref{eq:Zhatz0} as
\begin{equation}
    \label{eq:Zhatz0 in terms of K}
    \widehat{Z}_{[b\vert_{\Gamma}]}(z, q,t) =  (t^{-\frac{1}{2}}z-t^{\frac{1}{2}}z^{-1})^{1-\delta_0} \sum_{K \in \relmap_\eps([b])} \h{W}_{\Gamma_{v_0}}(K) q^{\qexp(K)} z^{\zexp(K)}t^{\texp(K)}.
\end{equation}
We will often write $z$ in place of $z_0$ as above.
 The coefficient $\h{W}_{\mgraph}(K)$ is given by
\begin{equation}
\label{eq:W-hat weight}
       \h{W}_{\Gamma_{v_0}}(K) =  \prod_{i=1}^s \h{W}_{\delta_i} ((K -\eps( M\one + \onev) )_i)
\end{equation}
where $\h{W} = \{\h{W}_n\}_{n\geq 0}$ is the admissible family in \eqref{eq:W hat}. 
The exponents of $q, z$, and $t$ are 
\begin{align}
\label{eq:q exp}
\begin{aligned}
    \qexp(K) &= -\frac{3s+\hdual{\delta} u + 2\sum\limits_{v\neq v_0} m_v + 2\hdual{\onev} M^{-1} \onev}{4}-\frac{\hdual{K}M^{-1}K}{4}+\frac{\eps \hdual{K}M^{-1}\onev}{2}+\frac{\eps \hdual{K} u}{2},
    \end{aligned}
    \end{align}

   \begin{align}
       \label{eq:z exp}
       \zexp(K) &= \hdual{\onev}M^{-1}K-\eps\hdual{\onev} M^{-1}\onev-\eps \delta_0,
   \end{align}

    \begin{align}
    \label{eq:t exp}
    \begin{aligned}
    \texp(K) &= \frac{\hdual{K} u -\eps \sum\limits_{v\neq v_0}(\delta_v+m_v)}{2}.
    \end{aligned}
\end{align}

The series \eqref{eq:Zhatz0naive} can be similarly rewritten as a sum over $\relmap_\eps([b])$. 
The only modification is that the power of $z$ is given by 
\begin{align}
\label{eq:z exp b}
\zexp_{[b]}(K) &= \hdual{\onev}M^{-1}K-\eps\hdual{\onev} M^{-1}\onev-\eps \delta_0-\hdual{\onev}M^{-1}(b\vert_{\agraph} )+b_0 = \zexp(K) -\hdual{\onev}M^{-1}(b\vert_{\agraph} )+b_0.
\end{align} 
See Remark \ref{rem:weighted graded root with b} for a further discussion.

\subsection{The weighted (bi)graded root for plumbed knot complements}
\label{sec:weighted graded root for plumbed knot complements}

In this section we introduce the main construction of the paper: three-variable weights assigned to each node of the graded roots of the ambient plumbing graph. 
We can equivalently package this as a three-variable weight assigned to each node of the bigraded roots, as in Definition \ref{def:weighted bigraded root}, which is often more convenient. Invariance of these objects under Neumann moves is proven in Section \ref{sec:invariance}.   

To begin, let $\mgraph$ be a negative definite marked plumbing graph with $s+1$ vertices. 
As usual, $\agraph = \mgraph\setminus\{v_0\}$  denotes the ambient plumbing graph with intersection form $M$. 
We also have a fixed choice of $\eps \in \{\pm 1\}$, which is often omitted from the notation. 
Let $\Ring$ be a commutative ring and let $W = \{ W_n : \Z\to \Ring\}_{n\geq 0}$ be an admissible family of functions, as in Definition \ref{def:admissible family}. 
As a generalization of \eqref{eq:W-hat weight}, define $W_{\mgraph} : \Char(\agraph) \to \Ring$ by 
\begin{equation}
    \label{eq:weight of component knot complement}
    W_{\Gamma_{v_0}}(K) =  \prod_{i=1}^s W_{\delta_i} ((K -\eps( M\one + \onev) )_i).
\end{equation} 

We now assign $3$-variable Laurent polynomial weights to the graded root of $\agraph$ at a $\spinc$ structure $[k] \in \spinc(\agraph)$. 
For a connected component $C\subset \superlevelone_{i}(\agraph,[k])$, we define the weight of $C$ to be
\begin{align}
\label{eq:U weight}
    W_{\mgraph}(C; z,q,t) = (t^{-\frac{1}{2}}z-t^{\frac{1}{2}}z^{-1})^{1-\delta_0} \sum\limits_{K\in C \cap[k]}W_{\Gamma_{v_0}}(K) q^{\qexp(K)} z^{\zexp(K)}t^{\texp(K)}.
\end{align}
In the above sum, we intersect with $[k]$ to indicate that we only consider the lattice points ($0$-cells) of the component $C$. 
For $\delta_0 \geq 2$, the factor $(t^{-\frac{1}{2}}z-t^{\frac{1}{2}}z^{-1})^{1-\delta_0} $ is expanded according to \eqref{eq:expansion} with the change of variables $z\mapsto t^{-1/2}z$.

\begin{defn}
\label{def:weighted graded root for knot complements}
The object obtained by assigning to each vertex $C$ of the graded root $\root(\agraph, [k])$ the weight $W_{\mgraph}(C;z,q,t)$ is called the \textit{weighted graded root for the knot complement} $\Gamma_{v_0}$ and denoted by $\root_\eps(\mgraph, [k],W)$. 
\end{defn}

\begin{rem}
    While the BPS $q$-series \eqref{eq:Zhatz0} does not depend on the choice of $\eps$, the weighted graded root does. 
    The relationship between the two choices of $\eps$ is the content of Proposition \ref{prop:eps vs -eps knot complement}. 
\end{rem}

\begin{rem}
    From the construction, it is evident that in the limit $h_U \rightarrow -\infty$, the weights stabilize (in a sense analogous to \cite[Definition 6.1]{AJK}) to the series 
    \[
(t^{-\frac{1}{2}}z-t^{\frac{1}{2}}z^{-1})^{1-\delta_0} \sum_{K\in [k]} W_{\mgraph}(K) q^{\qexp(K)} z^{\zexp(K)} t^{\texp(K)}.
    \] 
  In particular, when the admissible family is $\h{W}$, we recover the BPS $q$-series \eqref{eq:Zhatz0} from the weights of the weighted graded root $\root_\eps(\mgraph, [k], \h{W})$. 
\end{rem}

\begin{exmp}[Weighted graded root for the unknot]
Consider the marked plumbing graph $\mgraph$ from Example \ref{ex:bigraded root unknot in S3}, representing the unknot in $S^3$. 
The graded root for $S^3$ consists of a single node in each non-positive even integer grading.

For $K\in \Char(\agraph) = 1+2\Z$, we have $K-\eps(\amatrix u + \onev) = K$, so that $W_{\mgraph}(K) = W_1(K)$ is zero unless $K=\pm 1$. 
We compute:
\[
    W_{\mgraph}(\pm 1) = \mp 1, \ \ \ 
    \qexp(\pm 1)  = 0, \ \ \ 
    \zexp(\pm 1) = \mp 1, \ \ \ 
    \texp(\pm 1) = \mp \frac{ 1 }{2}.
\]
Observe that $h_U(\pm 1) = 0$, so  $1$ and $-1$, the only two characteristic vectors that contribute to the weight, are already contained in the maximal non-empty superlevel set. 
We also note that the above weights in this example are independent of the choice of $\eps \in \{\pm 1\}$. 
The weighted graded root of the unknot is shown in Figure \ref{fig:unknot weighted graded root}. 
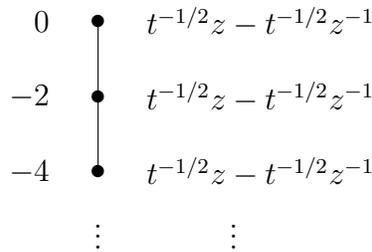
\begin{figure}[H]
    \centering
    \begin{tikzpicture}[scale=0.5]
\node at (0,0) {$\bullet$};
\node at (0,-2) {$\bullet$};
\node at (0,-4) {$\bullet$};

\draw (0,0)--(0,-2)--(0,-4)--(0,-4);

\node at (-1,0) [align = right, anchor=east] {$ 0$};
\node at (-1,-2) [align = right, anchor=east] {$ -2$};
\node at (-1,-4) [align = right, anchor=east] {$ -4$};

\node at (1,0) [align = left, anchor=west] {$t^{-1/2}z-t^{-1/2}z^{-1}$};
\node at (1,-2) [align = left, anchor=west] {$t^{-1/2}z-t^{-1/2}z^{-1}$};
\node at (1,-4) [align = left, anchor=west] {$t^{-1/2}z-t^{-1/2}z^{-1}$};
\node at (3.6, -5.5) {$\vdots$};
\node at (0, -5.5) {$\vdots$};
\end{tikzpicture}
    \caption{The weighted graded root for the unknot.}
    \label{fig:unknot weighted graded root}
\end{figure}

\end{exmp}

\begin{exmp}[Weighted graded root for the trefoil knot] A plumbing representation for the trefoil as well as its bigraded root was given in Figure \ref{fig:trefoil_bigraded_root}. 
The corresponding weighted graded root at $t=1$ is shown in Figure \ref{fig:trefoil weighted graded root}.
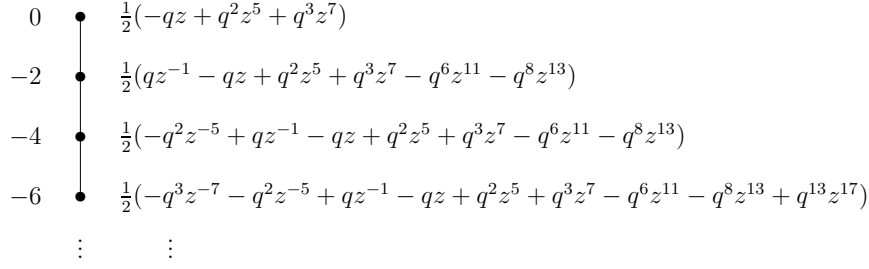
\begin{figure}[H]
    \centering
    \begin{tikzpicture}[scale=0.5]
\node at (0,0) {$\bullet$};
\node at (0,-2) {$\bullet$};
\node at (0,-4) {$\bullet$};
\node at (0,-6) {$\bullet$};

\draw (0,0)--(0,-2)--(0,-4)--(0,-6);

\node at (-1,0) [align = right, anchor=east] {$0$};
\node at (-1,-2) [align = right, anchor=east] {$-2$};
\node at (-1,-4) [align = right, anchor=east] {$-4$};
\node at (-1,-6) [align = right, anchor=east] {$-6$};

\node at (1,0) [align = left, anchor=west] {$\frac{1}{2}(-q z + q^2 z^5 + q^3 z^7)$};
\node at (1,-2) [align = left, anchor=west] {$\frac{1}{2}(q z^{-1} -q z + q^2 z^5 + q^3 z^7 -q^6 z^{11} -q^8 z^{13})$};
\node at (1,-4) [align = left, anchor=west] {$\frac{1}{2}(-q^2 z^{-5} + q z^{-1} -q z + q^2 z^5 + q^3 z^7 -q^6 z^{11} -q^8 z^{13})$};
\node at (1,-6) [align = left, anchor=west] {$\frac{1}{2}(-q^3 z^{-7} -q^2 z^{-5} + q z^{-1} -q z + q^2 z^5 + q^3 z^7 -q^6 z^{11} -q^8 z^{13} + q^{13}z^{17})$};
\node at (3, -7.5) {$\vdots$};
\node at (0, -7.5) {$\vdots$};
\end{tikzpicture}
    \caption{The weighted graded root for the trefoil at $\eps=+1$, specialized to $t=1$.}
    \label{fig:trefoil weighted graded root}
\end{figure}
\end{exmp}

\begin{defn}
    \label{def:weighted bigraded root}
    For a negative definite marked plumbing graph $\mgraph$, consider its bigraded root $\rootbi(\mgraph, [k])$ for some $\spinc$ structure $[k]\in \spinc(\agraph)$ of the ambient manifold. 
    Given a node of the bigraded root corresponding to a connected component $C$ in some $\superlevelone_{i,j}(\mgraph,[k])$, let $C^U$ denote the connected component of $\superlevelone{}^{U}_i(\mgraph,[k])$ which contains $C$, and define the weight of $C$ to be 
    \[
(t^{-\frac{1}{2}} z - t^{\frac{1}{2}}z^{-1})^{1-\delta_0}\sum_{K\in C^U \cap [k]} W_{\mgraph}(K) q^{\qexp(K)} z^{\zexp(K)} t^{\texp(K)}
    \]
    for some fixed choice of $\eps \in \{\pm 1\}$ and admissible family $W$. 
    We call the result the \emph{weighted bigraded root}, denoted $\rootbi_\eps(\mgraph, [k], W)$. 
\end{defn}

If $D \subset \superlevelone_{i,j-2}(\mgraph,[k])$ is a connected component containing $C$ (meaning there is an edge in the bigraded root in the $V$-direction between the nodes represented by $C$ and $D$) then $C^U = D^U$, so it follows that the weights of $C$ and of $D$ in $\rootbi_\eps(\mgraph, [k], W)$ are equal. 
In other words, all the nodes in a fixed  $U$-coordinate (in the sense of Definition \ref{def:coordinates}) have the same weight in $\rootbi_\eps(\mgraph, [k], W)$. 
See, for example, Figure \ref{fig:trefoil weighted bigraded root}. 
In light of the procedure in Remark \ref{rem:recovering graded root from bigraded root}, it follows that  $\rootbi_\eps(\mgraph, [k], W)$ carries the same information as the pair consisting of the weighted graded root $\root_\eps(\mgraph, [k],W)$ (as defined in Definition \ref{def:weighted graded root for knot complements}) and the bigraded root $\rootbi(\mgraph,[k])$.

\begin{rem}
\label{rem:sum over intersection}
 Let us demonstrate that the weights in \eqref{eq:U weight} do not constitute an invariant if one were to sum over lattice points in connected components of $\superlevelone_{i,j}(\Gamma,[k])$, the intersection of the $h_U$ and $h_V$ superlevel sets. 

For simplicity, we set $q=t=1$. Consider the three marked plumbing graphs $\mgraph, \mgraph', \mgraph''$ below.
  \begin{equation*}
        \begin{tikzpicture}
        \draw (0,0) -- (1,0);
            \node (a) at (0,0) {$\bullet$}; 
            \draw[fill=white] (1,0) circle (2.5pt);
            \node[below] at (a) {$-1$};
            \node[below] at (.5,-.6) {$\mgraph$};
        \end{tikzpicture}
        \hskip6em
         \begin{tikzpicture}
        \draw (-1,0) -- (1,0);
            \node (a) at (0,0) {$\bullet$}; 
            \node (b) at (-1,0) {$\bullet$}; 
            \draw[fill=white] (1,0) circle (2.5pt);
            \node[below] at (a) {$-1$};
             \node[below] at (b) {$-2$};
             \node[below] at (0,-.6) {$\mgraph'$};
        \end{tikzpicture}
        \hskip6em
         \begin{tikzpicture}
        \draw (-1,0) -- (1,0);
            \node (a) at (1,0) {$\bullet$}; 
            \node (b) at (-1,0) {$\bullet$}; 
            \draw[fill=white] (0,0) circle (2.5pt);
            \node[below] at (a) {$-1$};
             \node[below] at (b) {$-1$};
             \node[below] at (0,-.6) {$\mgraph''$};
        \end{tikzpicture}
    \end{equation*}
    Each of them represents the unknot in $S^3$; indeed, $\mgraph'$ (resp.  $\mgraph''$) is obtained from $\mgraph$ by a type \hyperlink{A0}{(A0)} move (resp. a type \hyperlink{B0}{(B0)} move). 
    We denote by $\agraph, \agraph'$, and $\agraph''$ the ambient plumbing graphs, and let $\mathfrak{s}_0$ denote the unique $\spinc$ structure on the ambient manifold in each case. 
    
For $K\in \Char(\agraph)$, 
\[
h_U(K) = \frac{ K^2 + 1}{4}, \hskip2em h_V(K) = \frac{(K+2)^2+1}{4},
\]
so that the maximum value of both $h_U$ and $h_V$ is equal to zero. 
Then $\superlevelone_{0,0} (\mgraph,\mathfrak{s}_0)= \{-1\}$ is a singleton. One can verify that, at $q=t=1$,  
\[
W_{\mgraph}(-1) = z
\]
for both choices of $\eps \in \{\pm 1\}$.

Next, for $K=(n_1, n_2) \in \Char(\agraph')$,
\begin{align*}
h_U(K) &= \frac{-n_1^2 - 2 n_1 n_2-2 n_2n_2+2}{4},\\
h_V(K) &= \frac{-n_1 n_1 - 2 n_1 (n_2+2)-2 (n_2+2) (n_2+2)+2}{4} 
\end{align*}
It is straightforward to verify that $\superlevelone_{0,0}(\mgraph',\mathfrak{s}_0) = \{(0,-1)\}$ is again a singleton. At $q=t=1$, we have
\[
W_{\mgraph'}(0,-1) = 
\begin{cases}
    0 & \text{ if } \eps=1, \\
    z & \text{ if } \eps=-1.
\end{cases}
\]
Finally, for a characteristic vector $K=(n_1, n_2) \in \Char(\agraph'')$, we have 
\begin{align*}
h_U(K) &= \frac{-n_1 n_1-n_2 n_2+2}{4}, \\
h_V(K) &= \frac{-(n_1+2) (n_1+2)-(n_2+2) (n_2+2)+2}{4}.
\end{align*}
It is straightforward to check that again $\superlevelone_{0,0}(\mgraph'',\mathfrak{s}_0) = \{(-1,-1)\}$ is a singleton. The weight at $q=t=1$, for both choices of $\eps$, is
\[
W_{\mgraph''}(-1,-1) = (z-z^{-1})^{-1}z^{2}.
\]
A modification of the notion of an admissible family or the weight assigned to a characteristic vector could potentially yield an invariant weighted bigraded root in which weights are assigned by summing over lattice points in connected components of $\superlevelone_{i,j}(\mgraph,[k])$; 
however, we do not pursue this in the present paper. 
\end{rem}

We end this subsection with an analogue of Proposition \ref{prop:epsilon vs -epsilon} for knot complements. 

\begin{prop}
\label{prop:eps vs -eps knot complement}
    Let $W$ be an admissible family satisfying property \eqref{eq:AD3}. 
    For any  negative definite marked plumbing graph $\mgraph$ and $\spinc$ structure $[k] \in \spinc(\agraph)$, $\root_{-\eps}(\mgraph, [-k],W)$ is obtained from $\root_\eps(\mgraph, [k],W)$ by the change of variables $z\mapsto z^{-1}, t\mapsto t^{-1}$, and negating each weight. 
\end{prop}

\begin{proof}
    Recall the involution $\iota$ of $\Char(\agraph)$, $\iota(K) = -K$, which induces an isomorphism of graded roots $\root(\agraph, [-k]) \cong \root(\agraph, [k])$. 
    We refine the notation in equations \eqref{eq:q exp}, \eqref{eq:z exp}, \eqref{eq:t exp} and \eqref{eq:weight of component knot complement} to include the choice of $\eps$, writing $\qexp_{\eps}$, $\zexp_\eps$, $\texp_\eps$, and $W_{\mgraph, \eps}$. 
    It is straightforward to verify that 
    \[
\qexp_{-\eps}(-K) = \qexp_\eps(K), \ \ \zexp_{-\eps}(-K) = -\zexp_{\eps}(K), \ \ \texp_{-\eps}(-K) = -\texp_\eps(K).
    \]
    Property \eqref{eq:AD3} implies
    \[
W_{\mgraph, -\eps}(-K) = (-1)^{\sum\limits_{v\neq v_0} \delta_v} W_{\mgraph, \eps}(K).
    \]
    We also have
    \[
(-1)^{\sum\limits_{v\neq v_0} \delta_v}   (t^{\frac{1}{2}}z^{-1}-t^{-\frac{1}{2}}z)^{1-\delta_0}
 = (-1)^{1+\sum\limits_{v} \delta_v} (t^{-\frac{1}{2}}z - t^{\frac{1}{2}}z^{-1})^{1-\delta_0}
 =
 -(t^{-\frac{1}{2}}z - t^{\frac{1}{2}}z^{-1})^{1-\delta_0},
\]
which completes the proof. 
\end{proof}


\subsection{Invariance}
\label{sec:invariance}

In this section we establish invariance of the weighted graded root for knot complements. 
The following notation will be used throughout. 
Recall the maps $\alpha_{rel}$ from  \eqref{eqref:eq:alpha rel maps}. 
Fix $\eps \in \{\pm1\}$ and a pair of  negative definite marked plumbing graphs $\mgraph$ and $\mgraphp$ with $s+1$ and $s+2$ vertices, respectively, such that $\mgraphp$ is obtained from $\mgraph$ by one of the four Neumann moves. 
Recall from the discussion surrounding the diagram \eqref{eq:spinc maps Neumman moves} that the ambient plumbing graphs $\agraph = \mgraph\setminus\{v_0\}$ and $\agraphp = \mgraphp\setminus\{v_0\}$ transform according to one of the type \hyperlink{A}{(A)}, \hyperlink{B}{(B)}, or \hyperlink{C}{(C)} moves. 
We fix $[k]\in \spinc(\agraph)$ and $[k']\in \spinc(\agraphp)$ such that $[k'] = \beta([k])$, where $\beta$ is the corresponding map from \eqref{eq:A0 closed}, \eqref{eq:B0 closed}, \eqref{eq:C}.

It follows from  \cite{Nem_On_the,  Niemi-Colvin} that the graded roots $\root(\agraph, [k])$ and $\root(\agraphp, [k'])$ of the ambient plumbing graphs are isomorphic. 
For our proof of Theorem \ref{thm:invariance}, we would like to have a specific map of lattices which induces this isomorphism.

\begin{lem}
\label{lem:ambient graded root is invariant}
    For each of the four Neumann moves, the corresponding map $\beta_\pm$ from \eqref{eq:A0 closed lattice}, \eqref{eq:B0 closed lattice}, \eqref{eq:C lattice} induces an isomorphism  $\root(\agraph, [k]) \cong \root(\agraphp, [k'])$ of graded roots of the ambient plumbing graphs.
\end{lem}

\begin{proof}
    Each of the type \hyperlink{A}{(A)}, \hyperlink{A0}{(A0)}, and \hyperlink{B}{(B)} moves have the effect of transforming $\agraph$ according to the type \hyperlink{A}{(A)} or \hyperlink{B}{(B)} moves, so the statement is given by Lemma \ref{lem:beta1 and beta-1 induce isos on graded roots}. 
    
   It remains to verify the type \hyperlink{B0}{(B0)} move, which transforms the ambient plumbing graphs according to the type \hyperlink{C}{(C)}. 
   From \eqref{eq:how matrices transform}, we see that $h_U(\beta_\pm(K)) = h_U(K)$ for all $K\in \Char(\agraph)$. 
   Moreover, for $h\in h_U(k) + 2\Z$ and $1\leq i \leq s$, if $K,K+ 2\amatrix e_i \in \superlevel_h(\agraph, [k])$, then 
   \[
    \beta_{\pm }(K+ 2\amatrix e_i) = (K+ 2\amatrix e_i, \pm 1) = \beta_{\pm }(K) + 2 \amatrixp e_i.
   \]
   Therefore $\beta_\pm$ induces a map $\til{\beta}_\pm$ from the connected components of $\superlevelone_h(\agraph, [k])$ to the connected components of $\superlevelone_h(\agraphp, [k'])$, given by sending a component $C$ containing a characteristic vector $K$ to the component $\til{\beta}_\pm(C)$ containing $\beta_\pm(K)$. 
   We will show $\til{\beta}_\pm$ is a bijection. 
   First, for a vertex $K'$ in $\superlevel_h(\agraphp, [k'])$ and $1\leq i \leq s+1$, we have
   \begin{align}
   \begin{aligned}
   \label{eq:type C iso}
\hdual{(K' + 2\amatrixp e_{i})}  (\amatrixp)^{-1} (K' + 2\amatrixp e_{i})
& =\hdual{(K')} 
 (\amatrixp)^{-1} K' + 4 K' e_i +4 M'_{ii}, \\
 \hdual{(K' - 2\amatrixp e_{i})}  (\amatrixp)^{-1} (K' - 2\amatrixp e_{i})
& =\hdual{(K')} 
 (\amatrixp)^{-1} K' - 4 K' e_i +4 M'_{ii}.
 \end{aligned}
   \end{align}
   
   To show surjectivity of $\til{\beta}_\pm$, write $K' = (\b{K'}, K_{s+1})$ for $\b{K'} \in \Char(\agraph, [k])$, $K_{s+1} \in 1+2\Z$. 
   Applying \eqref{eq:type C iso} with $i=s+1$, we see that
   if $K_{s+1} \geq 1$ then 
   \[
   h \leq h_U(K') \leq h_U(K' + 2\amatrixp e_{s+1}) = h_U(K'- 2e_{s+1}),
   \]
   and if $K_{s+1} \leq -1$ then 
   \[
   h \leq h_U(K') \leq h_U(K' - 2\amatrixp e_{s+1}) = h_U(K'+ 2e_{s+1}).
   \]
   In either case, $K'$ is always connected by a sequence of edges lying inside $\superlevelone_h(\agraphp, [k'])$ to a vertex of $\superlevelone_h(\agraphp, [k'])$ that is in the image of $\beta_\pm$, which establishes surjectivity. 

    We now prove injectivity. 
    First, note that $h_U(K') \leq h_U(\b{K'})$. Setting $L = K' + 2M' e_i$, we have 
    \[
    \b{L} = 
    \begin{cases}
        \b{K'} + 2\amatrix e_i & \text{ if } i \leq s, \\
        \b{K'} & \text{ if } i = s+1
    \end{cases}
    \]
    Therefore any path of edges in $\superlevelone_h(\agraphp, [k'])$ projects to a path of edges in $\superlevelone_h(\agraph, [k])$, which demonstrates injectivity. 
\end{proof}

\begin{thm}
\label{thm:invariance}
For any  negative definite marked plumbing graph $\mgraph$, $\spinc$ structure $[k]\in \spinc(\agraph)$, admissible family of functions $W$, and fixed $\eps\in \{\pm 1\}$, the weighted graded root $\root_\eps(\mgraph, [k],W)$ for plumbed knot complements is invariant under Neumann moves. 
It follows that the weighted bigraded root $\rootbi_\eps(\mgraph, [k], W)$ is also invariant under Neumann moves. 
\end{thm}

\begin{proof}
Fix $i\in h_U(k) + 2\Z$, a connected component $C\subset \superlevelone_i(\agraph, [k])$, and set $C' = \til{\beta}_\pm(C)$, where $\til{\beta}_\pm$ is the induced map on components of superlevel sets as in the proof of Lemma \ref{lem:ambient graded root is invariant} (note $C'$ is independent of $\pm$ since $\beta_+$ and $\beta_-$ differ by an edge). 
For each of the four Neumann moves, we will show that 
\begin{equation}
\label{eq:invariance main equation}
W_{\mgraph}(C; z,q,t) = W_{\mgraphp}(C'; z,q,t).  
\end{equation}

Note that $h_U(\beta_\pm(K)) = h_U(K)$ for all $K\in \Char(\agraph)$, which is equivalent to 
\[
\hdual{(\beta_{\pm}(K))} (\amatrixp)^{-1}\beta_{\pm}(K) = \hdual{K} \amatrixp^{-1} K -1.
\]
Note also that the degree of the marked vertex changes only for the \hyperlink{B0}{(B0)} move, so  
\[
(t^{-\frac{1}{2}}z-t^{\frac{1}{2}}z^{-1})^{1-\delta_0} = (t^{-\frac{1}{2}}z-t^{\frac{1}{2}}z^{-1})^{1-\delta_0'}
\]
for the other three moves. 
For this reason, we will not mention this factor in the proof until the \hyperlink{B0}{(B0)} move. 
In the same spirit, due to how the ambient graphs transform, the proofs of invariance under the type \hyperlink{A}{(A)}, \hyperlink{A0}{(A0)}, and \hyperlink{B}{(B)} moves are similar in structure to the proof of \cite[Theorem 5.9]{AJK}, though the weights in the present paper are quite different. 
The \hyperlink{B0}{(B0)} move, left to the end, is the most technical.   \\

\noindent
\textbf{Type \hyperlink{A}{(A)}:} Let $K\in \Char(\agraph)$ and let $K' = \beta_{\eps}(K)$. 
We begin by showing 
\[
W_{\mgraph}(K) q^{\qexp(K)}z^{\zexp(K)}t^{\texp(K)} = W_{\mgraphp}(K') q^{\qexp(K')}z^{\zexp(K')}t^{\texp(K')} .
\]
First, observe that if $M^{-1}K = x = (x_1, x_2, \ldots, x_s)$ then $
  \amatrixp(x, x_1+x_2-\eps) = K'.$ 
We also have $\onevp = (\onev, 0)$, so
\begin{align*}
    \hdual{(K')}(\amatrixp)^{-1}\onevp = \hdual{(x, x_1+x_2-\eps)} (\onev, 0) =  \hdual{K}\amatrix^{-1}\onev.
\end{align*}
Similarly, if $\amatrix^{-1}\onev = y$ then $  M'(y, y_1+y_2) = \onevp$, which implies 
\begin{align*}
    \hdual{(\onevp)}(\amatrixp)^{-1}\onevp = \hdual{\onev} \amatrix^{-1}\onev.
\end{align*}
We also have 
\[
    \hdual{(\delta')} u =\hdual{ \delta} u +2, \ \  \sum_{v\neq v_0}m'_v = \sum_{v\neq v_0}m_v-3, \ \ \text{and } \hdual{(K')} u= \hdual{K}u -\eps. 
\]
These calculations together imply $\qexp(K) = \qexp(K'), \zexp(K) = \zexp(K')$, and $\texp(K) = \texp(K')$.  

It remains to verify $W_{\mgraph}(K) = W_{\mgraphp}(K')$. 
We have 
\[
K'-\eps(\onevp +\amatrixp u) = (K -\eps(\onev+\amatrix u),0),
\]
which, using the formula for $W_2$ from \eqref{eq:W_1 and W_0}, gives 
\[
W_{\mgraphp}(K')= W_{\mgraph}(K) \cdot W_2(0) =  W_{\mgraph}(K). 
\]

To finish invariance under the type \hyperlink{A}{(A)} move, we will show that any  characteristic vector $H\in C'$ which is not in the image of $\beta_\eps$ has weight zero. 
To that end, note that the $(s+1)$-st entry of $H- \eps(\onevp + \amatrixp u)$ is equal to $H_{s+1} - \eps$. 
Since $\delta'_{s+1} = 2$, we see that if $H_{s+1} \neq \eps$ then $W_{\agraphp, [b']}(H) = 0$. On the other hand, if $H_{s+1} = \eps$, then  we may take $\beta_\eps((H_1+\eps, H_2+\eps, H_3 \ldots, H_s)) = H$. 
Therefore equation \eqref{eq:invariance main equation} is established for the type \hyperlink{A}{(A)} move.\\

\noindent
\textbf{Type \hyperlink{A0}{(A0)}:} 
Let $K\in \Char(\agraph)$ and let $K' = \beta_{\eps}(K)$. We again begin by showing 
\[
W_{\mgraph}(K) q^{\qexp(K)}z^{\zexp(K)}t^{\texp(K)} = W_{\mgraphp}(K') q^{\qexp(K')}z^{\zexp(K')}t^{\texp(K')} .
\]
If $\amatrix^{-1}K = x$ then $ \amatrixp(x, x_1-\eps) = K' $. We also have $\onevp = (\onev, 0)+(-1,0, \ldots, 0, 1)$, which gives 
\[
\hdual{(K')}(\amatrixp)^{-1}\onevp =\hdual{K}M^{-1}\lambda-\eps.
\]
Further, if $\amatrix^{-1}\lambda = y$, then $\amatrixp(y, y_1-1) =  \onevp$, which implies
\[
 \hdual{(\onevp)}(\amatrixp)^{-1}\onevp  =\hdual{\lambda}\amatrix ^{-1}\onev-1.
\]
We also have 
\[
\hdual{(\delta')} u = \hdual{\delta}u +2, \ \ \sum\limits_{v\neq v_0}m'_v = \sum_{v\neq v_0}m_v-2, \ \ \text{and } \hdual{ (K')} u = \hdual{K} u.
\]
It follows that $\qexp(K) = \qexp(K'), \zexp(K) = \zexp(K')$, and $\texp(K) = \texp(K')$.

Next, we have
\[
K'-\eps(\onevp +\amatrixp u) = (K -\eps(\onev+\amatrix u),0),
\]
which gives 
\[
W_{\mgraph}(K')= W_{\mgraph}(K) \cdot W_2(0) =  W_{\mgraph}(K).
\]

To finish the proof of equation \eqref{eq:invariance main equation} in this case, just like in the type \hyperlink{A}{(A)} move we will show that any characteristic vector not in the image of $\beta_\eps$ has weight zero. 
Let $H\in [k']$. 
Then the $(s+1)$-st entry of $H- \eps(\onevp + \amatrixp u)$ is equal to $H_{s+1} - \eps$. 
Since $\delta'_{s+1} = 2$, we see that if $H_{s+1} \neq \eps$ then $W_{\agraphp}(H) = 0$. 
If $H_{s+1} = \eps$, then we may take $\beta_\eps((H_1+\eps, H_2, H_3 \ldots, H_s)) = H$, which establishes equation \eqref{eq:invariance main equation} for the type \hyperlink{A0}{(A0)} move.\\

\noindent
\textbf{Type \hyperlink{B}{(B)}:} The proof of this case will use both $\beta_+$ and $\beta_{-}$, even though our choice of $\eps$ is fixed. For $K\in C$, set
\begin{align*}
    K_+' &= \beta_+(K) =  (K, 0) + (-1, 0,\ldots, 0, 1),\\
    K_{-}'  &= \beta_{-}(K) = (K,0) + (1, 0,\ldots, 0, -1).
\end{align*}

Our first goal for the type \hyperlink{B}{(B)} move is to show that 
\begin{equation}
\label{eq:type B invariance equation}
W_{\mgraph}(K) q^{\qexp(K)}z^{\zexp(K)}t^{\texp(K)} 
= W_{\mgraphp}(K_+') q^{\qexp(K'_{+})}z^{\zexp(K'_{+})}t^{\texp(K'_{+})} 
+ W_{\mgraphp}(K_{-}') q^{\qexp(K'_{-})}z^{\zexp(K'_{-})}t^{\texp(K'_{-})}.
\end{equation}
To that end, if $\amatrix^{-1}K = x$, then $ \amatrixp(x, x_1\mp 1) = K_\pm'$. 
Since $\onevp = (\onev, 0)$, we have $\hdual{(K'_\pm)}(\amatrixp)^{-1}\onevp = \hdual{K}\amatrix^{-1}\onev$. 
Similarly, if $\amatrix^{-1}\onev = y$, then $\amatrixp(y, y_1) = \onevp$, which gives
\[
\hdual{(\onevp)}(\amatrixp)^{-1}\onevp  =\hdual{\onev}\amatrix^{-1}\onev.
\]
We also have \[
\hdual{(\delta')} u = \hdual{\delta}u + 2, \ \ \sum\limits_{v\neq v_0}m'_v = \sum_{v\neq v_0}m_v-2, \ \text{and \ }   \hdual{(K_\pm')} u = \hdual{K} u. 
\]
These computations imply that $\qexp(K) = \qexp(K'_+)=\qexp(K'_{-})$, $\zexp(K) =  \zexp(K'_+) = \zexp(K'_{-})$, and $\texp(K) = \texp(K'_+) = \texp(K'_{-})$.

Next, setting $\til{K} = K -\eps(\lambda+Mu)$, we have
\begin{align*}
K_+' - \eps(\onevp + \amatrixp u )
&=
(\til{K},0)+(-1,0,\ldots, 0,1), \\
K_{-}' - \eps(\onevp + \amatrixp u)
&=
(\til{K},0)+(1,0,\ldots, 0,-1). 
\end{align*}
Using the formula for $W_1$ from \eqref{eq:W_1 and W_0}, it follows that
\begin{align*}
    W_{\mgraphp}(K_+') &=W_{\delta'_1}(\til{K}_1-1)W_{\delta'_{s+1}}(1)\prod_{i=2}^{s}W_{\delta'_i}(\til{K}_i) \\
    &=W_{\delta_1+1}(\til{K}_1-1)W_{1}(1)\prod_{i=2}^{s}W_{\delta_i}(\til{K}_i)
    =-W_{\delta_1+1}(\til{K}_1-1)\prod_{i=2}^{s}W_{\delta_i}(\til{K}_i),
\end{align*}
and similarly,  $W_{\mgraphp}(K_{-}') = W_{\delta_1+1}(\til{K}_1+1)\prod_{i=2}^{s}W_{\delta_i}(\til{K}_i)$. 
Property \ref{item:AD2} gives
\[
W_{\delta_1+1}(\til{K}_1+1)-W_{\delta_1+1}(\til{K}_1-1) =W_{\delta_1}(\til{K}_1),
\]
which implies $W_{\mgraph}(K) = W_{\mgraphp}(K_+') + W_{\mgraphp}(K_{-}')$. Together with the earlier computations of the $q,z$, and $t$ exponents we arrive at \eqref{eq:type B invariance equation}. 

To finish the proof of invariance under the \hyperlink{B}{(B)} move, we will show that  $H\in [k']$ has weight zero unless it is in the image of $\beta_+$ or $\beta_{-}$. 
To see this, observe that the $(s+1)$-st entry of $H - \eps(\onevp + \amatrixp u)$ is equal to $H_{s+1}$. Since $\delta'_{s+1} =1$, we see that the weight of $H$ is zero unless $H_{s+1} = \pm 1$. 
On the other hand, $\beta_{\pm }((H_1\pm 1, H_2, H_3 \ldots, H_s)) = H$, which completes the proof of invariance under the type \hyperlink{B}{(B)} move.
\\

\noindent
\textbf{Type \hyperlink{B0}{(B0)}:} Let $K\in C$. For this move we will again use both maps $\beta_+$ and $\beta_{-}$. As in the \hyperlink{B}{(B)} move, set
\begin{align*}
    K_+' &= \beta_+(K) =  (K,1 ), \\
    K_{-}'  &= \beta_{-}(K) = (K,-1).
\end{align*}
Our first goal is to show that 
\begin{equation}
\label{eq:B1 equality}
\begin{split}
 & \left(t^{-\frac{1}{2}}z-t^{\frac{1}{2}}z^{-1}\right)^{1-\delta_0} W_{\mgraph}(K)q^{\qexp(K)}z^{\zexp(K)}t^{\texp(K)}  = \\ 
& ( t^{-\frac{1}{2}}z-t^{\frac{1}{2}}z^{-1})^{1-\delta'_0}
\left[W_{\mgraphp}(K'_{-})q^{\qexp(K'_{-})}z^{\zexp(K'_{-})}t^{\texp(K'_{-})}   +  W_{\mgraphp}(K'_{+})q^{\qexp(K'_{+})}z^{\zexp(K'_{+})}t^{\texp(K'_{+})} \right].
\end{split}
\end{equation}

To start, $(\amatrixp)^{-1} K'_{\pm} = (\amatrix^{-1}K, \mp 1)$. Since $\onevp = (\onev,1)$, this gives 
\begin{align*}
    \hdual{(K'_\pm)}(\amatrixp)^{-1}\onevp  &= \hdual{K}\amatrix^{-1}\onev\mp 1, \\ \hdual{(\lambda')}(\amatrixp)^{-1}\onevp & =\hdual{\onev}\amatrix^{-1}\onev-1.
\end{align*}
We also have
\[
\hdual{( \delta')} u = \hdual{\delta}u + 2, \ \ \sum_{v\neq v_0}m'_v = \sum_{v\neq v_0}m_v-1, \ \text{ and } \hdual{(K_\pm')} u = \hdual{K}u \pm 1. 
\]
The above calculations imply that 
\[
\qexp(K) = \qexp(K'_{+}) = \qexp(K'_{-}). 
\]
Next, observe that $\delta_0' = \delta_0+1$ in this move, which together with the above equalities gives 
\begin{align*}
\zexp(K'_+) + 1 = \zexp(K) = \zexp(K'_{-}) - 1, \\
\texp(K'_+)  -\frac{1}{2} = \texp(K) = \texp(K_{-}') + \frac{1}{2}. 
\end{align*}
Next, as in the proof of the type \hyperlink{B}{(B)} move, if we set $\til{K} = K-\eps(\onev+Mu)$, then  
\[
K_+'-\eps(\onevp +M'u) = (\til{K},1) \text{ and } K_{-}'-\eps(\onevp +M'u) = (\til{K},-1).
\]
Since $\delta_{s+1} = 1$, it follows that
\begin{align*}
    W_{\mgraphp}(K_+') & = W_{1}(1)\prod_{i=1}^{s}W_{\delta'_i}(\til{K}_i) = -W_{\mgraph}(K), \\
    W_{\mgraphp}(K_{-}') & = W_{1}(-1)\prod_{i=1}^{s}W_{\delta'_i}(\til{K}_i) = W_{\mgraph}(K),
\end{align*}
which, together with $\delta_0' = \delta_0 +1$ and the earlier calculations of the $q, z$, and $t$, exponents, establishes equation \eqref{eq:B1 equality}. 

To finish the proof of invariance under the \hyperlink{B0}{(B0)} move, we will show that any $H\in [k']$ contributes zero to the weight of $C'$ unless $H$ is in the image of $\beta_+$ or $\beta_{-}$. 
To that end, we see that the $(s+1)$-st entry of $H- \eps(\onevp + \amatrixp u)$ is equal to $H_{s+1}$, so that the weight of $H$ is equal to zero unless $H_{s+1} = \pm 1$, in which case we have $\beta_{\pm}(H_1, \ldots, H_s) = H$. 
This verifies equation \eqref{eq:invariance main equation} for the type \hyperlink{B0}{(B0)} move and concludes the proof of the first part of the theorem. 

Invariance of the weighted bigraded root follows almost immediately. Theorem \ref{thm:bigraded root invariance} states that the bigraded root is invariant under Neumann moves. The weights of all nodes in a fixed $U$-coordinate in $\rootbi_\eps(\mgraph, [k], W)$ are  equal to the  weight of the corresponding node in $\root_\eps(\mgraph, [k],W)$. Invariance of $\rootbi_\eps(\mgraph, [k], W)$ then follows from invariance of $\root_\eps(\mgraph, [k],W)$.  
\end{proof}

\begin{rem}
\label{rem:weighted graded root with b}

One could define weights based on $ \widehat{Z}_{[b]}$ from \eqref{eq:Zhatz0naive}, which depends on a relative $\spinc$ structure $[b]$, rather than based on $\Zhat_{[b\vert_\Gamma]}$ from \eqref{eq:Zhatz0}, which depends only on the $\spinc$ structure $\relmap_\eps([b])$. 
The only modification is to use $\zexp_{[b]}$ from \eqref{eq:z exp b} as the exponent of $z$. 
The resulting weighted graded root is also an invariant under Neumann moves, where the relative $\spinc$ structures transform according to the maps $\alpha_{rel}$ given in \eqref{eqref:eq:alpha rel maps}. 
Letting $b'= \alpha_{rel}(b)$, for the type \hyperlink{A}{(A)}, \hyperlink{A}{(A0)}, and \hyperlink{B}{(B)} move we have $\hdual{\onevp}\amatrixp^{-1}(b'\vert_{\agraphp} ) = \hdual{\onev} \amatrix^{-1}(b\vert_{\agraph})$ and $b_0' = b_0$, while for the type \hyperlink{B0}{(B0)} move we have $\hdual{\onevp}\amatrixp^{-1}(b'\vert_{\agraphp}) = \hdual{\onev} \amatrix^{-1}(b\vert_{\agraph}) + 1$ and $b_0' = b_0 +1$. 
Since $\zexp_{[b]} = \zexp -\hdual{\onev}M^{-1}(b\vert_{\agraph} )+b_0$, invariance of this alternative weighted graded root then follows from the proof of Theorem \ref{thm:invariance}. 
Also, it is a straightforward calculation to see that if $[b_1], [b_2] \in \spinc(\mgraph)$  satisfy $\relmap_\eps([b_1]) = \relmap_\eps([b_2])$, then for any $K \in \relmap_\eps([b_2])$ 
\[
\zexp_{[b_1]}(K) = \zexp_{[b_2]}(K) + 2r,
\]
where $r\in \Z$ is some fixed integer. 
Therefore  different relative $\spinc$ structures in the same fiber of $\relmap_\eps$ lead to power series which differ by some even power of $z$. 

\end{rem}


\section{Surgery formula for weighted graded roots}\label{sec:surgery_formula}

Let $\mgraph$ be a negative definite marked plumbing tree with $s+1$ vertices, and let $m_0\in \Z$ be a framing on $v_0$ such that the surgered graph $\sgraph$ is negative definite. 
As usual, we set $\agraph = \mgraph \setminus \{v_0\}$. 
In Section \ref{sec:surgery formula for bigradd roots} we described how to obtain the graded root for $\sgraph$ at a $\spinc$ structure $\mathfrak{t}\in \spinc(\sgraph)$ from the bigraded roots of $\mgraph$. 
A surgery formula for the $q$-series $\Zhat$ of the closed manifold $\smfld$ in terms of the $q,z$-series of the knot complement represented by $\mgraph$ was given in \cite[Theorem 1.2 and Section 6.8]{GM}. 
We note also that a more general gluing formula was provided in \cite[Section 6.3]{GM}. 
In this section we unify and refine the surgery formulas: namely, we describe how to obtain the weighted graded roots of $\sgraph$ from the weighted bigraded roots of $\mgraph$.

As in Section \ref{sec:surgery formula for bigradd roots}, in this section we will use $L$ to denote a characteristic vector of $\sgraph$ and $K$ to denote a characteristic vector of $\agraph$. 
We introduce the following additional notation. 
For $L\in \Char(\sgraph)$, set
\begin{align*}
&W_{\sgraph}(L;q,t)\\
&= q^{  -\frac{3(s+1) + \sum m_v + (L- \eps \smatrix u)^2}{4} }t^{\frac{\hdual{L}u - \eps \hdual{u}\smatrix u}{2}} W_{\sgraph}(L) \\
&= 
q^{-\frac{3(s+1) + \sum m_v}{4}}
\left. 
\left( 
\oint \prod_{v}\frac{dz_v}{2\pi i z_v} (t^{-1/2}z_v - t^{1/2}z_v^{-1})^{2-\delta_v} q^{- \frac{\hdual{\ell}\smatrix^{-1} \ell}{4}} \prod_v z_v^{\ell_v} 
\right)
\right\rvert_{\ell = L - \eps \smatrix u}.
\end{align*}
Throughout this section, $(t^{-1/2}z_v - t^{1/2}z_v^{-1})^{2-\delta_v}$ is expanded according to \eqref{eq:expansion} via the substitution $z\mapsto t^{-1/2}{z_v}$, and the integral is interpreted as recording the constant term of the integrand as discussed in Section \ref{sec:Zhat closed}. 

Note that $W_{\sgraph}(L;q,t)$ differs from $W_{\sgraph}(L)$ in that the former includes also the monomial in $q$ and $t$ that $L$ contributes. 
The weight in \eqref{eq:weights in weighted graded root} is then obtained by summing $W_{\sgraph}(L;q,t)$ over the $0$-cells of a connected component $C$. 
Analogously, for $K\in \Char(\agraph)$, define
\begin{align*}
W_{\mgraph}(K; z,q,t) & = 
 (t^{-\frac{1}{2}}z - t^{\frac{1}{2}}z^{-1})^{1-\delta_0}  W_{\mgraph}(K)q^{\qexp(K)}z^{\zexp(K)}t^{\texp(K)}\\
&=
(t^{-\frac{1}{2}}z - t^{\frac{1}{2}}z^{-1})^{1-\delta_0} q^{-\frac{3s + \sum_{v\neq v_0}m_v}{4}} \times \\
& 
\;\;  
\left(
\oint \prod_{v\neq v_0}\frac{dz_v}{2\pi i z_v} (t^{-1/2}z_v - t^{1/2}z_v^{-1})^{2-\delta_v} q^{- \frac{\hdual{\ell}\amatrix^{-1} \ell}{4}} \prod_{v\neq v_0} z_v^{\ell_v} 
\right)
\Bigg\rvert_{\ell = K - \eps(\onev + \amatrix u)}.
\end{align*}
Recall that $\sf = \hdual{\onev}\amatrix^{-1}\onev\in \Q$ is the (rational) Siefert framing and that
\[
p = \Sigma^2 = m_0 - \hdual{\onev}\amatrix^{-1}\onev = m_0 - \sf 
\in \Q.
\]

\begin{lem}
\label{lem:surgery lattice point}
Let $L\in \Char(\sgraph, \mathfrak{t})$ with $a= a(L)$, and let $K = L\vert_{\agraph}  \in \Char(\mgraph, \mathfrak{t}_a)$. Then
\[
W_{\sgraph}(L;q,t) 
=
\oint
\frac{dz}{2\pi i z} (t^{-\frac{1}{2}}z - t^{\frac{1}{2}}z^{-1})
W_{\mgraph}(K;z,q,t)\;z^{2(a-\frac{1+\eps}{2}p)}q^{\frac{-3-p}{4}} q^{-\frac{1}{p}(a - \frac{1+\eps}{2}p)^2}.
\]
\end{lem}

\begin{proof}
    We  first establish the following regarding the $q$-power in $W_{\sgraph}(L;q,t)$:
    \begin{equation}
    \label{eq:surgery - q exp of L and its restriction}
   -\dfrac{3(s+1) + \sum m_v + (L- \eps \smatrix u)^2}{4} 
   =
   -\dfrac{3+p}{4}
   - \dfrac{\left(a- \frac{1+\eps}{2}p\right)^2}{p}
   +
   \qexp(K). 
\end{equation}
We have
    \begin{align*}
        (L-\eps \smatrix u )^2 & = K^2 + \dfrac{(p-2a)^2}{p} - 2\eps \hdual{u}L +  \hdual{u} \smatrix u \\
        &=K^2 + \dfrac{(p-2a)^2}{p} - 2\eps \hdual{u}K  -2\eps L_0 +  m_0 + \hdual{\delta } u  + \sum_{v\neq v_0} m_v, \\
    \end{align*}
where in the first equality we use Lemma \ref{lem:L^2 and K^2}. 
From the proof of Lemma \ref{lem:L^2 and K^2}, we also have
\[
L_0 = 2a - m_0 + \hdual{\onev}\amatrix^{-1}\onev + \hdual{K}\amatrix^{-1} \onev = 2a - p +  \hdual{K}\amatrix^{-1}\onev.
\]
The left-hand side of \eqref{eq:surgery - q exp of L and its restriction} is then
\begin{align*}
-\dfrac{3+2m_0}{4} - \dfrac{(p-2a)^2}{4 p} + \dfrac{\eps(2a-p)}{2} +\left[-\dfrac{3s + \hdual{\delta} u+ 2\sum_{v\neq v_0} m_v}{4}  -\dfrac{K^2}{4}  + \dfrac{\eps \hdual{K}u}{2}  + \dfrac{\eps \hdual{K} \amatrix^{-1} \onev}{2} \right].
\end{align*}
Note that the term in square brackets above is equal to $
\qexp(K) + \frac{\hdual{\onev}\amatrix^{-1} \onev}{2}$.
On the right-hand side of \eqref{eq:surgery - q exp of L and its restriction}, we have
\begin{align*}
    \dfrac{\left(a- \frac{1+\eps}{2}p\right)^2}{p} 
    =
    \dfrac{\left(p-2a \right)^2}{4p} - \dfrac{\eps(2a-p)}{2} + \dfrac{m_0 - \hdual{\onev} \amatrix^{-1} \onev}{4},
\end{align*}
so that 
\[
-\dfrac{3+p}{4}
- \dfrac{\left(a- \frac{1+\eps}{2}p\right)^2}{p} 
= 
- \dfrac{3+2m_0}{4} -\dfrac{\left(p-2a \right)^2}{4p} + \dfrac{\eps(2a-p)}{2}  +\dfrac{ \hdual{\onev} \amatrix^{-1} \onev}{2},
\]
which completes the proof of \eqref{eq:surgery - q exp of L and its restriction}. 
This together with the straightforward computation  
\[
\zexp(K) 
=
\onev \amatrix^{-1}(K-\eps(\onev + \amatrix u)) 
=
-2\left(a- \frac{1+\eps}{2}\Sigma^2\right) + L_0 - \eps(m_0 + \delta_0)
\]
implies the statement of the lemma.
\end{proof}

\begin{defn}
Define the \emph{Laplace transform} $\mathcal{L}_p^{(a, \eps)}$ to be 
\begin{align*}
\mathcal{L}_p^{(a, \eps)} : 
z^u 
\;\mapsto\;
&\oint
\frac{dz}{2\pi i z} (t^{-\frac{1}{2}}z - t^{\frac{1}{2}}z^{-1})
z^u\;z^{2(a-\frac{1+\eps}{2}p)}q^{\frac{-3-p}{4}} q^{-\frac{1}{p}(a - \frac{1+\eps}{2}p)^2}\\
&\quad = 
\begin{cases}
    t^{-\frac{1}{2}}q^{-\frac{1}{p}(a - \frac{1+\eps}{2}p)^2 + \frac{-3-p}{4}} &\text{if } u = -2(a-\frac{1+\eps}{2}p) - 1\\
    -t^{\frac{1}{2}}q^{-\frac{1}{p}(a - \frac{1+\eps}{2}p)^2 + \frac{-3-p}{4}} &\text{if } u = -2(a-\frac{1+\eps}{2}p) + 1\\
    0 &\text{otherwise}
\end{cases},
\end{align*}
and extend linearly. 
\end{defn}

Combining Lemma \ref{lem:surgery lattice point} with the proof of invariance of weighted (bi)graded root (Theorem \ref{thm:invariance}) and the surgery algorithm for (bi)graded roots (Proposition \ref{prop:bigraded root algorithm}),
we immediately have: 
\begin{thm}
\label{thm:p surgery}
The weighted graded root of $(\sgraph,\mathfrak{t})$ is determined by the weighted bigraded roots of $\mgraph$, according to the following algorithm:
\begin{enumerate}
\item Consider the weighted graded graph 
\[
\bigsqcup\limits_{a \in \coset(\mathfrak{t})}
\mathcal{L}_p^{(a,\eps)}\left[\root_{\eps}(\mgraph, \mathfrak{t}_a, W)\right]\{-\shift(a)\},
\]
where $\{d\}$ denotes up upwards grading shift by $d$, 
and we are applying the Laplace transform $\mathcal{L}_p^{(a,\eps)}$ to the individual weights of the weighted graded root. 
\item For each pair of nodes $\eta_1$ of 
\[
\mathcal{L}_p^{(a,\eps)}\left[\root_{\eps}(\mgraph, \mathfrak{t}_a, W)\right]\{-\shift(a)\}
\]
and $\eta_2$ of 
\[
\mathcal{L}_p^{(a+\Sigma^2,\eps)}\left[\root_{\eps}(\mgraph, \mathfrak{t}_{a+\Sigma^2}, W)\right]\{-\shift(a+\Sigma^2)\}
\]
which are in the same grading, we identify $\eta_1$ and $\eta_2$ if there is a node in the bigraded root $\rootbi(\mgraph, \mathfrak{t}_{a})$ with  coordinate $(\eta_1, \eta_2)$ (see Definition \ref{def:coordinates}). 
When we identify the nodes, we add up the weights. 
After all of these identifications, we remove multiple edges connecting the same pair of vertices. 
\end{enumerate}
\end{thm}

Let us explain how Theorem \ref{thm:p surgery} yields Gukov-Manolescu's Dehn surgery formula \cite[Theorem 1.2]{GM} (see in particular \cite[Section 6.8]{GM}) in a special limit. 
First, from \eqref{eq:Sigma} we see that 
\begin{equation}
    \label{eq:p}
    p = \dfrac{1}{e_0 \smatrix^{-1} e_0}.
\end{equation}
For the following discussion, as in \cite[Section 6.8]{GM}, assume that the ambient manifold $Y$ is an integer homology sphere. 
In this case, $\hdual{\onev} \amatrix^{-1} \onev = m_0 - p$ is an integer, and performing $p$ surgery on $\knot \subset \amfld$ yields precisely $\smfld$.

\begin{lem}
\label{lem:ZHS bijection}
 If $\amfld$ is an integer homology sphere, then the map
\begin{align*}
    \spinc(\smfld) \to \Z/p\Z
\end{align*}
given by $[L]\mapsto a(L) \bmod{p}$ is a well-defined bijection. 
\end{lem}
\begin{proof}

Note that $\Sigma$ is an integer vector since the ambient manifold is an integer homology sphere. Therefore, since $L$ is characteristic, $L(\Sigma)+\Sigma^2$ is even. 
Next, suppose $[L] = [L']$, so that $L' = L +2M_{v_0, m_0}x$ for some $x\in\Z^{s+1}$. We have
 \[
\frac{L'(\Sigma)+p}{2} 
= \frac{L(\Sigma)+p}{2} + \hdual{x} M_{v_0, m_0}\Sigma
= \frac{L(\Sigma)+p}{2} +px_0,
\]
so the map is well-defined. For any $L\in \Char(\sgraph)$, $[L +2ne_0]$ maps to 
\begin{align*}
     \left(a(L)+n\right) \bmod{p}.
\end{align*}
So by varying $n$, we see that this map is surjective. By cardinality considerations this map must also be a bijection.   
\end{proof}

\begin{rem}
    Lemma \ref{lem:ZHS bijection} is analogous to \cite[Lemma 2.2]{ozsvath2008knot}
\end{rem}

\begin{rem}
    We point out that the map in Lemma \ref{lem:ZHS bijection} is a well-defined bijection in the more general setting where $\knot$ is nullhomologous in $Y$, since in this case the Seifert framing $\sf = \hdual{\onev} \amatrix^{-1} \onev$ is an integer as well. 
\end{rem}

Lemma \ref{lem:ZHS bijection} tells us that
$\coset(\mathfrak{t})$ is a coset of $p\mathbb{Z}$ inside $\mathbb{Z}$ determined by the $\spinc$ structure $\mathfrak{t} \in \spinc(Y_{v_0, m_0}) \cong \mathbb{Z}/p\mathbb{Z}$. 
It follows that, in the $h_U \rightarrow -\infty$ limit, 
Theorem \ref{thm:p surgery} says that
the BPS $q$-series $\widehat{Z}_{\mathfrak{t}}(q,t)$ of the surgered 3-manifold $Y_{v_0, m_0}$ can be obtained as
\begin{align*}
\widehat{Z}_{\mathfrak{t}}(q,t) 
&= \sum_{a\in \coset(\mathfrak{t})} \mathcal{L}_p^{(a,\eps)}[\widehat{Z}(z,q,t)]\\
&= \oint \frac{dz}{2\pi i z} (t^{-\frac{1}{2}}z - t^{\frac{1}{2}}z^{-1})
\widehat{Z}(z,q,t)\;
\sum_{a\in \coset(\mathfrak{t})}z^{2(a-\frac{1+\eps}{2}p)}q^{\frac{-3-p}{4}} q^{-\frac{1}{p}(a - \frac{1+\eps}{2}p)^2}\\
&= \oint \frac{dz}{2\pi i z} (t^{-\frac{1}{2}}z - t^{\frac{1}{2}}z^{-1})
\widehat{Z}(z,q,t)\;
\sum_{a\in \coset(\mathfrak{t})}z^{2 a}q^{\frac{-3-p}{4}} q^{-\frac{a^2}{p}},
\end{align*}
where $\widehat{Z}(z,q,t)$ is the BPS $q$-series of the plumbed knot complement (with respect to the unique $[b\vert_\Gamma]$). 
The last expression is exactly the integer surgery formula of Gukov-Manolescu.



\begin{exmp}[$-1$-surgery on trefoil]
Figure \ref{fig:trefoil_-1_surgery_with_weights} illustrates how the weighted graded root of $S^3_{-1}(\mathbf{3}_1) = \Sigma(2,3,7)$ can be recovered from the weighted (bi-)graded root of $\mathbf{3}_1$ (Figure \ref{fig:trefoil weighted bigraded root}) using the surgery formula. 
We have set $t=1$ for simplicity. 
\end{exmp}

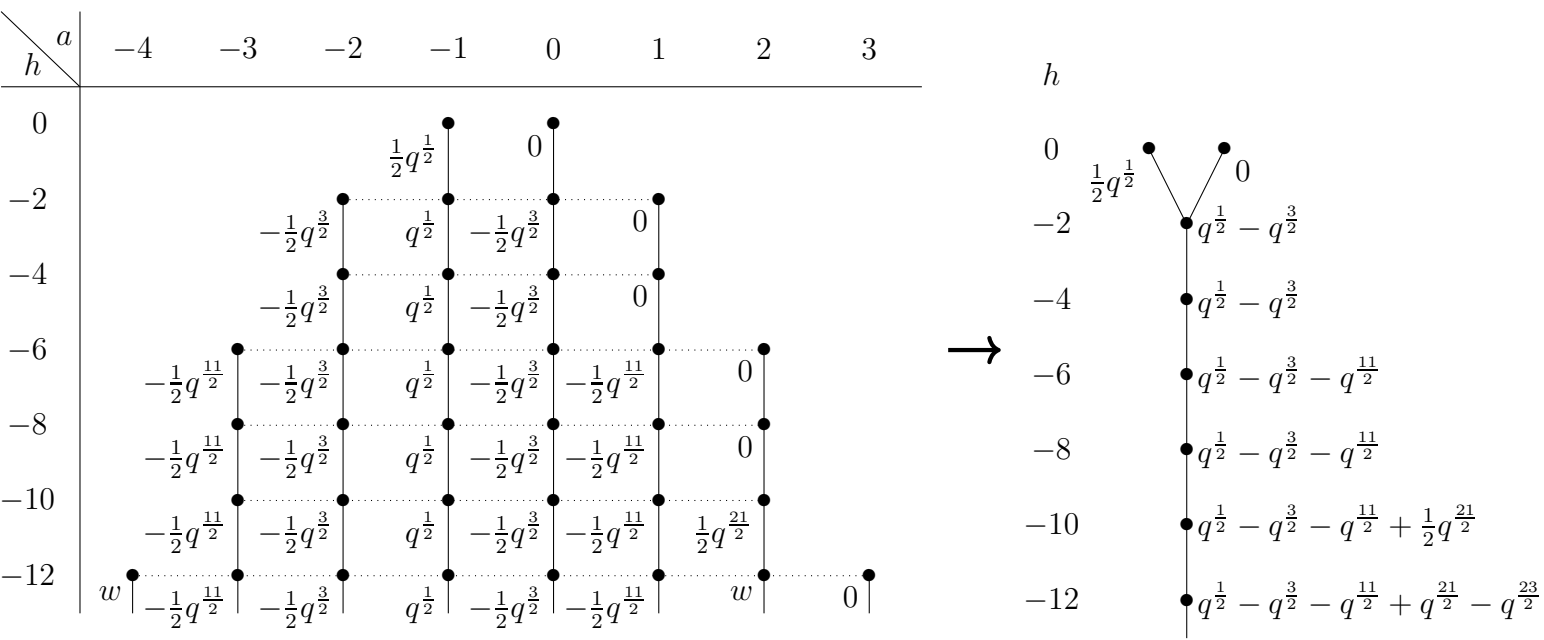
\begin{figure}
    \centering
    \rotatebox{90}{
\begin{tikzpicture}[yscale=0.5, xscale=.7]
\node at (0,0) {$\bullet$}; \node[anchor = north east  ] at (0,0) {$0$};
\node at (0,-2) {$\bullet$}; \node[anchor = north east  ] at (0,-2) {$-\frac{1}{2}q^{\frac{3}{2}}$};
\node at (0,-4) {$\bullet$}; \node[anchor = north east  ] at (0,-4) {$-\frac{1}{2}q^{\frac{3}{2}}$};
\node at (0,-6) {$\bullet$}; \node[anchor = north east  ] at (0,-6) {$-\frac{1}{2}q^{\frac{3}{2}}$};
\node at (0,-8) {$\bullet$}; \node[anchor = north east  ] at (0,-8) {$-\frac{1}{2}q^{\frac{3}{2}}$};
\node at (0,-10) {$\bullet$}; \node[anchor = north east  ] at (0,-10) {$-\frac{1}{2}q^{\frac{3}{2}}$};
\node at (0,-12) {$\bullet$}; \node[anchor = north east] at (0,-12) {$-\frac{1}{2}q^{\frac{3}{2}}$};

\node at (-2,0) {$\bullet$}; \node[anchor = north east] at (-2,0) {$\frac{1}{2}q^{\frac{1}{2}}$};
\node at (-2,-2) {$\bullet$}; \node[anchor = north east] at (-2,-2) {$q^{\frac{1}{2}}$};
\node at (-2,-4) {$\bullet$}; \node[anchor = north east] at (-2,-4) {$q^{\frac{1}{2}}$};
\node at (-2,-6) {$\bullet$}; \node[anchor = north east] at (-2,-6) {$q^{\frac{1}{2}}$};
\node at (-2,-8) {$\bullet$}; \node[anchor = north east] at (-2,-8) {$q^{\frac{1}{2}}$};
\node at (-2,-10) {$\bullet$}; \node[anchor = north east] at (-2,-10) {$q^{\frac{1}{2}}$};
\node at (-2,-12) {$\bullet$}; \node[anchor = north east] at (-2,-12) {$q^{\frac{1}{2}}$};

\node at (2,-2) {$\bullet$}; \node[anchor = north east] at (2,-2) {$0$};
\node at (2,-4) {$\bullet$}; \node[anchor = north east] at (2,-4) {$0$};
\node at (2,-6) {$\bullet$}; \node[anchor = north east] at (2,-6) {\scalebox{1}{$-\frac{1}{2}q^{\frac{11}{2}}$}};
\node at (2,-8) {$\bullet$}; \node[anchor = north east] at (2,-8) {\scalebox{1}{$-\frac{1}{2}q^{\frac{11}{2}}$}};
\node at (2,-10) {$\bullet$}; \node[anchor = north east] at (2,-10) {\scalebox{1}{$-\frac{1}{2}q^{\frac{11}{2}}$}};
\node at (2,-12) {$\bullet$}; \node[anchor = north east] at (2,-12) {\scalebox{1}{$-\frac{1}{2}q^{\frac{11}{2}}$}};

\node at (-4,-2) {$\bullet$}; \node[anchor = north east] at (-4,-2) {$-\frac{1}{2}q^{\frac{3}{2}}$};
\node at (-4,-4) {$\bullet$}; \node[anchor = north east] at (-4,-4) {$-\frac{1}{2}q^{\frac{3}{2}}$};
\node at (-4,-6) {$\bullet$}; \node[anchor = north east] at (-4,-6) {$-\frac{1}{2}q^{\frac{3}{2}}$};
\node at (-4,-8) {$\bullet$}; \node[anchor = north east] at (-4,-8) {$-\frac{1}{2}q^{\frac{3}{2}}$};
\node at (-4,-10) {$\bullet$}; \node[anchor = north east] at (-4,-10) {$-\frac{1}{2}q^{\frac{3}{2}}$};
\node at (-4,-12) {$\bullet$}; \node[anchor = north east] at (-4,-12) {$-\frac{1}{2}q^{\frac{3}{2}}$};

\node at (4,-6) {$\bullet$}; \node[anchor = north east] at (4,-6) {$0$};
\node at (4,-8) {$\bullet$}; \node[anchor = north east] at (4,-8) {$0$};
\node at (4,-10) {$\bullet$}; \node[anchor = north east] at (4,-10) {\scalebox{1}{$\frac{1}{2}q^{\frac{21}{2}}$}};
\node[] at (4,-12) {$\bullet$}; \node[anchor = north east] at (4,-12) {\scalebox{1}{$w$}};

\node at (-6,-6) {$\bullet$}; \node[anchor = north east] at (-6,-6) {\scalebox{1}{$-\frac{1}{2}q^{\frac{11}{2}}$}};
\node at (-6,-8) {$\bullet$}; \node[anchor = north east] at (-6,-8) {\scalebox{1}{$-\frac{1}{2}q^{\frac{11}{2}}$}};
\node at (-6,-10) {$\bullet$}; \node[anchor = north east] at (-6,-10) {\scalebox{1}{$-\frac{1}{2}q^{\frac{11}{2}}$}};
\node at (-6,-12) {$\bullet$}; \node[anchor = north east] at (-6,-12) {\scalebox{1}{$-\frac{1}{2}q^{\frac{11}{2}}$}};

\node at (6,-12) {$\bullet$}; \node[anchor = north east] at (6,-12) {$0$};

\node[] at (-8,-12) {$\bullet$}; \node[anchor = north east] at (-8,-12) {\scalebox{1}{$w$}};

\draw (0,0)--(0,-2)--(0,-4)--(0,-6)--(0,-8)--(0,-10)--(0,-12)--(0,-13);
\draw (-2,0)--(-2,-2)--(-2,-4)--(-2,-6)--(-2,-8)--(-2,-10)--(-2,-12)--(-2,-13);
\draw (2,-2)--(2,-4)--(2,-6)--(2,-8)--(2,-10)--(2,-12)--(2,-13);
\draw (-4,-2)--(-4,-4)--(-4,-6)--(-4,-8)--(-4,-10)--(-4,-12)--(-4,-13);
\draw (4,-6)--(4,-8)--(4,-10)--(4,-12)--(4,-13);
\draw (-6,-6)--(-6,-8)--(-6,-10)--(-6,-12)--(-6,-13);
\draw (6,-12)--(6,-13);
\draw (-8,-12)--(-8,-13);

\node at (0,2) {$0$};
\node at (-2,2) {$-1$};
\node at (2,2) {$1$};
\node at (-4,2) {$-2$};
\node at (4,2) {$2$};
\node at (-6,2) {$-3$};
\node at (6,2) {$3$};
\node at (-8,2) {$-4$};

\node at (-10,0) {$\phantom{-}0$};
\node at (-10,-2) {$-2$};
\node at (-10,-4) {$-4$};
\node at (-10,-6) {$-6$};
\node at (-10,-8) {$-8$};
\node at (-10,-10) {$-10$};
\node at (-10,-12) {$-12$};

\draw[] (-10.5,1)--(7,1);
\draw[] (-9,3)--(-9,-13);
\draw[] (-9,1)--(-10.5,3);

\node at (-9.3,2.3) {$a$};
\node at (-9.9,1.6) {$h$};

\draw [dotted] (-4,-2)--(2,-2);
\draw [dotted] (-4,-4)--(2,-4);
\draw [dotted] (-6,-6)--(4,-6);
\draw [dotted] (-6,-8)--(4,-8);
\draw [dotted] (-6,-10)--(4,-10);
\draw [dotted] (-8,-12)--(6,-12);

\draw[ultra thick, ->] (7.5,-6)--(8.5,-6);
\end{tikzpicture}

\begin{tikzpicture}[scale=0.5]
\node at (-3,0) {$0$};
\node at (-3,-2) {$-2$};
\node at (-3,-4) {$-4$};
\node at (-3,-6) {$-6$};
\node at (-3,-8) {$-8$};
\node at (-3,-10) {$-10$};
\node at (-3,-12) {$-12$};

\node at (-3,2) {$h$};

\begin{scope}[shift={(.6,0)}]
\node at (1,0) {$\bullet$}; \node[anchor = north west] at (1,0) {$0$};
\node at (-1,0) {$\bullet$}; \node[anchor = north east] at (-1,0) {$\frac{1}{2}q^{\frac{1}{2}}$}; 
\node at (0,-2) {$\bullet$}; \node[anchor = west] at (0,-2) {$q^{\frac{1}{2}} - q^{\frac{3}{2}}$};
\node at (0,-4) {$\bullet$}; \node[anchor = west] at (0,-4) {$q^{\frac{1}{2}} - q^{\frac{3}{2}}$};
\node at (0,-6) {$\bullet$}; \node[anchor = west] at (0,-6) {$q^{\frac{1}{2}} - q^{\frac{3}{2}} - q^{\frac{11}{2}}$};
\node at (0,-8) {$\bullet$}; \node[anchor = west] at (0,-8) {$q^{\frac{1}{2}} - q^{\frac{3}{2}} - q^{\frac{11}{2}}$};
\node at (0,-10) {$\bullet$}; \node[anchor = west] at (0,-10) {$q^{\frac{1}{2}} - q^{\frac{3}{2}} - q^{\frac{11}{2}} +\frac{1}{2}q^{\frac{21}{2}}$};
\node at (0,-12) {$\bullet$}; \node[anchor = west] at (0,-12) {$q^{\frac{1}{2}} - q^{\frac{3}{2}} - q^{\frac{11}{2}} + q^{\frac{21}{2}} - q^{\frac{23}{2}}$};
\draw (1,0)--(0,-2)--(0,-4)--(0,-6)--(0,-8)--(0,-10)--(0,-12)--(0,-13);
\draw (-1,0)--(0,-2);
\end{scope}

\end{tikzpicture}
}
\caption{$-1$-surgery on trefoil. The two weights labeled $w$ are both given by $w = \frac{1}{2}q^{\frac{21}{2}}-\frac{1}{2}q^{\frac{23}{2}}$.}
\label{fig:trefoil_-1_surgery_with_weights}
\end{figure}


\appendix

\section{Remarks on invariants of plumbed manifolds}
\label{sec:remarks on invariants of plumbed mflds}

In this appendix we discuss more precisely the sense in which constructions both in this paper and elsewhere in the literature constitute invariants of plumbed manifolds (or plumbed knot complements).

One way to formulate the notion of an invariant of a closed, oriented $3$-manifold $Y$ 
equipped with $\mathfrak{s} \in \spinc(Y)$ is to assign an object in some category to $(Y, \mathfrak{s})$, such that if $g: Y \to Y'$ is a diffeomorphism with $g_*(\mathfrak{s}) = \mathfrak{s}' \in \spinc(Y')$, then $(Y, \mathfrak{s})$ and $(Y', \mathfrak{s}')$ are assigned isomorphic objects. 
For instance, Heegaard Floer homology is known to be an invariant\footnote{We note that the properties established in \cite{HF_naturality} are much stronger than the type of invariance considered in this appendix.} of the pair $(Y, \mathfrak{s})$ in the above sense \cite{HF_naturality}. 

Restricting to the present setting of negative definite plumbed manifolds, suppose one has an object $I(\Gamma, [k])$ for each negative definite plumbing tree $\Gamma$ and $[k]\in \spinc(\Gamma)$. 
One way to ensure that $I$ constitutes  an invariant in the above sense is to verify the following: 
given negative definite plumbing trees $\Gamma$ and $\Gamma'$, $\spinc$ structures $[k]\in \spinc(\Gamma)$ and $[k]' \in \spinc(\Gamma')$, and a diffeomorphism $g: Y(\Gamma) \to Y(\Gamma')$ with $g_*([k]) = [k]'$, there is a sequence of moves through negative definite plumbings $(\Gamma, [k]) = (\Gamma_1, [k]_1) \to  \cdots \to (\Gamma_n, [k]_n) = (\Gamma', [k]')$ such that $I(\Gamma_i, [k]_i)$ are all isomorphic. 
Theorem \ref{thm:Neumann closed} describes how to relate $\Gamma$ and $\Gamma'$, but, to our knowledge, Neumann moves alone do not address the additional $\spinc$ structure labels. 
Compare with Kirby's calculus of framed links \cite{Kirby}, as formulated in \cite[Theorem 5.3.6]{GompfStipsicz}, 
which says that for any framed links $L$ and $L'$ in $S^3$ and any orientation-preserving diffeomorphism $g$ between the manifolds obtained by surgery on $L$ and $L'$, there is a sequence of Kirby moves (adding or removing a $\pm 1$ framed unknot and handle slides) which realizes $g$ up to isotopy.

A similar discussion applies to the case of marked plumbed knot complements equipped with a relative $\spinc$ structure. 
In light of this, invariance of the main construction in the present paper, Theorem \ref{thm:invariance}, is with respect to Neumann moves and their induced map on $\spinc$ structures. 
Theorem \ref{thm:bigraded root invariance} and Theorem \ref{thm:weighted graded root invariance} are stated analogously.

Automorphisms of plumbing trees also act on the set of $\spinc$ structures. 
Precisely, let $\agraph$ be a negative definite plumbing tree and let $\varphi$ be a graph automorphism of $\agraph$ which preserves the weight at each vertex. 
Pick an ordering $v_1, \ldots, v_s$ of $\V(\agraph)$ for convenience, and view $\varphi$ as a permutation of $\{1, \ldots, s\}$. 
This gives a map $\Z^s\to \Z^s$, still denoted $\varphi$, given by $\varphi(x_1, \ldots, x_s) = (x_{\varphi^{-1}(1)}, \ldots, x_{\varphi^{-1}(s)})$, which in turn induces a map $\varphi_* : \spinc(\agraph) \to \spinc(\agraph)$.

\begin{prop}
    For any $\eps\in \{\pm 1\}$, admissible family of functions $W$, and $\mathfrak{s}\in \spinc(\agraph)$, the weighted graded roots $\root_{\eps}(\agraph, \mathfrak{s}, W)$ and $\root_{\eps}(\agraph, \varphi_*(\mathfrak{s}), W)$ are isomorphic. 
\end{prop}

\begin{proof}
   It is straightforward to see that  $K^2 = (\varphi(K))^2$ for any $K\in \Char(\agraph)$, so that $\varphi$ restricts to a map $\superlevel_h(\agraph, \mathfrak{s}) \to \superlevel_h(\agraph, \varphi_*(\mathfrak{s}))$ for each $h \in h_U(k) + 2\Z^s$. 
   This evidently yields an isomorphism of $1$-dimensional CW complexes, and the contribution of $K \in \superlevel_h(\agraph, \mathfrak{s})$ to the weight of the connected component it lies in is equal to the contribution of $\varphi(K) \in \superlevel_h(\agraph, \varphi_*(\mathfrak{s}))$. 
\end{proof}

It is natural to ask if Neumann moves and graph isomorphisms suffice to generate all maps of $\spinc$ structures induced by orientation-preserving diffeomorphism between negative definite plumbed manifolds. 
Example \ref{ex:conjugation example} below demonstrates that this is not the case. 
We first record some preliminary observations.

Let $Y$ be a closed oriented $3$-manifold. 
In Turaev's convention \cite[Chapter I, Section 4.3]{Turaev}, the first Chern class is viewed as a map $c :\spinc(Y) \to H_1(Y;\Z)$. 
It satisfies $c(\b{\mathfrak{s}}) = - c(\mathfrak{s})$, where $\b{\mathfrak{s}}$ is the conjugate of $\mathfrak{s}$, and $c$ is injective if $H_1(Y;\Z)$ has no $2$-torsion.

\begin{lem}
\label{lem:strongly invertible}
    For a plumbing graph $\Gamma$, its associated framed link $\mathcal{L}(\Gamma) \subset S^3$ is strongly invertible. 
\end{lem}

\begin{proof}
    By induction on the number of vertices, a diagram of $\mathcal{L}(\Gamma)$ can be arranged so that each component is an unknot which intersects a fixed line in two points. 
    A rotation of $\pi$ about this line  demonstrates that $\mathcal{L}(\Gamma)$ is strongly invertible. 
\end{proof}

Now fix a plumbing graph $\Gamma$, and let $f$ be an orientation-preserving diffeomorphism of $S^3$ sending $\mathcal{L}(\Gamma)$ to itself with orientation reversed (for instance, the one constructed in the proof of Lemma \ref{lem:strongly invertible}). 
Denote by $\til{f} : Y(\Gamma) \to Y(\Gamma)$ the induced diffeomorphism. 

\begin{prop}
\label{prop:diffeo acting by conjugation}
   With the above notation, if $H_1(Y(\Gamma);\Z)$ has no $2$-torsion, then the map $\spinc(Y(\Gamma)) \to \spinc(Y(\Gamma))$ induced by $\til{f}$ is given by $\mathfrak{s}\mapsto \b{\mathfrak{s}}$. 
\end{prop}

\begin{proof}
    Naturality of the Chern class gives a commutative diagram
    \[
\begin{tikzcd}
    \spinc(Y(\Gamma)) \ar[r, "c"] \ar[d] & H_1(Y(\Gamma);\Z) \ar[d, "\til{f}_*"] \\
    \spinc(Y(\Gamma)) \ar[r, "c"] & H_1(Y(\Gamma);\Z)
\end{tikzcd}
    \]
    with injective horizontal arrows. 
    The right vertical map is given by $\til{f}_*(x) = -x$ for all $x\in H_1(Y(\Gamma);\Z)$. 
    The statement follows. 
\end{proof}

\begin{exmp}
\label{ex:conjugation example}
Consider the plumbing tree $\Gamma$, shown below.
\begin{center}
\begin{tikzpicture}[scale=.7]
\node (m) at (0,0) {$\bullet$};
\node (r) at (2,0) {$\bullet$};
\node (l) at (-2,0) {$\bullet$};
\node (b) at (0,-2) {$\bullet$};
\node (t) at (0,2) {$\bullet$};

\draw (-2,0) --++ (4,0);
\draw (0,2) --++ (0,-4);

\node[above] at (-.5,0) {$-1$};
\node[above] at (r) {$-7$};
\node[above] at (l) {$-11$};
\node[left] at (b) {$-10$};
\node[left] at (t) {$-3$};

\end{tikzpicture}
\end{center}
In \cite[Example 8.4]{AJK}, a $\spinc$ structure $\mathfrak{s} \in \spinc(Y(\Gamma))$ was found such that the weighted graded roots of $(\Gamma, \mathfrak{s})$ and $(\Gamma, \b{\mathfrak{s}})$ are distinct. 
On the other hand, $H_1(Y(\Gamma);\Z) \cong \Z/ 769\Z$ has no $2$-torsion, so by Proposition \ref{prop:diffeo acting by conjugation} there is a diffeomorphism of $Y(\Gamma)$ acting as conjugation on $\spinc(Y(\Gamma))$. 

Since the weighted graded root is invariant under Neumann moves and graph isomorphisms, this shows that these moves alone are not enough to induce the action of any given diffeomorphism on the set of $\spinc$ structures. 
\end{exmp}

We discuss the behavior of the weighted graded root under $\spinc$ conjugation in more detail in Section \ref{sec:spinc conjugation revisited}. 

\begin{rem}
    Manifolds associated with negative definite plumbing trees are perhaps more naturally viewed from the perspective of singularity theory rather than the perspective of low-dimensional topology. 
    For a negative definite plumbing tree $\agraph$, there is a \emph{canonical} $\spinc$ structure $\mathfrak{s}_{can}$, which is represented by the characteristic vector $(-m_1 -2, \ldots, -m_s-2)$. 
    This $\spinc$ structure is the restriction to $Y(\Gamma)$ of the $\spinc$ structure determined on $X(\Gamma)$ by an almost-complex structure; 
    see for instance the discussion in \cite[Section 2]{NemethiNicolaescu}. 
    The diffeomorphisms associated with Neumann moves preserve the canonical $\spinc$ structure, while a general diffeomorphism need not. 
\end{rem}

\bibliographystyle{amsalpha}
\bibliography{main}

\newcommand{\etalchar}[1]{$^{#1}$}
\providecommand{\bysame}{\leavevmode\hbox to3em{\hrulefill}\thinspace}
\providecommand{\MR}{\relax\ifhmode\unskip\space\fi MR }
\providecommand{\MRhref}[2]{%
  \href{http://www.ams.org/mathscinet-getitem?mr=#1}{#2}
}
\providecommand{\href}[2]{#2}
\begin{thebibliography}{GPPV20}

\bibitem[AJK23]{AJK}
Rostislav Akhmechet, Peter~K. Johnson, and Vyacheslav Krushkal, \emph{Lattice cohomology and {$q$}-series invariants of 3-manifolds}, J. Reine Angew. Math. \textbf{796} (2023), 269--299. \MR{4554472}

\bibitem[BMM20]{BMM}
Kathrin Bringmann, Karl Mahlburg, and Antun Milas, \emph{Quantum modular forms and plumbing graphs of 3-manifolds}, J. Combin. Theory Ser. A \textbf{170} (2020), 105145, 32. \MR{4015713}

\bibitem[CCF{\etalchar{+}}19]{3d_modularity}
Miranda~C.N. Cheng, Sungbong Chun, Francesca Ferrari, Sergei Gukov, and Sarah~M. Harrison, \emph{3d modularity}, J. High Energy Phys. (2019), no.~10, 010, 93. \MR{4059684}

\bibitem[DM19]{Dai-Manolescu}
Irving Dai and Ciprian Manolescu, \emph{Involutive {H}eegaard {F}loer homology and plumbed three-manifolds}, J. Inst. Math. Jussieu \textbf{18} (2019), no.~6, 1115--1155. \MR{4021102}

\bibitem[GM21]{GM}
Sergei Gukov and Ciprian Manolescu, \emph{A two-variable series for knot complements}, Quantum Topol. \textbf{12} (2021), no.~1, 1--109. \MR{4233201}

\bibitem[GPPV20]{GPPV}
Sergei Gukov, Du~Pei, Pavel Putrov, and Cumrun Vafa, \emph{B{PS} spectra and 3-manifold invariants}, J. Knot Theory Ramifications \textbf{29} (2020), no.~2, 2040003, 85. \MR{4089709}

\bibitem[GS99]{GompfStipsicz}
Robert~E. Gompf and Andr\'{a}s~I. Stipsicz, \emph{{$4$}-manifolds and {K}irby calculus}, Graduate Studies in Mathematics, vol.~20, American Mathematical Society, Providence, RI, 1999. \MR{1707327}

\bibitem[Jac21]{jackson2021invariance}
Matthew~PF Jackson, \emph{The invariance of knot lattice homology}, arXiv preprint arXiv:2111.05229 (2021).

\bibitem[JTZ21]{HF_naturality}
Andr\'{a}s Juh\'{a}sz, Dylan Thurston, and Ian Zemke, \emph{Naturality and mapping class groups in {H}eegard {F}loer homology}, Mem. Amer. Math. Soc. \textbf{273} (2021), no.~1338, v+174. \MR{4337438}

\bibitem[Kir78]{Kirby}
Robion Kirby, \emph{A calculus for framed links in {$S\sp{3}$}}, Invent. Math. \textbf{45} (1978), no.~1, 35--56. \MR{467753}

\bibitem[LM23]{Liles-McSpirit}
Louisa Liles and Eleanor McSpirit, \emph{Infinite families of quantum modular 3-manifold invariants}, arXiv preprint arXiv:2306.14765 (2023).

\bibitem[LZ99]{Lawrence-Zagier}
Ruth Lawrence and Don Zagier, \emph{Modular forms and quantum invariants of {$3$}-manifolds}, vol.~3, 1999, Sir Michael Atiyah: a great mathematician of the twentieth century, pp.~93--107. \MR{1701924}

\bibitem[Mur23]{Mur}
Yuya Murakami, \emph{{A proof of a conjecture of Gukov-Pei-Putrov-Vafa}}, arXiv preprint arXiv:2302.13526 (2023).

\bibitem[NC24]{Niemi-Colvin}
Seppo Niemi-Colvin, \emph{Invariance and naturality of knot lattice homology and homotopy}, arXiv preprint arXiv:2202.08941 (2024).

\bibitem[Neu81]{Neumann}
Walter~D. Neumann, \emph{A calculus for plumbing applied to the topology of complex surface singularities and degenerating complex curves}, Trans. Amer. Math. Soc. \textbf{268} (1981), no.~2, 299--344. \MR{632532}

\bibitem[NN02]{NemethiNicolaescu}
Andr\'{a}s N\'{e}methi and Liviu~I. Nicolaescu, \emph{Seiberg-{W}itten invariants and surface singularities}, Geom. Topol. \textbf{6} (2002), 269--328. \MR{1914570}

\bibitem[Né05]{Nem_On_the}
András Némethi, \emph{On the {O}zsv\'{a}th-{S}zab\'{o} invariant of negative definite plumbed 3-manifolds}, Geom. Topol. \textbf{9} (2005), 991--1042. \MR{2140997}

\bibitem[Né08]{Nem_Lattice_cohomology}
\bysame, \emph{Lattice cohomology of normal surface singularities}, Publ. Res. Inst. Math. Sci. \textbf{44} (2008), no.~2, 507--543. \MR{2426357}

\bibitem[OS03]{OS-on_the_Floer}
Peter Ozsv\'{a}th and Zolt\'{a}n Szab\'{o}, \emph{On the {F}loer homology of plumbed three-manifolds}, Geom. Topol. \textbf{7} (2003), 185--224. \MR{1988284}

\bibitem[OS04]{OS-HF}
\bysame, \emph{Holomorphic disks and topological invariants for closed three-manifolds}, Ann. of Math. (2) \textbf{159} (2004), no.~3, 1027--1158. \MR{2113019}

\bibitem[OS08]{ozsvath2008knot}
Peter Ozsv{\'a}th and Zolt{\'a}n Szab{\'o}, \emph{{Knot Floer homology and integer surgeries}}, Algebraic \& Geometric Topology \textbf{8} (2008), no.~1, 101--153.

\bibitem[OS11]{OS-rational}
Peter~S. Ozsv\'{a}th and Zolt\'{a}n Szab\'{o}, \emph{Knot {F}loer homology and rational surgeries}, Algebr. Geom. Topol. \textbf{11} (2011), no.~1, 1--68. \MR{2764036}

\bibitem[OSS14a]{OSS}
Peter Ozsv\'{a}th, Andr\'{a}s~I. Stipsicz, and Zolt\'{a}n Szab\'{o}, \emph{Knots in lattice homology}, Comment. Math. Helv. \textbf{89} (2014), no.~4, 783--818. \MR{3284295}

\bibitem[OSS14b]{OSS_spectral_sequence}
\bysame, \emph{A spectral sequence on lattice homology}, Quantum Topol. \textbf{5} (2014), no.~4, 487--521. \MR{3317341}

\bibitem[OSS16]{OSS_L_spaces}
\bysame, \emph{Knot lattice homology in {$L$}-spaces}, J. Knot Theory Ramifications \textbf{25} (2016), no.~1, 1650003, 24. \MR{3449536}

\bibitem[RT91]{RT}
N.~Reshetikhin and V.~G. Turaev, \emph{Invariants of {$3$}-manifolds via link polynomials and quantum groups}, Invent. Math. \textbf{103} (1991), no.~3, 547--597. \MR{1091619}

\bibitem[Tur02]{Turaev}
Vladimir Turaev, \emph{Torsions of {$3$}-dimensional manifolds}, Progress in Mathematics, vol. 208, Birkh\"{a}user Verlag, Basel, 2002. \MR{1958479}

\bibitem[Wit89]{Witten}
Edward Witten, \emph{Quantum field theory and the {J}ones polynomial}, Comm. Math. Phys. \textbf{121} (1989), no.~3, 351--399. \MR{990772}

\bibitem[Zag10]{Zagier}
Don Zagier, \emph{Quantum modular forms}, Quanta of maths, Clay Math. Proc., vol.~11, Amer. Math. Soc., Providence, RI, 2010, pp.~659--675. \MR{2757599}

\bibitem[Zem21]{Zem}
Ian Zemke, \emph{{The equivalence of lattice and Heegaard Floer homology}}, arXiv preprint arXiv:2111.14962 (2021).

\end{thebibliography}
\end{document}